\documentclass{amsart}
\usepackage{color}
\usepackage{graphicx}
\usepackage{hyperref}
\usepackage{amssymb}
\usepackage{amsfonts}
\usepackage{euscript}
\usepackage[all]{xy}
\usepackage{bbm}

\numberwithin{equation}{section}
\numberwithin{figure}{section}
\renewcommand{\subsection}[1]{\hspace{-\parindent}\refstepcounter{subsection}{\bf (\arabic{section}\alph{subsection}) #1.}\addcontentsline{toc}{subsection}{\bf #1.}}

\newenvironment{nouppercase}{%
  \renewcommand{\uppercasenonmath}[1]{}}{}

\theoremstyle{plain}
\newtheorem{thm}{Theorem}[section]
\newtheorem{theorem}[thm]{Theorem}

\newtheorem{corollary}[thm]{Corollary}

\newtheorem{prop}[thm]{Proposition}

\newtheorem{definition}[thm]{Definition}

\newtheorem{remark}[thm]{Remark}

\newtheorem{proposition}[thm]{Proposition}
\newtheorem{example}[thm]{Example}
\newtheorem{non-example}[thm]{Non-example}

\newtheorem{properties}[thm]{Properties}

\newtheorem{lemma}[thm]{Lemma}

\newtheorem{conjecture}[thm]{Conjecture}

\newtheorem*{claim*}{Claim} 
\newtheorem*{lemma*}{Lemma}
\newtheorem*{theorem*}{Theorem}
\newtheorem*{conjecture*}{Conjecture}

\newtheorem{terminology}[thm]{Terminology}

\newcommand{\bC}{{\mathbb C}}

\newcommand{\bF}{{\mathbb F}}
\newcommand{\bG}{{\mathbb G}}

\newcommand{\bQ}{{\mathbb Q}}
\newcommand{\bR}{{\mathbb R}}

\newcommand{\bZ}{{\mathbb Z}}

\newcommand{\scrA}{\EuScript A}
\newcommand{\scrB}{\EuScript B}
\newcommand{\scrC}{\EuScript C}
\newcommand{\scrD}{\EuScript D}

\newcommand{\scrF}{\EuScript F}

\newcommand{\scrH}{\EuScript H}
\newcommand{\scrI}{\EuScript I}

\newcommand{\scrK}{\EuScript K}
\newcommand{\scrL}{\EuScript L}
\newcommand{\scrM}{\EuScript M}

\newcommand{\scrR}{\EuScript R}

\newcommand{\iso}{\cong}
\newcommand{\htp}{\simeq}
\newcommand{\smooth}{C^\infty}

\renewcommand{\hom}{\mathit{hom}}

\newcommand{\howmany}[1]{\underline{#1}}

\title[Formal groups]{\Large\larger\rm Formal groups and quantum cohomology}
\author{Paul Seidel}

\begin{document}
\begin{nouppercase}
\maketitle
\end{nouppercase}

\begin{abstract}
We use chain level genus zero Gromov-Witten theory to associate to any closed monotone symplectic manifold a formal group (loosely interpreted), whose Lie algebra is the odd degree cohomology of the manifold (with vanishing bracket). When taken with coefficients in $\bF_p$ for some prime $p$, the $p$-th power map of the formal group is related to quantum Steenrod operations. The motivation for this construction comes from derived Picard groups of Fukaya categories, and from arithmetic aspects of mirror symmetry.
\end{abstract}

\section{Introduction}

This paper is concerned with aspects of genus zero Gromov-Witten theory, which are specifically of interest if one works with integer or mod $p$ cohomological coefficients. There is a shared context between this and arithmetic aspects of Fukaya categories (e.g.\ \cite{evans-lekili15, lekili-polishchuk16, lekili-perutz12, allston-amorim12}), even though we do not work on a categorical level. Instead, our constructions will resemble those of certain chain level structures, and of cohomology operations, in algebraic topology.

\subsection{Background}
Gromov-Witten theory on a closed symplectic manifold $X$ can be axiomatized as a Cohomological Field Theory \cite{kontsevich-manin94}, which means that operations are parametrized by Deligne-Mumford moduli spaces of curves. We will only consider genus zero curves, where the notion of Cohomological Field Theory is related to ones from classical topology: namely, one can start with the little disc operad \cite[Chapter IV]{may72}, then enlarge it to the framed little disc operad \cite{getzler94b}, and finally trivialize the circle action \cite{drummond-cole14} to obtain the genus zero Deligne-Mumford operad. It is important for this paper to work on the chain level. An example of a chain level construction is the quantum $A_\infty$-ring structure \cite{ruan-tian95}, which refines the small quantum product. In abstract terms, this comes from mapping Stasheff associahedra to Deligne-Mumford spaces, compatibly with the operad structures.

To define the genus zero Cohomological Field Theory for a general $X$, one usually has to work with coefficient rings containing $\bQ$, because of the multivalued perturbations involved in making moduli spaces regular. However, in the special case where $X$ is weakly monotone (also called semi-positive; see e.g.\ \cite[Section 6.4]{mcduff-salamon}), the relevant Gromov-Witten invariants, which count genus zero curves with $\geq 3$ marked points in a given homology class, can be defined over $\bZ$. If one reduces coefficients to a finite field $\bF_p$, there are two obvious constructions of cohomology operations. One can use the relation with the little disc operad to obtain analogues of the Cohen operations on the homology of double loop spaces \cite{cohen76}. For ease of reference, let's call these quantum Cohen operations. The second approach is to introduce quantum Steenrod operations, which were proposed in \cite{fukaya93b} and have attracted some recent attention \cite{wilkins18}. These are both facets of a common story, which involves the equivariant cohomology of Deligne-Mumford space with $(p+1)$ marked points, with respect to the action of the symmetric group $\mathit{Sym}_p$ permuting all but one point.

We take our bearings from both the ``one-dimensional'' quantum $A_\infty$-structure and the ``two-dimensional'' quantum Cohen and Steenrod operations. To thread our way between the two, we use another family of moduli spaces, which come from the convolution theory of Lagrangian correspondences \cite{mau-wehrheim-woodward18, bottman-wehrheim18, bottman17, fukaya17}. They map to Deligne-Mumford spaces, and on the other hand, their boundary structure is governed by Stasheff associahedra. The effect of using (the simplest of) these spaces is to equip the set of quantum Maurer-Cartan solutions with a multiplicative structure. After reduction mod $p$, that structure will admit a partial description in terms of a specific quantum (Cohen or) Steenrod operation.

\subsection{Algebraic terminology\label{subsec:mc}}
Before continuing the discussion, we need to recall some definitions. In a ``functor of points'' approach, an object is often described as a functor from a class of ``coefficient rings'' to sets. We use the following coefficient rings, familiar from the theory of formal schemes and from deformation theory.

\begin{definition} \label{th:adic-ring}
An adic ring is a non-unital commutative ring $N$ such that the map $N \rightarrow \underleftarrow{\lim}_m N/N^m$ is an isomorphism. In other words, $\bigcap_m N^m = 0$, and $N$ is complete with respect to the topology given by the decreasing filtration $\{N^m\}$. Note that one can adjoin a unit, forming the augmented ring $\bZ 1 \oplus N$, which contains $N$ as an ideal.
\end{definition}

\begin{example} \label{th:example-of-ring}
Standard examples are $N = q\bZ[[q]]$ (power series with zero constant term) or its truncations $N = q\bZ[q]/q^{m+1}$. We can also use field coefficients, for instance taking $N = q\bF_p[[q]]$, which simplifies the algebraic behaviour slightly. An example with ``unequal characteristic'' is $N = p\bZ_p$, the maximal ideal in the ring of $p$-adic integers, where $N/N^m = \bZ/p^{m-1}$.
\end{example}


\begin{definition}
A ``formal group'' is a functor from adic rings to groups. 
\end{definition}

This is somewhat weaker than the classical notion of formal group \cite{lazard75}: there, one imposes additional conditions on the functor, leading to representability results in an appropriate category of formal schemes. In our application, we will be truncating what should really be an object of derived geometry, and representability in the classical sense is not expected to hold. For simplicity, we have chosen to ignore the issue, resulting in the definition given above.

As mentioned before, adic rings are a standard way to formulate deformation problems \cite{schlessinger68}. The specific problem relevant for us is the following. Let $\scrA$ be an $A_\infty$-ring (see Section \ref{subsec:hh} for our conventions). Given an adic ring $N$, let $\scrA \hat\otimes N$ be the inverse limit of tensor products $\scrA \otimes (N/N^m)$. We consider solutions $\gamma \in \scrA^1 \hat\otimes N$ of the (generalized) Maurer-Cartan equation
\begin{equation} \label{eq:mc}
\sum_{d \geq 1} \mu^d_{\scrA}(\gamma,\dots,\gamma) = 0.
\end{equation}
Two such solutions $\gamma, \tilde\gamma$ are considered equivalent if there is an $h \in \scrA^0 \hat\otimes N$ such that
\begin{equation} \label{eq:mc-equivalence}
\sum_{p,q} \mu_{\scrA}^{p+q+1}(\overbrace{\gamma,\dots,\gamma}^{p},h,\overbrace{\tilde\gamma,\dots,\tilde\gamma}^q) = \gamma - \tilde\gamma.
\end{equation}

\begin{definition} $\mathit{MC}(\scrA;N)$ is the set of equivalence classes of Maurer-Cartan elements in $\scrA \hat\otimes N$. This is functorial in $N$, giving a functor $\mathit{MC}(\scrA)$ from adic rings to sets.
\end{definition}

If $N^2 = 0$, \eqref{eq:mc} reduces to $\mu^1_{\scrA}(\gamma) = 0$, and \eqref{eq:mc-equivalence} to $\mu^1_{\scrA}(h) = \gamma - \tilde{\gamma}$. Hence, in this case $\mathit{MC}(\scrA;N) = H^1(\scrA; N)$ is the cohomology with coefficients in the abelian group $N$. Correspondingly, the general $\mathit{MC}(\scrA;N)$ can be viewed as nonlinear analogues of cohomology groups. Note that what we are studying is not the deformation theory of $\scrA$ as an $A_\infty$-ring: instead, it can be viewed as the deformation theory of the free module $\scrA$, inside the dg category of $A_\infty$-modules.

\subsection{The formal group structure}
With this in mind, let's return to symplectic geometry. To keep the formalism in the simple form set up above (avoiding Novikov rings), we will assume that our symplectic manifold $X$ is monotone (rather than weakly monotone), which means that its symplectic form satisfies
\begin{equation} \label{eq:monotone}
[\omega_X] = \delta c_1(X) \in H^2(X;\bR) \;\; \text{for some $\delta > 0$.}
\end{equation}
Take a suitable chain complex $\scrC = C^*(X)$ representing its integral cohomology, equipped with the quantum $A_\infty$-structure $\mu_{\scrC}$. Note that the quantum $A_\infty$-structure is only $\bZ/2$-graded; hence, the definition above should be interpreted so that Maurer-Cartan elements are taken in $\scrC^{odd} \hat\otimes N$, and correspondingly, the entire odd degree cohomology of $X$ appears. Let $\mathit{MC}(X;N) = \mathit{MC}(\scrC;N)$ be the set of equivalence classes of Maurer-Cartan solutions. One can think of this as the deformation theory of the diagonal $\Delta_X$ as an object of the Fukaya category $\scrF(X \times \bar{X})$, where $\bar{X}$ means that we have reversed the sign of the symplectic form. In other words, deformations are ``bounding cochains'' for $\Delta_X$ in the sense of \cite{fooo}. If the closed-open map is an isomorphism, one can also think of this theory as deformations of the identity functor on $\scrF(X)$, which describes the formal neighbourhood of the identity in the ``automorphism group'' of that category. From the composition of automorphisms, one would expect additional structure, and indeed:

\begin{prop} \label{th:main}
The functor $\mathit{MC}(X)$ has the canonical structure of a ``formal group''. 
\end{prop}

As mentioned above, if one makes suitable assumptions on the closed-open map, this structure has an explanation purely within homological algebra. If one drops that assumption, one could still obtain the group structure by looking at $\scrF(X \times \bar{X})$ together with its monoidal structure (in a suitable sense, which we will not try to make precise) given by convolution of correspondences (see \cite{mau-wehrheim-woodward18, bottman-wehrheim18}; another approach is \cite{lekili-lipyanskiy10, fukaya17}). Compared to those constructions, the definition given here (which avoids talking about Fukaya categories or Lagrangian correspondences) is less general but more direct, and hence more amenable to computations.

\begin{proposition} \label{th:commutative}
The groups $\mathit{MC}(X;N)$ are commutative if $N^3 = 0$. They are also commutative if $N^4 = 0$, provided that additionally, $H^*(X;\bZ)$ is torsion-free.
\end{proposition}

Commutativity mod $N^3$ is not surprising: it amounts to the well-known fact that the Lie bracket on cohomology, which exists for any algebra over the little disc operad, becomes zero for Cohomological Field Theories. For general algebraic reasons (formal exponentiation), one expects commutativity to hold always if $N$ is an algebra over $\bQ$; and the same should be true if $N$ is an algebra over $\bF_p$ and $N^{p+1} = 0$. In contrast, the origin of the second part of Proposition \ref{th:commutative} is more geometric: it reflects an explicit (if poorly understood, partly due to a lack of examples) enumerative obstruction to commutativity.
%
%

\begin{remark}
Our construction focuses on the odd degree cohomology of $X$. One could try to include even degree classes by enlarging the notion of formal group to its derived counterpart, which in our terms means allowing $N$ to be a commutative dg (or maybe better simplicial) ring. Another potential use of even degree classes (with different enumerative content) would be as ``bulk insertions'' at points in arbitrary position, as in big quantum cohomology. Note however that, for classes of degree $>2$, the standard algebraic formalism of ``bulk insertions'' involves dividing by factorials. Hence, it would have to be modified for our applications. Neither direction will be attempted in this paper.
\end{remark}

\subsection{Quantum Steenrod operations}
Fix a prime $p$. The quantum Steenrod operation, in a form slightly simplified by the monotonicity assumption \eqref{eq:monotone}, is a map
\begin{equation} \label{eq:quantum-steenrod}
\begin{aligned}
& Q\mathit{St}_{X,p} =  \textstyle\sum_A Q\mathit{St}_{X,p,A}: H^*(X;\bF_p)
\longrightarrow H^*(X;\bF_p) \otimes H^*_{\bZ/p}(\bF_p), \\
& Q\mathit{St}_{X,p,A}: H^l(X;\bF_p) \longrightarrow \big(H^*(X;\bF_p) \otimes H^*_{\bZ/p}(\bF_p) \big)^{pl-2c_1(A)}.
\end{aligned}
\end{equation}
Here, $H^*_{\bZ/p}(\bF_p)$ is the group cohomology of the cyclic group with coefficients mod $p$, which is one-dimensional in each degree. We fix generators
\begin{equation} \label{eq:t-theta}
H^*_{\bZ/p}(\bF_p) = \bF_p[t,\theta], \quad |t| = 2, \, |\theta| = 1.
\end{equation}
The notation here requires some explanation. For $p = 2$, we have $\theta^2 = t$ (or $\theta = t^{1/2}$), so the two generators are not independent. For $p>2$, it is implicit that our description is as a graded commutative algebra, so $\theta^2 = 0$. The sum in \eqref{eq:quantum-steenrod} is over $A \in H_2(X;\bZ)$, and the notation $c_1(A)$ is shorthand for integrating the first Chern class of $X$ over $A$. The classical Steenrod operations \cite{steenrod-epstein} are encoded in the $A = 0$ term. More precisely, if we write $\mathit{St}_{X,p} = \mathit{QSt}_{X,p,0}$, the relation with the classical notation is that
\begin{equation} \label{eq:classical-steenrod}
\mathit{St}_{X,p}(x) = \begin{cases}
\sum_i  \mathit{Sq}^i(x) t^{(|x|-i)/2} & p = 2, \\
(-1)^\ast {\textstyle \big( \frac{p-1}{2} ! \big)}^{|x|} 
\sum_i (-1)^i P^i(x)\, t^{(|x|-2i)(p-1)/2} + \theta \text{\it (terms invoving $\beta P^i$)} &
p>2,
\end{cases}
\end{equation}
where $\beta$ is the Bockstein, and 
\begin{equation} \label{eq:the-sign}
\textstyle \ast = \frac{|x|(|x|-1)}{2} \frac{p-1}{2}.
\end{equation}
When handling the constants in \eqref{eq:classical-steenrod} in practice, one should bear in mind that \cite[Lemma 6.3]{steenrod-epstein}
\begin{equation} \label{eq:number-theory}
{\textstyle \big( \frac{p-1}{2} ! \big)}^2 \equiv (-1)^{\frac{p+1}{2}}\; \mathrm{mod}\, p.
\end{equation}
For instance, if $|x|$ is even and $p > 2$,
\begin{equation}
\begin{aligned}
\text{\it $t^0$ term of } \mathit{St}_{X,p}(x) & = (-1)^\ast {\textstyle \big( \frac{p-1}{2} ! \big)}^{|x|}
(-1)^{\frac{|x|}{2}} P^{\frac{|x|}{2}}(x) \\ 
& = (-1)^{\frac{|x|}{2} \frac{p-1}{2}} (-1)^{\frac{|x|}{2} \frac{p+1}{2}}
(-1)^{\frac{|x|}{2}} P^{\frac{|x|}{2}}(x) = P^{\frac{|x|}{2}}(x) = x^p.
\end{aligned}
\end{equation}
%

\begin{definition}
Define an endomorphism $Q\Xi_{X,p}$ of $H^{\mathrm{odd}}(X;\bF_p)$ by
\begin{equation}
Q\Xi_{X,p}(x) = \begin{cases} \text{the $t^{\frac12}$ (or $\theta$) component of $Q\mathit{St}_{X,2}(x)$} & p = 2, \\
\textstyle{\big(\frac{p-1}{2} !\big)}^{-1} \,\, \text{times the $t^{\frac{p-1}{2}}$-component of $Q\mathit{St}_{X,p}(x)$}
& p > 2.
\end{cases}
\end{equation}
\end{definition}

To recapitulate, this has the form 
\begin{equation} \label{eq:quantum-xi}
\begin{aligned}
& Q\Xi_{X,p} = \textstyle \sum_A Q\Xi_{X,p,A}, \\
& Q\Xi_{X,p,A}: H^l(X;\bF_p) \longrightarrow H^{pl-(p-1)-2c_1(A)}(X;\bF_p),
\end{aligned}
\end{equation}
and where the classical component is
\begin{equation} \label{eq:classical-xi}
\Xi_{X,p}(x) = Q\Xi_{X,p,0}(x) = 
\begin{cases} \mathit{Sq}^{|x|-1}(x) & p = 2, \\
P^{\frac{|x|-1}{2}}(x) & p> 2.
\end{cases}
\end{equation}

\subsection{The $p$-th power maps\label{subsec:p-power}}
Let's return to the formal group $\mathit{MC}(X)$. The group structure gives rise to $m$-th power (meaning the $m$-fold product) maps for each $m \geq 1$, which are functorial endomorphisms of $\mathit{MC}(X;N)$ for any $N$.


\begin{theorem} \label{th:p-power}
The power maps of prime order fit into a diagram
\begin{equation} \label{eq:p-power-diagram}
\xymatrix{
\ar[d]_-{\text{projection}}
\mathit{MC}(X;q\bF_p[q]/q^{p+1}) \ar[rrrr]^-{\text{$p$-th power of the formal group}} &&&& \mathit{MC}(X;q\bF_p[q]/q^{p+1}) 
\\
\mathit{MC}(X;q\bF_p[q]/q^2)  &&&& \mathit{MC}(X;q^p\bF_p[q]/q^{p+1}) \ar[u]_-{\text{inclusion}}
\\
\ar@{=}[u]
H^{\mathrm{odd}}(M;\bF_p) \ar[rrrr]^{Q\Xi_{X,p}} &&&& H^{\mathrm{odd}}(M;\bF_p). \ar@{=}[u]
}
\end{equation}
\end{theorem}

\begin{remark} \label{th:new-remark}
Because of the monotonicity of $X$ and the grading of our operations, see \eqref{eq:quantum-xi}, one always has
\begin{equation} \label{eq:deg1}
Q\Xi_{X,p}(x) = \Xi_{X,p}(x) = x \quad \text{for $x \in H^1(X;\bF_p)$.} 
\end{equation}
For comparison, consider the formal completion $\hat\bG_m$ of the multiplicative group. In a local coordinate $1+z \in \hat\bG_m$, the $p$-th power map is
\begin{equation}
\overbrace{z \bullet \cdots \bullet z}^p = (1+z)^p - 1 = z^p + p(\text{\it something}) = z \quad \text{for $z \in \bF_p$,}
\end{equation}
which matches what we have seen in \eqref{eq:deg1}. There is a categorical explanation for the occurrence of the multiplicative group. Recall that $H^1(X;\bG_m)$ classifies flat line bundles over $X$. The definition of the Fukaya category $\scrF(X)$ includes having the Lagrangian submanifolds equipped with flat bundles. By tensoring with the restriction of flat line bundles on $X$, one gets an action of $H^1(X;\bG_m)$ on the Fukaya category. For us, it is better to think of the action as being given by the trivial Lagrangian correspondence, namely the diagonal in $X \times \bar{X}$, equipped with a flat line bundle. From that viewpoint, one can pass to the formal completion: one has a formal family of objects in $\scrF(X \times \bar{X})$, which consists of the diagonal together with a formal deformation of the trivial line bundle; that gives rise to a deformation of the identity functor on $\scrF(X)$; and composition of such deformations corresponds to the tensor product of line bundles. Of course, within the present framework this discussion is of very limited concrete use: the known examples of monotone symplectic manifolds with nontrivial $H^1$ (obtained by combining \cite{reznikov} and \cite{millson}, see \cite{fine-panov} for a discussion) are somewhat esoteric.
\end{remark}

\begin{example} \label{th:blowup-example}
Let $X \subset \bC P^1 \times \bC P^3$ be a hypersurface of bidegree $(1,2)$, which has odd cohomology $H^3(X;\bF_p) = \bF_p^2$ for any $p$. Then $Q\Xi_{X,2} = \mathit{id}$, by a computation from \cite[Section 8]{wilkins18}. More generally, each $Q\Xi_{X,p}$ is a multiple of the identity. Here are the results for the first few primes:
\begin{equation} \label{eq:numbers}
\begin{array}{c|ccccccccccccc}
p & 2 & 3 & 5 & 7 & 11 & 13 & 17 & 19 & 23 & 29 & 31 & 37 & 41 \\
\hline
\text{$Q\Xi_p/\mathit{id}$} & -1 & -1 & 1 & 0 & -4 & -2 & 2  & 4 & 0 & -2 & 0 & -10 & 10
\end{array}
\end{equation}
The entries lie in $\bF_p$, and we have chosen integer representatives with the least absolute value (with some fudging for $p = 2$). Those integers are meaningful: they are the $q^p$ coefficients of the modular form \cite[Newform 15.2.a.a]{modular-forms-database}
\begin{equation} \label{eq:modular-form}
\eta(q)\eta(q^3)\eta(q^5)\eta(q^{15}),\;\; \text{ where } 
\eta(q) = q^{1/24} \prod_{n=1}^{\infty} (1-q^n).
\end{equation}
One can interpret this observation via mirror symmetry and arithmetic geometry. The (conjectural, but supported by superpotential computations) statement is that a specific elliptic curve appears in the mirror geometry, and hence is encoded in the Fukaya category of $X$. Correspondingly, the automorphism group of the Fukaya category would contain the derived automorphism group of that curve, and in particular, the product of two copies of the curve itself. What we see in \eqref{eq:numbers} is the leading coefficient of the $p$-th power map of the formal group law of the elliptic curve. For general number theory reasons, this is closely related to counting $\bF_p$-points on the curve, and the appearance of \eqref{eq:modular-form} is an instance of the modularity of elliptic curves. For further discussion, see Example \ref{th:blowup-example-2} and Conjecture \ref{th:isogeny}.
\end{example}

The computation underlying Example \ref{th:blowup-example} turns out to involve only those quantum Steenrod operations which can ultimately (using forthcoming work of Wilkins and the author) be reduced to ordinary Gromov-Witten invariants. To push the understanding of $Q\Xi_{X,p}$ further, one would have to study the contribution of $p$-fold covered curves, which is beyond our scope here.

\begin{example}
Let $X \subset \bC P^1 \times \bC P^5$ be a hypersurface of bidegree $(1,2)$. In this case, $Q\Xi_{X,p}$ is unknown. The answer involves stable maps to $X$ with first Chern number $2p-2$. The difficulty is that there are points in the relevant space of stable maps which have $\bZ/p$ isotropy groups.
\end{example}

\subsection{Structure of the paper} 
In order to make the underlying ideas appear clearly, the paper is set up as follows. Most of the time (Sections \ref{sec:mc}--\ref{sec:power}) we work in an abstract operadic framework. In principle, one could aim to prove that quantum cohomology is an instance of this general setup, but that would overshoot the desired target somewhat. Instead, we will explain (in Section \ref{sec:floer}) how to convert the previous arguments into symplectic terms, in a more ad hoc way. In Section \ref{sec:fukaya}, we outline an alternative approach to parts of the construction, based on \cite{fukaya17}. After that, Section \ref{sec:quantum} is a bit of an outlier: it is concerned with computational techniques for quantum Steenrod operations, and is formulated in a language much closer to standard Gromov-Witten theory.  At this point, we should make one apology for the paper. Because of the complexity of the formulae involved, signs are sometimes not worked out, which we signal by $\pm$; however, we have made sure that signs are given at key points. Part of this involves spelling out certain conventions for equivariant cohomology, which is done in Section \ref{sec:signs}.

{\em Acknowledgments.} I would like to thank Nate Bottman and Kenji Fukaya for providing important insights into moduli spaces and Fukaya categories; Nicholas Wilkins for many conversations about quantum Steenrod operations; John Pardon, Bjorn Poonen, and Andrew Sutherland for teaching me bits of arithmetic geometry; Alessio Corti, Vasily Golyshev, and Victor Przyjalkowski for useful information about mirror symmetry and periods; and Mikhail Kapranov for pointing out related homological algebra results.

{\em Funding.} This research was partially supported by the Simons Foundation, through a Simons Investigator award as well as the Simons Collaboration in Homological Mirror Symmetry; and by the National Science Foundation, through award DMS-1904997. I would also like to thank Columbia University, Princeton University, and the Institute for Advanced Study, for generous visiting appointments during which I worked on this paper.

\section{Maurer-Cartan theory\label{sec:mc}}

After some introductory remarks about solutions of the Maurer-Cartan equations in general $A_\infty$-rings, we turn to a specific situation, namely the induced $A_\infty$-structure on Hochschild cochains. Maurer-Cartan solutions in Hochschild cochains carry a formal group structure, which can be considered as a purely algebraic counterpart of our main construction. This algebraic viewpoint will not really be used later on: we include it here for expository purposes, and also because it would provide the background for linking the results in this paper to the Fukaya category. To make things more intuitive from a classical homological algebra viewpoint, we will take the $A_\infty$-structures to be $\bZ$-graded in this section, even though as mentioned before, the quantum $A_\infty$-structure is only $\bZ/2$-graded.

\subsection{$A_\infty$-structures}
To clarify our conventions, let's spell out the definition of an $A_\infty$-ring. This is a free graded abelian group $\scrA$ is with multilinear operations $\{\mu^d_{\scrA}\}$, $d \geq 1$, which satisfy the $A_\infty$-associativity relations
\begin{equation} \label{eq:associativity}
0 = \sum_{ij} (-1)^{\maltese_i} \mu^{d-j+1}_{\scrA}(a_1,\dots,\mu^j_{\scrA}(a_{i+1},\dots,a_{i+j}),\dots,a_d).
\end{equation}
Here, $\maltese_i = \|a_1\| + \cdots + \|a_i\|$, where $\|a\| = |a| - 1$ is the reduced degree; both will be standing notation from now on.  If we consider $\scrA$ as a chain complex with differential $d_{\scrA} = -\mu^1_{\scrA}$, the associative algebra structure on $H^*(\scrA)$ is induced by the chain level product
\begin{equation} \label{eq:coh-product}
a_1 \cdot a_2 = (-1)^{|a_1|} \mu^2_{\scrA}(a_1,a_2).
\end{equation}

From the overall ``$A_\infty$-lingo'', the notions of $A_\infty$-homomorphism and homotopy between such homomorphisms will be the ones that occur most frequently in our discussion. Homotopy admits the following useful interpretation. Take the following dg ring (cochains on the interval as a simplicial complex, with the Alexander-Whitney product):
\begin{equation} \label{eq:nc-interval}
\begin{aligned}
& \scrI = \bZ u \oplus \bZ \tilde{u} \oplus \bZ v,  \quad |u| = |\tilde{u}| = 0, \;\; |v| = 1, \\
& u^2 = u, \;\; \tilde{u}^2 = \tilde{u}, \;\; \tilde{u}u  = u \tilde{u} = 0, \;\; u v  = v = v \tilde{u}\; \text{ (and hence $vu = \tilde{u}v= 0$),} \\
& d_{\scrI}u = v, \;\; d_{\scrI}\tilde{u} = -v.
\end{aligned}
\end{equation}
If $\scrA$ is an $A_\infty$-ring, the tensor product $\scrA \otimes \scrI$ inherits the same structure, with
\begin{equation} \label{eq:tensor-product}
\left\{
\begin{aligned}
& \mu^1_{\scrA \otimes \scrI}(a \otimes x) = \mu^1_{\scrA}(a) \otimes x + (-1)^{|a|} a \otimes d_{\scrI}x, \\
& \mu^d_{\scrA \otimes \scrI}(a_1 \otimes x_1,\dots, a_d \otimes x_d) = (-1)^* \mu^d_{\scrA}(a_1,\dots,a_d) \otimes
x_1 \cdots x_d, \;\; d \geq 2, \\
\end{aligned}
\right.
\end{equation}
where $* = \sum_{i>j} \|a_i\|\cdot |x_j|$. This $A_\infty$-structure is compatible with the projections
\begin{equation} \label{eq:two-projections}
\xymatrix{
\scrA \otimes \scrI \ar@/^.5pc/[rrr]^-{\text{project to $\scrA \otimes \bZ u$}} 
\ar@/_.5pc/[rrr]_-{\text{project to $\scrA \otimes \bZ \tilde{u}$}} &&& \scrA.
}
\end{equation}
Two $A_\infty$-homomorphisms $\tilde\scrA \rightarrow \scrA$ are homotopic iff they can be obtained from a common homomorphism $\tilde\scrA \rightarrow \scrA \otimes \scrI$ by composing with \eqref{eq:two-projections}. We will often use the following fact:

\begin{lemma} \label{th:homotopy-inverse}
Let $\scrF: \tilde{\scrA} \rightarrow \scrA$ be an $A_\infty$-homomorphism such that the linear term $\scrF^1$ is a chain homotopy equivalence (in view of our freeness assumption, that will be the case whenever it's a quasi-isomorphism). Then $\scrF$ has an inverse up to homotopy.
\end{lemma}

Unitality conditions, while not always strictly necessary, are both convenient for the theory and satisfied in most applications (including ours). A homology unit for $\scrA$ is a cocycle $e_{\scrA} \in \scrA^0$ such that the products
\begin{equation} \label{eq:homotopy-unit}
\begin{aligned}
& a \longmapsto a \cdot e_{\scrA} = (-1)^{|a|} \mu^2_{\scrA}(a,e_{\scrA}), \\
& a \longmapsto e_{\scrA} \cdot a = \mu^2_{\scrA}(e_{\scrA},a)
\end{aligned}
\end{equation}
are homotopic to the identity (when working over a field, one asks that these products induce the identity on cohomology, but that is obviously inadequate over $\bZ$; the notion used here goes back to \cite[Definition 7.3]{lyubashenko08}). One says that $e_{\scrA}$ is a strict unit if: the inclusion $\bZ e_{\scrA} \rightarrow \scrA^0$ splits, as a map of abelian groups; the maps \eqref{eq:homotopy-unit} are equal to the identity; and in addition, all operations $\mu^d_{\scrA}(\cdots, e_{\scrA}, \cdots)$, $d \geq 3$, are zero. The following is \cite[Theorem 3.7 and Remark 3.8]{lm}:

\begin{lemma} \label{th:unit-unit}
Given any homologically unital $A_\infty$-ring $\scrA$, there is a strictly unital one $\tilde{\scrA}$ and an inclusion $\scrA \hookrightarrow \tilde{\scrA}$, compatible with the $A_\infty$-structures, which is a chain homotopy equivalence. (Note that by Lemma \ref{th:homotopy-inverse}, we then also have an inverse $A_\infty$-functor $\scrF: \tilde{\scrA} \rightarrow \scrA$, such that $\scrF^1$ is a chain homotopy equivalence.)
\end{lemma}

The result in \cite{lm} is more explicit: one can enlarge the $A_\infty$-structure to $\tilde{\scrA} = \scrA \oplus \bZ h \oplus \bZ e_{\tilde{\scrA}}$, where $e_{\tilde{\scrA}}$ is the strict unit, and 
\begin{equation} \label{eq:p}
\begin{aligned}
& \mu^1_{\tilde{\scrA}}(h) \in e_{\tilde{\scrA}} + \scrA^0, \\
& \mu^d_{\tilde{\scrA}}(\scrA \oplus \bZ h,\dots, \scrA \oplus \bZ h) \subset \scrA
\;\; \text{ for $d \geq 2$.}
\end{aligned}
\end{equation}
This has a consequence which we find useful to state, even though it goes slightly beyond the limits of our current terminology. Introduce an $A_\infty$-category with two objects $Y$ and $\tilde{Y}$, morphism spaces
\begin{equation}
\begin{aligned}
& \mathit{hom}(Y,Y) = \hom(Y,\tilde{Y}) = \hom(\tilde{Y},Y) = \scrA, \\ 
& \mathit{hom}(\tilde{Y},\tilde{Y}) = \tilde{\scrA},
\end{aligned}
\end{equation}
and with all $A_\infty$-structures inherited from $\tilde{\scrA}$ (the second part of \eqref{eq:p} ensures that this makes sense). The two objects are quasi-isomorphic, and so we arrive at the following:

\begin{lemma} \label{th:new-unit}
Given any homologically unital $A_\infty$-ring $\scrA$, there is a homologically unital $A_\infty$-category with two objects, such that: the endomorphism ring of the first object is $\scrA$; the endomorphism ring of the second object is strictly unital; and the two objects are mutually quasi-isomorphic.
\end{lemma}

\subsection{Maurer-Cartan elements}
We have already mentioned the notions of Maurer-Cartan element \eqref{eq:mc} and of equivalence between such elements \eqref{eq:mc-equivalence}.   Given an $A_\infty$-homomorphism $\scrF: \tilde\scrA \rightarrow \scrA$, we define the induced map $\mathit{MC}(\scrF;N): \mathit{MC}(\tilde\scrA;N) \rightarrow \mathit{MC}(\scrA;N)$ by
\begin{equation} \label{eq:functorial-mc}
\tilde\gamma \longmapsto \gamma = \sum_d \scrF^d(\tilde\gamma,\dots,\tilde\gamma).
\end{equation}
The basic results (the second is a consequence of the first and Lemma \ref{th:homotopy-inverse}) are:

\begin{lemma} \label{th:homotopy-mc}
Homotopic $A_\infty$-homomorphisms induce the same map $\mathit{MC}(\tilde\scrA;N) \rightarrow \mathit{MC}(\scrA;N)$.
\end{lemma}

\begin{lemma} \label{th:mc-bijection}
Suppose that we have an $A_\infty$-homomorphism $\tilde\scrA \rightarrow \scrA$, whose linear part is a chain homotopy equivalence. Then the induced map $\mathit{MC}(\tilde\scrA;N) \rightarrow \mathit{MC}(\scrA;N)$ is bijective.
\end{lemma}

One can think of equivalence of Maurer-Cartan elements in several ways. In terms of \eqref{eq:tensor-product},
\begin{equation} \label{eq:decompose-gamma}
\gamma \otimes u + \tilde{\gamma} \otimes \tilde{u} + h \otimes v \in (\scrA \otimes \scrI\, \hat\otimes N)^1
\end{equation}
is a Maurer-Cartan element for $\scrA \otimes \scrI$ if and only $\gamma$ and $\tilde{\gamma}$ are Maurer-Cartan elements for $\scrA$, and $h$ satisfies \eqref{eq:mc-equivalence}. This makes Lemma \ref{th:homotopy-mc} particularly intuitive. Another possible interpretation goes as follows. Let's add a strict unit, forming $\bZ e \oplus \scrA \hat\otimes N$. There is an $A_\infty$-category whose objects are Maurer-Cartan elements in $\scrA \hat\otimes N$, with morphisms between any two elements given by $\bZ e \oplus \scrA \hat\otimes N$. The differential for morphisms $\tilde{\gamma} \rightarrow \gamma$ is 
\begin{equation} \label{eq:mc-differential}
g \longmapsto \sum_{p,q} \mu^{p+q+1}_{\bZ e \oplus \scrA \hat\otimes N}(\overbrace{\gamma,\dots,\gamma}^p,g,\overbrace{\tilde{\gamma},\dots, \tilde{\gamma}}^{q}),
\end{equation}
and the formulae for higher $A_\infty$-compositions are similar. Clearly, $h$ satisfes \eqref{eq:mc-equivalence} iff $g = e+h$
is a closed morphism $\tilde{\gamma} \rightarrow \gamma$ in our category. This viewpoint can be useful when thinking about the transitivity and functoriality of the notion of equivalence. Finally, if $\scrA$ is homologically unital, one can introduce a modified version of the Maurer-Cartan category, by setting the morphisms between objects to be $\scrA \otimes (\bZ 1 \oplus N)$, which means using the natural identity of $\scrA$ rather than artificially adjoining one. The resulting version of our previous observation (obvious in the strictly unital case, and generalized from there using Lemmas \ref{th:unit-unit} and \ref{th:mc-bijection}) is this:

\begin{lemma} \label{th:alternative-mc}
Suppose that $\scrA$ is homologically unital. Then, two Maurer-Cartan solutions are equivalent if and only if there is a $g \in \scrA^0 \hat\otimes (\bZ 1 \oplus N)$, which modulo $N$ reduces to a cocycle homologous to $e_{\scrA}$, and which satisfies
\begin{equation} \label{eq:alternative-mc}
\sum_{p,q} \mu_{\scrA}^{p+q+1}(\overbrace{\gamma, \dots, \gamma}^p,g,\overbrace{\tilde\gamma, \dots,\tilde \gamma}^q) = 0.
\end{equation}
\end{lemma}

\subsection{Hochschild cochains\label{subsec:hh}}
As before, let $\scrA$ be an $A_\infty$-ring. Our attention will now shift to its Hochschild complex (the complex underlying Hochschild cohomology)
\begin{equation} \label{eq:hochschild-complex}
\scrC = \mathit{CC}^*(\scrA) = \prod_{d \geq 0} \mathit{Hom}(\scrA[1]^{\otimes d}, \scrA).
\end{equation}
The Hochschild differential is
\begin{equation}
\begin{aligned}
& (d_{\scrC}c)^d(a_1,\dots,a_d) =
-\sum_{ij} (-1)^{\maltese_i \cdot \|c\|} \mu^{d-j+1}_{\scrA}(a_1,\dots,c^j(a_{i+1},\dots,a_{i+j}),\dots,a_d) \\
& \qquad \qquad \qquad \qquad
+ \sum_{ij} (-1)^{\maltese_i + \|c\|} c^{d-j+1}(a_1,\dots,\mu^j_{\scrA}(a_{i+1},\dots,a_{i+j}),\dots,a_d)
\end{aligned}
\end{equation}
(we apologize for the double use of $d$ as differential and as counting the number of entries); and its cohomology is the Hochschild cohomology $\mathit{HH}^*(\scrA)$. We will also use Hochschild cohomology with coefficients in a commutative ring $R$, denoted by $\mathit{HH}^*(\scrA;R)$, which is the cohomology of $\mathit{CC}^*(\scrA;R) = \scrC \hat\otimes R$ (here, completion means that we take each term in \eqref{eq:hochschild-complex} $\otimes R$ and then their product). $\scrC$ carries a canonical $A_\infty$-structure, with $\mu^1_{\scrC} = -d_{\scrC}$, and where the next term is
\begin{equation} \label{eq:hochschild-structure}
\begin{aligned}
& \mu^2_{\scrC}(c_1,c_2)^d(a_1,\dots,a_d) = \sum_{\substack{i_1,j_1,i_2,j_2 \\ i_2 \geq i_1+j_1}} 
(-1)^{\maltese_{i_1} \|c_1\| + \maltese_{i_2} \|c_2\|}
\mu^{d-j_1-j_2+2}_{\scrA}( a_1,\dots, \\ & \qquad \qquad \qquad \qquad c_1^{j_1}(a_{i_1+1},\dots,a_{i_1+ j_1}),  
 \dots, c_2^{j_2}(a_{i_2+1},\dots,a_{i_2+j_2}), \dots, a_d).
\end{aligned}
\end{equation}
The higher order $A_\infty$-operations  follow the same pattern as $\mu^2_{\scrC}$. If $\scrA$ has a homological unit, then so does $\scrC$. One way to show that is to apply Lemma \ref{th:new-unit}: in that situation, the restriction from the Hochschild complex of the $A_\infty$-category to the Hochschild complex of either $\scrA$ or $\tilde{\scrA}$ is a homotopy equivalence, allowing one to transfer properties from $\tilde{\scrA}$ to $\scrA$ in two steps.

Note that strictly speaking, $\scrC$ does not fit into the original context for $A_\infty$-rings, because \eqref{eq:hochschild-complex} is not usually free. However, it is the inverse limit of chain complexes of free groups, by using the (complete decreasing) length filtration, which is compatible with the $A_\infty$-structure. All the associated notions have to be modified to take this ``pro-object'' nature into account. We have already done that when defining Hochschild cohomology with coefficients, by using the completed tensor product $\scrC \hat\otimes R$. Maurer-Cartan elements, and homotopies between such elements, will live in such completed tensor products. To prove the analogue of Lemma \ref{th:alternative-mc} for Hochschild complexes, one again uses reduction to the strictly unital case via Lemma \ref{th:new-unit}.
 
The product on Hochschild cohomology induced from $\mu^2_{\scrC}$ is graded commutative. Additionally, Hochschild cohomology has a Lie bracket of degree $-1$. The two combine to form the structure of a Gerstenhaber algebra. When we take coefficients in a ring with $pR = 0$, let's say for concreteness $R = \bF_p$, there is one more operation 
\begin{equation} \label{eq:hh-theta}
\Xi_{\scrA,p}: \mathit{HH}^l(\scrA; \bF_p) \longrightarrow \mathit{HH}^{pl-(p-1)}(\scrA;\bF_p) 
\quad \begin{cases}\text{for odd $l$ if $p>2$,} \\ \text{for all $l$ if $p = 2$.}\end{cases}
\end{equation}
This combines with the bracket to form a restricted Lie algebra \cite{zimmermann06}. As we will now explain, following \cite{tourtchine06}, the underlying chain level map can be written as a sum over trees.

\begin{terminology} \label{th:trees}
A rooted tree with $d$ leaves is a tree which (in addition to its finite edges) has $d+1$ semi-infinite edges. One of the semi-infinite edges is singled out, and called the root; the other $d$ are the leaves. There is a unique way of orienting edges, so that they point towards the root. Given a vertex $v$, write $|v|$ for its valence. Among the edges adjacent to $v$, there is a unique outgoing one, and $\|v\| = |v|-1$ incoming ones.

In our applications, the rooted trees (unless otherwise indicated) come with the following structure. First, an ordering of the semi-infinite edges by $\{0,\dots,d\}$, starting with the root. Secondly, at any vertex, an ordering of the adjacent edges by $\{0,\dots,|v|\}$, again starting with the outgoing edge. A special case is that of rooted planar trees, where all orderings come from a single embedding of the tree into the plane, which implies certain compatibilities between them.
\end{terminology}

For now, we will only use rooted planar trees (the more general version will play a role later on, see Section \ref{subsec:fm}). Given such a tree and a Hochschild cochain $c$, one defines an operation $\scrA^{\otimes d} \rightarrow \scrA$, by starting with elements of $\scrA$ at the leaves, and having $c^{\|v\|}$ act at each vertex, with the output of that fed into the next vertex on our way to the root. To define the chain map underlying \eqref{eq:hh-theta} one considers those operations for trees with $p$ vertices, and adds them up with certain multiplicities: the multiplicity of a tree is the number of ways to order its vertices, so that the ordering increases when going towards the root (``causal orderings''). For $p = 2$, we get 
\begin{equation} \label{eq:theta-2}
(\Xi_{\scrA,2} c)^d(a_1,\dots,a_d) = \sum_{ij} c^{d-j+1}(a_1,\dots,c^j(a_{i+1},\dots,a_{i+j}),\dots,a_d).
\end{equation}
This is usually written as $c \circ c$, where $\circ$ is the operation which underlies the homotopy commutativity of $\mu^2_{\scrC}$, and which upon antisymmetrization yields the Lie bracket. The $p = 3$ case is less familiar \cite[Example 3.3]{tourtchine06}:
\begin{equation} \label{eq:theta-3}
\begin{aligned}
(\Xi_{\scrA,3} c)^d(a_1,\dots,a_d) = &  \;\;2 \!\!\!\sum_{\substack{i_1,j_1,i_2,j_2\\ i_1+j_1 \leq i_2}}
c^{d-j_1-j_2+2}( a_1,\dots, c^{j_1}(a_{i_1+1},\dots,a_{i_1+ j_1}),
\\[-1em] & \qquad \qquad  \qquad \qquad \qquad
 \dots, c^{j_2}(a_{i_2+1},\dots,a_{i_2+j_2}), \dots, a_d)
\\ & + \sum_{i_1,j_1,i_2,j_2} c^{d-j_1-j_2+2}(a_1,\dots, c^{j_1}(a_{i_1+1}, \dots c^{j_2}(a_{i_2+1},\dots,a_{i_2+j_2}), \\[-.5em]
& \qquad \qquad \qquad \qquad \qquad \dots, a_{i_1+j_1+j_2-1}),\dots,a_d).
\end{aligned}
\end{equation}
The summands in \eqref{eq:theta-3} correspond to trees as in Figure \ref{fig:3-tree}, where that on the left admits two causal orderings. Koszul signs as in \eqref{eq:hochschild-structure} are absent here, since $\|c\|$ is even (recall that for odd $p$, the operation $\Xi_{\scrA,p}$ is only defined on odd degree Hochschild cohomology).
\begin{figure}
\begin{centering}
\includegraphics{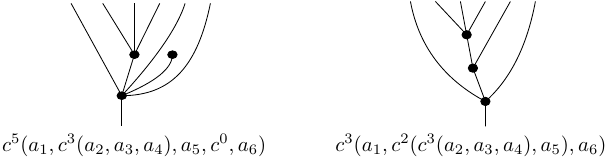}%
\caption{\label{fig:3-tree}Sample trees corresponding to expressions from \eqref{eq:theta-3}.}
\end{centering}
\end{figure}

\begin{example}
The first terms of $d_{\scrC}(c) = 0$, for $\|c\|$ even, are
\begin{equation}
\begin{aligned}
& \mu^1_{\scrA}(c^0) = 0, \\
& \mu^1_{\scrA}(c^1(a)) + \mu^2_{\scrA}(a,c^0) + \mu^2_{\scrA}(c^0,a) = c^1(\mu_{\scrA}^1(a)), \\
& \mu^1_{\scrA}(c^2(a_1,a_2)) + \mu^2_{\scrA}(c^1(a_1),a_2) + \mu^2_{\scrA}(a_1,c^1(a_2)) + 
\mu^3_{\scrA}(c^0,a_1,a_2) + \mu^3_{\scrA}(a_1,c^0,a_2) \\ & \qquad \qquad + \mu^3_{\scrA}(a_1,a_2,c^0) 
= 
c^1(\mu^2_{\scrA}(a_1,a_2)) + c^2(\mu^1_{\scrA}(a_1),a_2) + (-1)^{\|a_1\|} c^2(a_1,\mu^1_{\scrA}(a_2)).
\end{aligned}
\end{equation}
The constant term in \eqref{eq:theta-3} is
\begin{equation}
(\Xi_{\scrA,3} c)^0 = 2c^2(c^0,c^0) + c^1(c^1(c^0)).
\end{equation}
One sees that this is again a cocycle modulo 3:
\begin{equation}
\begin{aligned}
& \mu_{\scrA}^1(c^2(c^0,c^0)) = -\mu^2_{\scrA}(c^1(c^0),c^0) - \mu^2_{\scrA}(c^0,c^1(c^0))
 -3\mu^3_{\scrA}(c^0,c^0,c^0) + c^1(\mu^2_{\scrA}(c^0,c^0)) 
 \\ & \qquad \qquad = \mu_{\scrA}^1(c^1(c^1(c^0))) - 3\mu^3_{\scrA}(c^0,c^0,c^0) 
 + 3c^1(\mu^2_{\scrA}(c^0,c^0)).
 \end{aligned}
\end{equation}
\end{example}

\begin{example}
Suppose that $\scrA$ is a differential graded algebra ($\mu_{\scrA}^d = 0$ for $d \geq 3$). A derivation of $\scrA$ gives a cocycle in $\scrC^1$, and applying \eqref{eq:hh-theta} amounts to taking the $p$-th iterate of that derivation.
\end{example}

\subsection{The formal group structure}
Given an adic ring $N$, let $\scrC \hat\otimes N$ be the space obtained by taking each factor in \eqref{eq:hochschild-complex} $\hat\otimes N$, and then again forming their product. We consider Maurer-Cartan elements $\gamma \in \scrC \hat\otimes N$. Concretely, the first terms are
\begin{equation}
\begin{aligned}
& \gamma^0 \in \scrA^1 \hat\otimes N, && \textstyle \sum_d \mu^d_{\scrA}(\gamma^0,\dots,\gamma^0) = 0, \\
& \gamma^1 \in \mathit{Hom}(\scrA,\scrA)^0 \hat\otimes N,
&& \textstyle \sum_{p,q} \mu^{p+q+1}_{\scrA}(\underbrace{\gamma^0,\dots,\gamma^0}_{p},a,
\underbrace{\gamma^0\dots,\gamma^0}_{q}) = \gamma^1(\mu^1_{\scrA}(a)), \\[-1em] & \cdots
\end{aligned}
\end{equation}
One can think of $\gamma$ as a formal deformation of the identity endomorphism of $\scrA$. What this means is that $\gamma$ satisfies \eqref{eq:mc} if and only if, over $\bZ 1 \oplus N$,
\begin{equation} \label{eq:gamma-to-phi}
\phi^d = \begin{cases} \mathit{id}_{\scrA} + \gamma^1 & d = 1, \\ \gamma^d & d \neq 1 \end{cases}
\end{equation}
satisfies the (curved) $A_\infty$-homomorphism equations. Similarly, two Maurer-Cartan solutions are equivalent \eqref{eq:mc-equivalence} if the associated $A_\infty$-homomorphisms \eqref{eq:gamma-to-phi} are (curved) homotopic. The standard composition of $A_\infty$-homomorphisms \eqref{eq:gamma-to-phi} leads to the following composition law for Maurer-Cartan solutions:
\begin{equation} \label{eq:gamma-composition}
\begin{aligned}
& (\gamma_1 \bullet \gamma_2)^d(a_1,\dots,a_d)  = \gamma_2^d(a_1,\dots,a_d) \\ &  + 
\hspace{-3em} \sum_{\substack{m \geq 0 \\ i_1,j_1,\dots,i_m,j_m \\ i_1+j_1 \leq i_2, \dots, i_{m-1}+j_{m-1} \leq i_m}} \hspace{-3em}
\gamma_1^{d-j_1-\dots-j_m+m}(a_1,\dots,\gamma_2^{j_1}(a_{i_1+1},\dots,a_{i_1+j_1}), 
a_{i_1+j_1+1}, \dots, \\[-1.5em] & \qquad \qquad \qquad \qquad \qquad \gamma_2^{j_2}(a_{i_2+1},\dots,a_{i_2+j_2}), \dots, \gamma_2^{j_m}(a_{i_m+1},\dots,a_{i_m+j_m}), \dots).
\end{aligned}
\end{equation}
This is strictly associative, and descends to a product on $\mathit{MC}(\scrC;N)$. Moreover, by explicitly solving the equation $\phi_2 \phi_1 = \mathit{id}_{\scrA}$, one sees that this composition has inverses. The outcome is that $N \mapsto \mathit{MC}(\scrC;N)$ comes with the structure of a ``formal group''. The analogue of Theorem \ref{th:p-power} in this algebraic context is \cite[Equation (3-1)]{tourtchine06}:

\begin{lemma} \label{th:xi-algebra}
There is a commutative diagram
\begin{equation}
\xymatrix{
\ar[d]_-{\text{projection}}
\mathit{MC}(\scrC;q\bF_p[q]/q^{p+1}) \ar[rrrr]^-{\text{$p$-th power of the formal group}} &&&& \mathit{MC}(\scrC;q\bF_p[q]/q^{p+1}) 
\\
\mathit{MC}(\scrC;q\bF_p[q]/q^2)  &&&& \mathit{MC}(\scrC;q^p\bF_p[q]/q^{p+1}) \ar[u]_-{\text{inclusion}}
\\
\ar@{=}[u]
\mathit{HH}^1(\scrA;\bF_p) \ar[rrrr]^{\Xi_{\scrA,p}} &&&& \mathit{HH}^1(\scrA;\bF_p). \ar@{=}[u]
}
\end{equation}
\end{lemma}

The proof is quite straightforward. Namely, let's iterate \eqref{eq:gamma-composition} to form the $p$-th power of a Maurer-Cartan element $\gamma$. The outcome can be written as a sum over rooted planar trees, with multiplicities. These multiplicities count ``causal labelings'' of trees, where the vertices are labeled by $\{1,\dots,p\}$ and the numbers increase when going towards the root. This limits the depth of the tree to be $\leq p$, but does not by itself limit the number of vertices, since several vertices can carry the same label. However, in the formula for the $p$-th power map, each vertex carries a copy of $\gamma$, and since the coefficient ring $N = q\bF_p[q]/q^{p+1}$ satisfies $N^{p+1} = 0$, the contribution from trees with $>p$ vertices vanishes. The labels on trees with $\leq p$ vertices can be thought as consisting of two pieces: a choice of subset of $\{1,\dots,p\}$, and then a choice of labels which uses all numbers in that subset, and which obeys the causality condition. From that, it follows that the only trees with nontrivial mod $p$ contribution are those with exactly $p$ vertices, and where each label is used once. If we write $\gamma = c q + O(q^2)$, it then follows that
\begin{equation}
\overbrace{\gamma \bullet \cdots \bullet \gamma}^p = \Xi_{\scrA,p}(c)\, q^p \in \scrC \hat\otimes q\bF_p[q]/q^{p+1}.
\end{equation}

\begin{remark}
In characteristic zero, the deformation theory associated to the Maurer-Cartan equation in $\scrC$ is unobstructed: as a concrete illustration, the truncation map
\begin{equation}
\mathit{MC}(\scrC; q\bQ[[q]]) \longrightarrow \mathit{MC}(\scrC; q\bQ[q]/q^2) = \mathit{HH}^1(\scrA;\bQ)
\end{equation}
is onto. This is closely related to the formal group structure, since one can prove it by formal exponentiation. 
The analogous statement in positive characteristic is no longer generally true. The square of a class in $\mathit{HH}^1(\scrA;\bF_2)$ is not necessarily zero, and that gives an obstruction to lifting to $\mathit{MC}(\scrC;q\bF_2[q]/q^3)$. As an example, take a polynomial ring $\scrA = \bZ[a]$ with $|a| = 1$; the element $a$ becomes central over $\bF_2$, hence gives a Hochschild cohomology class. Instead, one could look at the $2$-adic lifting problem, but that's also obstructed: in the first step, which means lifting to $\mathit{MC}(\scrC; 2\bZ / 8\bZ)$, the requirement is that the square of the Hochschild cohomology class must be equal to its Bockstein (which fails in the same example).
%
\end{remark}

\begin{remark} \label{th:braces}
If $\scrA$ is a dg algebra, the Hochschild complex has the same structure. Let's follow classical notation and write $\smile$ for the product on Hochschild cochains. The Maurer-Cartan equation is
\begin{equation} \label{eq:mc-smile}
d_{\scrC}\gamma + \gamma \smile \gamma = 0,
\end{equation}
and two solutions are equivalent if 
\begin{equation}
d_{\scrC}h + (\gamma + \gamma \smile h) - (\tilde{\gamma} + h \smile \tilde{\gamma}) = 0.
\end{equation}
The composition law \eqref{eq:gamma-composition} can be written in terms of the brace operations from \cite{gerstenhaber-voronov95} as
\begin{equation} \label{eq:gamma-composition-braces}
\gamma_1 \bullet \gamma_2 = \gamma_2 + \sum_{m \geq 0} \gamma_1 \{\overbrace{\gamma_2,\dots, \gamma_2}^m\}.
\end{equation}
When put in this way, the formalism can be generalized to any complex $\scrC$ which is an algebra over the braces operad \cite{mcclure-smith02}, since that exactly provides the operations used in \eqref{eq:mc-smile}--\eqref{eq:gamma-composition-braces}. The formula \eqref{eq:gamma-composition-braces} can be viewed as an application of a construction \cite{gerstenhaber-voronov95b}  (see \cite{young13} for a review and further context) which equips the tensor coalgebra of $\scrC$ with a bialgebra structure. It is possible that the geometric results in this paper could be similarly sharpened, replacing ``formal groups'' with a suitable bialgebra language (where the comultiplication would be the standard tensor coalgebra structure, but the multiplication would be $A_\infty$); however, that would likely require the full generality of Bottman's witch ball spaces.
\end{remark}


\section{Parameter spaces\label{sec:operads}}

This section discusses the moduli spaces underlying our constructions. This is mostly an exposition of known material; the small amount that may be new appears towards the end of the section. Stasheff associahedra, Deligne-Mumford spaces, and Fulton-MacPherson spaces (for the latter, originally in their quasi-isomorphic guise \cite{salvatore01} as the little squares operad) belong to classical algebraic topology and geometry, and we include a brief exposition mainly as a warmup exercise. The more complicated spaces are borrowed from the theory of Lagrangian correspondences, variously combining \cite{mau-woodward10, mau-wehrheim-woodward18, bottman-wehrheim18, fukaya17, bottman17a, bottman17}.

\subsection{Associahedra\label{subsec:apply-stasheff}}
The Stasheff spaces (associahedra) $S_d$, $d \geq 2$, are compactifications of the space of ordered point configurations on the real line, modulo translations and positive dilations, meaning of
\begin{equation} \label{eq:configuration-space-1}
\frac{
\{ (s_1,\dots,s_d) \in \bR^d, \;\; s_1 < \cdots < s_d  \} }{ \{ (s_1,\dots,s_d) \sim (\tau(s_1),\dots,\tau(s_d))\;\; \text{for $\tau(s) = \lambda + \mu s$, $\lambda \in \bR$, $\mu > 0$}\} }.
\end{equation}
The collection $\{S_d\}$ has the structure of a non-symmetric operad, given by maps
\begin{equation} \label{eq:glue-tree-1}
\prod_{v}\! S_{\|v\|} \stackrel{T}{\longrightarrow} S_d,
\end{equation}
one for each rooted planar tree $T$ with $(d+1)$ semi-infinite edges, and where every vertex $v$ has valence $|v| \geq 3$ (see Terminology \ref{th:trees}; it will be our standard procedure to just denote such maps by the underlying tree). The single-vertex tree is a trivial special case, since it gives rise to the identity map on $S_d$. 

Topologically, $S_d$ is a (contractible) compact manifold with boundary, whose interior is \eqref{eq:configuration-space-1},
and whose boundary is the union of the images of the nontrivial maps \eqref{eq:glue-tree-1}. One can get a slightly more precise description by introducing a suitable smooth structure, for instance by embedding the Stasheff spaces into the real locus of Deligne-Mumford spaces. Then $S_d$ becomes a smooth (and in fact real sub-analytic) manifold with corners, whose open strata are the images of  $\prod_v (S_{\|v\|} \setminus \partial S_{\|v\|})$ under \eqref{eq:glue-tree-1}.

We orient $S_d$ by picking, on the interior \eqref{eq:configuration-space-1}, the parametrization where $(s_1,s_2)$ are fixed, and using the standard orientation of the remaining parameters $(s_3,\dots,s_d)$.

\subsection{Fulton-MacPherson spaces\label{subsec:fm}}
The Fulton-MacPherson spaces (the terminology is taken from \cite{getzler-jones94}; versions of the construction arose in \cite{axelrod-singer94, fulton-macpherson94, kimura-stasheff-voronov95}) $\mathit{FM}_d$, $d \geq 2$, are compactifications of planar configuration space up to translations and positive dilations:
\begin{equation} \label{eq:configuration-space-2}
\frac{
\{ (z_1,\dots,z_d) \in \bC^d \;:\; z_i \neq z_j \text{ for } i \neq j \} }{ \{ (z_1,\dots,z_d) \sim (\tau(z_1),\dots,\tau(z_d))\;\; \text{for $\tau(z) = \lambda + \mu z$, $\lambda \in \bC$, $\mu > 0$}\} }.
\end{equation}
The (symmetric) operad structure on $\{\mathit{FM}_d\}$ comes from permutations of the $z_k$, together with maps similar to \eqref{eq:glue-tree-1},
\begin{equation} \label{eq:glue-tree-2}
\prod_v \mathit{FM}_{\|v\|} \stackrel{T}{\longrightarrow} \mathit{FM}_d.
\end{equation}
Here, the rooted trees $T$ come with our usual structure (see Terminology \ref{th:trees}), but are not necessarily planar.
Changing the ordering of the semi-infinite edges of $T$ amounts to composing \eqref{eq:glue-tree-2} with an element of $\mathit{Sym}_d$ on the left; and changing the orderings at the vertices amounts to composing \eqref{eq:glue-tree-2} with an element of $\prod_v \mathit{Sym}_{\|v\|}$ on the right. The inclusion $\bR \subset \bC$ induces maps
\begin{equation} \label{eq:s-into-fm}
S_d \longrightarrow \mathit{FM}_d,
\end{equation}
which are compatible with \eqref{eq:glue-tree-1}, \eqref{eq:glue-tree-2} (they form a morphism of non-symmetric operads).

As before, $\mathit{FM}_d$ is topologically a compact manifold with boundary. One can complexify it by considering point configurations in $\bC^2$, which yields a smooth compact complex manifold, and then embed $\mathit{FM}_d$ into the real locus of that. As a consequence, it inherits the structure of a smooth (or real sub-analytic) manifold with corners, just as in the case of the associahedra. 

To orient $\mathit{FM}_d$, we consider representatives in \eqref{eq:configuration-space-2} where $z_1$ and $|z_1-z_2|$ are fixed. Then, rotating $z_2$ anticlockwise around $z_1$ yields the first coordinate, and the remaining coordinates are $(z_3,\dots,z_d)$ with their complex orientations. Equivalently, consider the classical configuration space $\mathit{Conf}_d(\bC)$, of which \eqref{eq:configuration-space-2} is a quotient by the action of $(\lambda,\mu) \in \bC \times \bR^{>0}$. The Lie algebra of that group fits into an exact sequence
\begin{equation} \label{eq:split-orientations}
0 \rightarrow \bC \oplus \bR \longrightarrow T_{(z_1,\dots,z_d)}\mathit{Conf}_d(\bC) \longrightarrow
T_{(z_1,\dots,z_d)} \mathit{FM}_d \rightarrow 0;
\end{equation}
our orientation of the quotient is compatible with that sequence and with the complex orientation of $\mathit{Conf}_d(\bC)$. In particular, $\mathit{Sym}_d$ acts orientation-preservingly. 

\subsection{Deligne-Mumford spaces}
For most of this paper, we will write $\mathit{DM}_d$ for the Deligne-Mumford moduli space of genus $0$ curves with $(d+1)$ marked points, bringing it in line with the notation for the other moduli spaces. One can consider it as a compactification of
\begin{equation} \label{eq:configuration-space-3}
\frac{
\{ (z_1,\dots,z_d) \in \bC^d \;:\; z_i \neq z_j \text{ for } i \neq j \} }{ \{ (z_1,\dots,z_d) \sim (\tau(z_1),\dots,\tau(z_d))\;\; \text{for $\tau(z) = \lambda + \mu z$, $\lambda \in \bC$, $\mu \in \bC^*$}\} },
\end{equation}
which is a free $S^1$ quotient of \eqref{eq:configuration-space-2}. The operadic structure takes on exactly the same form as for Fulton-MacPherson spaces. Indeed, the quotient map on configuration spaces extends to a map
\begin{equation} \label{eq:fm-to-dm}
\mathit{FM}_d \longrightarrow \mathit{DM}_d,
\end{equation}
which is compatible with \eqref{eq:glue-tree-2} and its Deligne-Mumford counterpart. 

We adopt the usual orientation of $\mathit{DM}_d$ as a complex manifold.

\subsection{Colored multiplihedra\label{subsec:mww}}
Ma'u-Wehrheim-Woodward \cite{mau-woodward10, mau-wehrheim-woodward18} introduced a geometric interpretation of the classical multiplihedra, as well as certain generalizations. We will call these spaces colored multiplihedra, and denote them by
\begin{equation} \label{eq:mww-spaces}
\mathit{MWW}_{d_1,\dots,d_r}, \;\; r \geq 1, \; d_1,\dots,d_r \geq 0, \; d = d_1+\cdots + d_r > 0.
\end{equation}
They are compactifications of
\begin{equation} \label{eq:configuration-space-4}
\frac{
\{ (s_{1,1},\dots,s_{1,d_1}; \dots ;s_{r,1},\dots,s_{r,d_r}) \in \bR^d, \;\; s_{k,1} <  \cdots < s_{k,d_k} \text{ for each $k$}\} 
}{ 
\{ s_{k,i} \sim s_{k,i} + \mu \; \text{for $\mu \in \bR$}\} 
}.
\end{equation}
The intuitive meaning of \eqref{eq:configuration-space-4} is that we have $d$ points on the real line, which are divided into $r$ colors, with $d_k$ points of any given color $k$. Points of different colors can have the same position, while those of the same color are distinct and lie on the real line in increasing order. We denote the compactification by $\mathit{MWW}_{d_1,\dots,d_r}$. It tracks what happens on a large scale, meaning the relative speeds as points diverge from each other, as well as on the small scale, where points of the same color converge. Therefore, a point in the compactification consists of ``screens'' (terminology taken from \cite{fulton-macpherson94}) which are either ``large-scale'', ``mid-scale'', or ``small-scale''. Correspondingly, the analogue of \eqref{eq:glue-tree-1} is of the form
\begin{equation} \label{eq:glue-tree-3}
\prod_{v \text{ large}} S_{\|v\|} \times \prod_{v \text{ mid}} \mathit{MWW}_{\|v\|_1,\dots,\|v\|_r} \times \prod_{v \text{ small}} S_{\|v\|} \stackrel{T}{\longrightarrow} \mathit{MWW}_{d_1,\dots,d_r}.
\end{equation}
Here, the tree $T$ has $d+1$ semi-infinite edges. We still single out a root, but the leaves are now divided into subsets of orders $d_k$, each subset being then ordered by $\{1,\dots,d_k\}$. Each vertex has one of three scales. The mid-scale vertices have the same kind of combinatorial data attached to them as the entire tree: their incoming edges are divided into $r$ subsets of different colors, whose sizes we denote by $\|v\|_1,\dots,\|v\|_r$, and then ordered within each subset. The large-scale vertices and small-scale vertices just come with an ordering of the incoming edges. The small-scale vertices are also labeled with a color in $\{1,\dots,r\}$. Any path going from a leaf of color $k$ to the root travels in nondecreasing order of scale: first through any number (which can be zero) of small-scale vertices of color $k$; then through exactly one single mid-scale vertex, which it enters by an edge with color $k$; and finally, through any number (which can be zero) of large-scale vertices. There are compatibility conditions between the orderings, which are somewhat tedious to write down combinatorially, see \cite[Section 6]{mau-wehrheim-woodward18} (they are similar in principle to those for planar rooted trees, but concern each color separately).

\begin{example} \label{th:ex-deg}
Suppose that in $\mathit{MWW}_{2,2}$, we have a sequence of configurations where one point (of the first color) moves to $-\infty$, and the remaining three points move towards the same position. The outcome is shown in Figure \ref{fig:limiting-configuration}.
\end{example}
\begin{figure}
\begin{centering}
\includegraphics{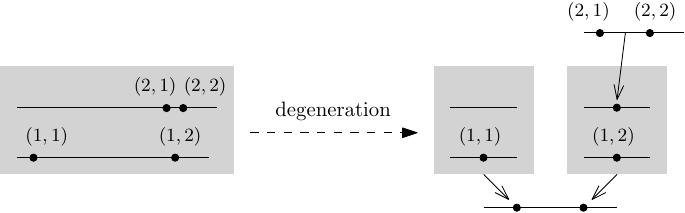}
\caption{\label{fig:limiting-configuration}A degeneration in $\mathit{MWW}_{2,2}$, see Example \ref{th:ex-deg}. The shaded regions are ``mid-scale screens''; we have drawn marked points of different colors as lying on separate copies of the real line.
}
\end{centering}
\end{figure}

Topologically, $\mathit{MWW}_{d_1,\dots,d_r}$ is again a compact manifold with boundary, having \eqref{eq:configuration-space-4} as its interior. Note that the codimension of the image of \eqref{eq:glue-tree-2} is the number of small-scale plus large-scale vertices, mid-scale vertices being irrelevant. As a consequence of the resulting combinatorial structure of boundary strata, $\mathit{MWW}_{d_1,\dots,d_r}$ can't be made into a smooth manifold with corners in the same way as the previously considered moduli spaces. However, it is naturally a (sub-analytic) manifold with generalized corners in the sense of \cite{joyce16}. To prove that, one introduces a complexification as in \cite{mau-woodward10, bottman-oblomkov19}, which is a complex variety with toric singularities, and embeds $\mathit{MWW}_{d_1,\dots,d_r}$ into its real locus.

As for orientations, we orient \eqref{eq:configuration-space-4} by ordering the coordinates lexicographically, and then keeping the first one fixed to break the translation-invariance.

\begin{example} \label{th:111-space}
In the spaces $\mathit{MWW}_{1,\dots,1}$, no small-scale vertices can appear. The maps \eqref{eq:glue-tree-3} with zero-dimensional domains correspond to trivalent planar rooted trees with an additional ordering of the $r$ leaves, hence there are $(2r-2)!/(r-1)!$ of them. For instance, the two-dimensional space $\mathit{MWW}_{1,1,1}$ is a $12$-gon, see Figure \ref{fig:12-gon}. The boundary sides each have either one large-scale screen containing three points, or one mid-scale screen with two points (each possibility occurs six times). Figure \ref{fig:12-gon-2} shows in more detail a neighbourhood of one of the corners of the $12$-gon, and in particular, the degenerate configuration associated to the vertex.
\end{example}
\begin{figure}
\begin{centering}
\includegraphics{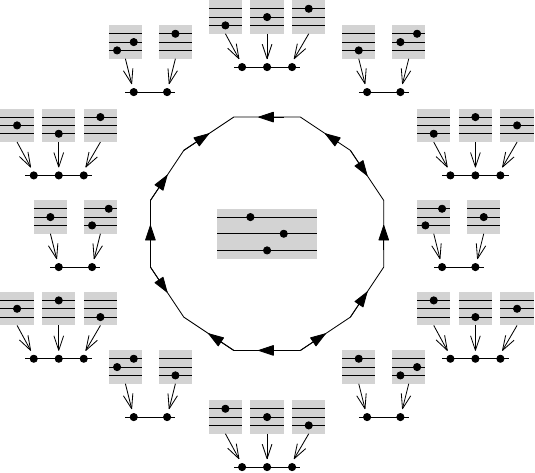}
\caption{\label{fig:12-gon}The space $\mathit{MWW}_{1,1,1}$ from Example \ref{th:111-space}.}
\end{centering}
\end{figure}%
\begin{figure}
\begin{centering}
\includegraphics{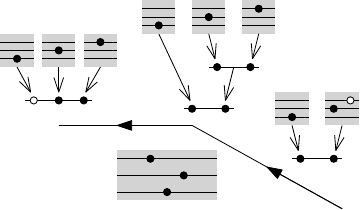}
\caption{\label{fig:12-gon-2}A specific part of $\mathit{MWW}_{1,1,1}$, compare Figure \ref{fig:12-gon}. As one approaches the vertex along the edge from the left, the leftmost of the three points on the large-scale screen moves to $-\infty$. As one approaches it along the other edge from the right, the rightmost point on the mid-scale screens moves to $+\infty$. We have colored the points that we think of as moving white. (Of course, because of translation-invariance, there are other equivalent ways of thinking about the degenerations.)}
\end{centering}
\end{figure}%

\begin{example} \label{th:21-space}
The space $\mathit{MWW}_{2,1}$ is an octagon, see Figure \ref{fig:8-gon}. There are only two points which belong to the same color, hence only one way in which a small-scale vertex can occur, which is the boundary side at the top of our figure. The other boundary sides are of two kinds, as in Example \ref{th:111-space}.
\end{example}
\begin{figure}
\begin{centering}
\includegraphics[scale=0.8]{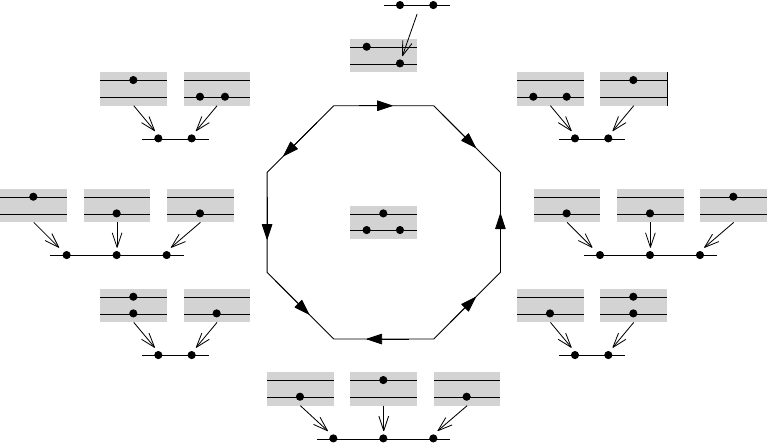}
\caption{\label{fig:8-gon}The space $\mathit{MWW}_{2,1}$ from Example \ref{th:21-space}. }
\end{centering}
\end{figure}%

One can associate to a real colored configuration a complex configuration, by setting
\begin{equation} \label{eq:imaginary-part-is-k}
z_{k,i} = s_{k,i} + k \textstyle\sqrt{-1}
\end{equation}
and then ordering the $z_{k,i}$ lexicographically (we use $\sqrt{-1}$ here to avoid notational confusion with the index $i$). This extends to a continuous map 
\begin{equation} \label{eq:mww-into-fm}
\mathit{MWW}_{d_1,\dots,d_r} \longrightarrow \mathit{FM}_d, \;\; \text{provided that } d = d_1 + \cdots + d_r \geq 2.
\end{equation}
In terms of \eqref{eq:glue-tree-3}, the extension uses the same formula \eqref{eq:imaginary-part-is-k} for the points on each  mid-scale screen, while the small-scale and large-scale screens use \eqref{eq:s-into-fm}. To be precise, there is one exception: mid-scale vertices with $|v| = 2$ have no Fulton-MacPherson counterpart, and we simply forget about them, which is unproblematic since $\mathit{MWW}_{0,\dots,0,1,0,\dots,0} = \mathit{point}$. There is a commutative diagram involving \eqref{eq:mww-into-fm} as well as \eqref{eq:s-into-fm}, \eqref{eq:glue-tree-2}, \eqref{eq:glue-tree-3}:
\begin{equation} \label{eq:mww-into-fm-compatible}
\xymatrix{
\ar[d]_-{\eqref{eq:mww-into-fm}}
\displaystyle \prod_{v \text{ small}} S_{\|v\|} \times \prod_{v \text{ mid}} \mathit{MWW}_{\|v\|_1,\dots,\|v\|_r} \times \prod_{v \text{ large}} S_{\|v\|} \ar[rr]^-{\eqref{eq:glue-tree-3}} && \mathit{MWW}_{d_1,\dots,d_r} \ar[dd]^-{\eqref{eq:mww-into-fm}} \\
\displaystyle
\prod_{v \text{ small}} S_{\|v\|} \times \prod_{\substack{v \text{ mid} \\ |v| > 2}} \mathit{FM}_{\|v\|} \times \prod_{v \text{ large}}
S_{\|v\|} 
\ar[d]_-{\eqref{eq:s-into-fm}}
\\
\displaystyle \prod_{|v| > 2} \mathit{FM}_{\|v\|}  \ar[rr]_-{\eqref{eq:glue-tree-2}} && \mathit{FM}_d.
}
\end{equation}

It can be convenient to allow more flexibility in the construction of \eqref{eq:mww-into-fm}. Namely, suppose that we have a collection of continuous functions 
\begin{equation} \label{eq:tau-functions}
\tau_{d_1,\dots,d_r} = (\tau_{d_1,\dots,d_r,1,1},\dots,\tau_{d_1,\dots,d_r,r,d_r}): \mathit{MWW}_{d_1,\dots,d_r} \longrightarrow \bR^d,
\end{equation}
with the following properties. In the interior of our space,
\begin{equation} \label{eq:traffic-rules}
\tau_{d_1,\dots,d_r,k,i} < \tau_{d_1,\dots,d_r,l,j} \;\; \text{at any point of \eqref{eq:configuration-space-4} where
$s_{k,i} = s_{l,j}$ for some $k<l$ and $i,j$.}
\end{equation}
Take the pullback of $\tau_{d_1,\dots,d_r}$ by \eqref{eq:glue-tree-3} for some tree $T$. Each index $(k,i)$ corresponds to a leaf of $T$, and the path from that leaf to the root enters a single mid-scale vertex $v$ through an incoming edge labeled $(k,j)$. Then, we require that the $(k,i)$ component of the pullback must be given by the $(k,j)$-component of $\tau_{\|v\|_1,\dots,\|v\|_r}$, as a function on the product in \eqref{eq:glue-tree-3}.  Instead of \eqref{eq:imaginary-part-is-k}, we can then set
\begin{equation} \label{eq:imaginary-part-is-tau}
z_{k,i} = s_{k,i} + \tau_{d_1,\dots,d_r,k,i}  \sqrt{-1}.
\end{equation}
Intuitively, the imaginary parts of the $z_{k,i}$ can vary depending on the modular parameters, but if two points of different colors $k<l$ come to lie on the same vertical axis, the point with the higher color $l$ always passes above that of color $k$ (in contrast, points of the same color still collide, ``bubbling off'' into a small-scale screen). The consistency condition we have imposed on \eqref{eq:tau-functions} ensures that \eqref{eq:imaginary-part-is-tau} extends to a continuous map  \eqref{eq:mww-into-fm}, with the same boundary compatibilities \eqref{eq:mww-into-fm-compatible} as before. This is a strict generalization of the previous construction, since the constant functions $\tau_{d_1,\dots,d_r,k,i} = k$ clearly satisfy our conditions. More general choices of \eqref{eq:tau-functions} can be defined inductively by extension from the boundary of $\mathit{MWW}_{d_1,\dots,d_r}$ to the entire space, which is unproblematic since \eqref{eq:traffic-rules} is a convex condition.

As one application of \eqref{eq:imaginary-part-is-tau}, note that we have (orientation-preserving) identifications
\begin{equation} \label{eq:drop-color}
\mathit{MWW}_{d_1,\dots,d_r} = \mathit{MWW}_{d_1,\dots,d_{l-1},d_{l+1},\dots,d_r} \;\; \text{if $d_l = 0$.}
\end{equation}
According to the original formula \eqref{eq:imaginary-part-is-k}, these two isomorphic spaces come with different maps to $\mathit{FM}_d$. However, when constructing the functions \eqref{eq:tau-functions}, one can additionally achieve that
\begin{equation} \label{eq:drop-color-2}
\tau_{d_1,\dots,d_r,k,i} = 
\begin{cases} \tau_{d_1,\dots,d_{l-1},d_{l+1},\dots,k,i} & k < l \\
\tau_{d_1,\dots,d_{l-1},d_{l+1},\dots,k-1,i} & k>l
\end{cases}
\;\; \text{if $d_l = 0$,}
\end{equation}
and then the maps \eqref{eq:mww-into-fm} obtained from \eqref{eq:imaginary-part-is-tau} become compatible with  \eqref{eq:drop-color}.

\subsection{Witch ball spaces\label{subsec:witch-ball-spaces}} 
Our next topic is a simplified version of Bottman's witch ball spaces \cite{bottman17}, for didactic reasons: we won't use them as such, but the discussion serves as a preparation for a related construction to be carried out afterwards. Our notation is
\begin{equation} \label{eq:bottman}
B_{d_1,\dots,(d_m,d_{m+1}),\dots,d_r}, \;\; r \geq 2, \; d = d_1 + \cdots + d_r > 0, \; 1 \leq m \leq r-1.
\end{equation}
The interior is the configuration space \eqref{eq:configuration-space-4} with an additional parameter $t \in (0,1)$. This parameter extends to a map
\begin{equation} \label{eq:map-to-t}
B_{d_1,\dots,(d_m,d_{m+1}),\dots,d_r} \longrightarrow [0,1].
\end{equation}
Over $t \in (0,1]$, we just have a copy of $(0,1] \times \mathit{MWW}_{d_1,\dots,d_r}$. In particular, by looking at $t = 1$ one gets boundary strata inherited from \eqref{eq:glue-tree-3}, which are images of maps
\begin{equation} \label{eq:glue-tree-4}
\prod_{v \text{ large}} S_{\|v\|} \times \prod_{v \text{ mid}} \mathit{MWW}_{\|v\|_1,\dots,\|v\|_r} \times \prod_{v \text{ small}} S_{\|v\|} \stackrel{T}{\longrightarrow} B_{d_1,\dots,(d_m,d_{m+1}),\dots,d_r}.
\end{equation}
At $t = 0$, the $m$-th and $(m+1)$-st color ``collide''. There, the analogue of \eqref{eq:glue-tree-4} is
\begin{equation} \label{eq:glue-tree-5}
\begin{aligned}
& 
\prod_{v \text{ mid}} \mathit{MWW}_{\|v\|_1,\dots,\|v\|_{r-1}} \times
\prod_{\!\!v \text{ small-mid}\!\!\!\!\!} \mathit{MWW}_{\|v\|_1,\|v\|_2} 
\times \prod_{\substack{v \text{ of any} \\ \!\!\text{other scale}}\!\!} S_{\|v\|}
\stackrel{T}{\longrightarrow} B_{d_1,\dots,(d_m,d_{m+1}),\dots, d_r}.
\end{aligned}
\end{equation}
This time six different scales are involved, which we call (unimaginatively) ``large'', ``mid'', ``small'', ``small-large'', ``small-mid'', ``small-small''. Suppose that we have a path from a leaf to the root. As usual, the leaves carry colors $\{1,\dots,r\}$. If the color of our leaf is $\neq m,m+1$, things proceed as for the $\mathit{MWW}$ spaces, with the path going through any number of small vertices, one mid-scale vertex, and then any number of large vertices.
(There is a re-labelling rule: if the color is $k>m+1$, it enters the mid-scale vertex through an edge with color $k-1$.) If the color is $m$ (or $m+1$), the path first goes through small-small vertices, and then through exacly one small-mid vertex, which it enters through an edge colored by $1$ (respectively $2$). It then proceeds through an arbitrary number of small-large vertices, then through a mid-scale vertex, which it always enters through the $m$-th color, following by large-scale vertices. To compute the codimension of \eqref{eq:glue-tree-5} one counts the number of ``other scale'' screens. Finally, our space has boundary strata which lie over the entire interval $[0,1]$, and those are images of maps
\begin{equation} \label{eq:glue-tree-fibered}
\prod_{v \text{ large}} S_{\|v\|} \times \prod_{v \text{ mid}}^{[0,1]} B_{\|v\|_1,\dots,(\|v\|_m,\|v\|_{m+1}),\dots,\|v\|_r} \times \prod_{v \text{ small}} S_{\|v\|} \stackrel{T}{\longrightarrow} B_{d_1,\dots,(d_m,d_{m+1}),\dots,d_r},
\end{equation}
where the superscript means that instead of a product, we have a fibre product over \eqref{eq:map-to-t}; compare \cite[Equation (1)]{bottman-carmeli18}. We refer to \cite{bottman17, bottman17a, bottman-carmeli18} for a detailed discussion; the results obtained there can easily be carried over to our version. 
\begin{figure}
\begin{centering}
\includegraphics{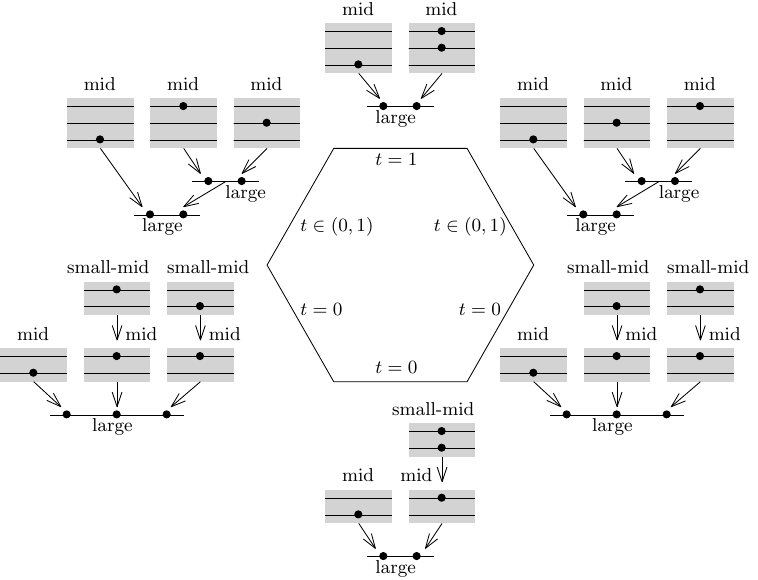}
\caption{\label{fig:6-gon}One of the boundary faces of $B_{1,(1,1)}$, see Example \ref{th:b111}.}
\end{centering}
\end{figure}

\begin{example} \label{th:b111}
The space $B_{1,(1,1)}$ is half (sliced through horizontally) of \cite[Figure 1b]{bottman-carmeli18}. Figure \ref{fig:6-gon} shows one of its boundary faces of type \eqref{eq:glue-tree-fibered}, namely $S_2 \times B_{1,(0,0)} \times_{[0,1]} B_{0,(1,1)} \iso B_{(1,1)}$.
\end{example}

The spaces \eqref{eq:bottman} are topological manifolds with boundary, and smooth manifolds with generalized corners. For Bottman's witch ball spaces, this is proved in \cite{bottman-oblomkov19}, and the same arguments apply to the (comparatively simpler) situation here. 

As was the case for the $\mathit{MWW}$ spaces, one can map our spaces to Fulton-MacPherson spaces,
\begin{equation}
B_{d_1,\dots,(d_m,d_{m+1}),\dots,d_r} \longrightarrow \mathit{FM}_d, \;\; \text{provided that $d \geq 2$,}
\end{equation}
compatibly with \eqref{eq:glue-tree-4}, \eqref{eq:glue-tree-5}, \eqref{eq:glue-tree-fibered}. Suppose for simplicity that the maps \eqref{eq:mww-into-fm} have been defined using \eqref{eq:imaginary-part-is-k}. Then, the corresponding formula for $B_{d_1,\dots,(d_m,d_{m+1}),\dots,d_r}$ is
\begin{equation} \label{eq:t-bottman}
z_{k,i} = \begin{cases} s_{k,i} + k\sqrt{-1} & k \leq m, \\
s_{k,i} + (k-1+t)\sqrt{-1} & k > m.
\end{cases}
\end{equation}
As before, the extension of this map to the entire space forgets any screens (necessarily of mid-scale or small-mid-scale) which carry configurations consisting only of one point.

\subsection{Strip-shrinking spaces\label{subsec:strip-shrinking}}
We will now introduce a modification of the idea of witch ball spaces, designed to avoid the kind of fibre products which appeared in \eqref{eq:glue-tree-fibered}. This is inspired by \cite{bottman-wehrheim18}, and correspondingly called strip-shrinking spaces. We will denote them by
\begin{equation} \label{eq:strip-shrinking}
\mathit{SS}_{d_1,\dots,(d_m,d_{m+1}),\dots,d_r}, \;\; r \geq 2, \; d = d_1 + \dots + d_r \geq 0, \; 1 \leq m \leq r-1.
\end{equation}
(Note that this time, unlike the situation in \eqref{eq:bottman}, it is possible to have all $d_k = 0$.) The $SS$ spaces compactify colored configuration space as in \eqref{eq:configuration-space-4}, but without dividing by common translation. The important point is an asymmetry between the two ways in which points in the configuration can go to infinity.
In the $s \rightarrow -\infty$ direction, we dictate a fairly standard behaviour, where $\mathit{MWW}$ spaces with $r$ colors appear. In the $s \rightarrow +\infty$ limit, we think of the $m$-st and $(m+1)$-st colored points as lying on lines that become asymptotically close to each other, at a rate of $1/s$. One way to make this more concrete is to consider the analogue of \eqref{eq:t-bottman}, which associates to a real configuration a complex one. Choose a function $\psi: \bR \rightarrow (-1,0]$ with asymptotics
\begin{equation}
\psi(s) \approx \begin{cases} 0 & s \ll 0, \\
-1 + 1/s & s \gg 0.
\end{cases}
\end{equation}
Then set (see Figure \ref{fig:convergent-example})
\begin{equation} \label{eq:imaginary-part-is-almost-k}
z_{k,i} = \begin{cases} s_{k,i} + k \sqrt{-1} & k \leq m, \\
s_{k,i} = s_{k,i} + (k + \psi(s_{k,i})) \sqrt{-1} & k > m.
\end{cases} 
\end{equation}
To relate the spaces to Fulton-MacPherson spaces, we can add two auxiliary marked points, say 
\begin{equation} \label{eq:stabilization}
z_{\pm} = \pm 1 + (r+1)\sqrt{-1}, 
\end{equation}
which stabilize the situation and otherwise stay out of the way. This gives continuous maps
\begin{equation} \label{eq:double-stabilization}
\mathit{SS}_{d_1,\dots,(d_m,d_{m+1}),\dots,d_r} \longrightarrow \mathit{FM}_{d+2}. 
\end{equation}
%
\begin{figure}
\begin{centering}
\includegraphics{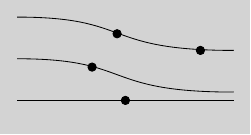}
\caption{\label{fig:convergent-example}A point in $\mathit{SS}_{(1,1),2}$, thought of as a configuration in $\bC$ as in \eqref{eq:imaginary-part-is-almost-k}.}
\end{centering}
\end{figure}%
%
%

For a more precise picture, consider the analogue of \eqref{eq:glue-tree-1},
\begin{equation} \label{eq:glue-tree-6}
\begin{aligned}
& 
\prod_{v < v_* \text{ mid}} \mathit{MWW}_{\|v\|_1,\dots,\|v\|_r} \times \mathit{SS}_{\|v_*\|_1,\dots,\|v_*\|_r}
\times \prod_{v > v_* \text{ mid}} \mathit{MWW}_{\|v\|_1,\dots,\|v\|_{r-1}} \qquad \\
& \qquad \qquad
\times \prod_{v \text{ small-mid}} \mathit{MWW}_{\|v\|_1,\|v\|_2} 
\times \prod_{\substack{v \text{ of any} \\ \text{other scale}}} S_{\|v\|}
\stackrel{T}{\longrightarrow} \mathit{SS}_{d_1,\dots,d_r}.
\end{aligned}
\end{equation}
Here, we have the same six scales as in \eqref{eq:glue-tree-5}, but with different roles. There is a distinguished mid-scale vertex, denoted by $v_*$, to which corresponds an $\mathit{SS}$ space. All other mid-scale vertices carry $\mathit{MWW}$ spaces, in two different versions: if $v<v_*$ (with respect to the ordering of mid-scale vertices determined by the ribbon structure at large-scale vertices) that space has $r$ colors, but for $v>v_*$ there are only $(r-1)$ colors. The part of the tree lying on top of $v \leq v_*$ vertices consists of small-scale vertices as in \eqref{eq:glue-tree-3}, and the same is true for $v > v_*$ if the color is $\neq m$. For that remaining color, we have a structure of small-large, small-mid, and small-small vertices parallel to \eqref{eq:glue-tree-5}.

\begin{example} \label{th:110-convergent}
Take the two-dimensional space $\mathit{SS}_{(1,1)}$, denoting points in its interior by $(s_1;s_2)$ for brevity. Consider sequences 
\begin{equation}
(s_1^{(k)}; s_2^{(k)}), \; k = 1,2,\dots, \;\; \text{where } s_1^{(k)} < s_2^{(k)} \text{ and } s_1^{(k)}; s_2^{(k)} \rightarrow +\infty.
\end{equation}
The possible limit configurations, shown in Figure \ref{fig:convergent-screens}, correspond to the following behaviours:
\begin{itemize}
\itemsep.5em
\item[(i)] $(s_2^{(k)}-s_1^{(k)})/s_1^{(k)} \rightarrow +\infty$.
\item[(ii)] $(s_2^{(k)}-s_1^{(k)})/s_1^{(k)}$ converges to a nonzero constant.
\item[(iii)] $s_2^{(k)} - s_1^{(k)} \rightarrow +\infty$, but $(s_2^{(k)}-s_1^{(k)})/s_1^{(k)} \rightarrow 0$.
\item[(iv)] $s_2^{(k)} - s_1^{(k)}$ converges to a nonzero constant.
\item[(v)] $s_2^{(k)} - s_1^{(k)} \rightarrow 0$, but $s_1^{(k)}(s_2^{(k)}-s_1^{(k)}) \rightarrow +\infty$. Since the two points get increasingly close to each other, the additional mid-scale screen carries only one marked point. On the other hand, rescaling by $s_1^{(k)}$ separates the two points in the limit. This leads to the appearance of a small-large screen. 
\item[(vi)]  $s_2^{(k)} - s_1^{(k)} \rightarrow 0$, but $s_1^{(k)}( s_2^{(k)} - s_1^{(k)} )$ converges to a constant (which can be zero). In that case, we get a small-mid scale screen with two marked points on it.
\end{itemize}
The whole space is a $14$-gon (Figure \ref{fig:14-gon}), with three adjacent sides corresponding to (ii), (iv), (vi) above, and corners corresponding to (i), (iii), (iv).
\end{example}
\begin{figure}
\begin{centering}
\includegraphics{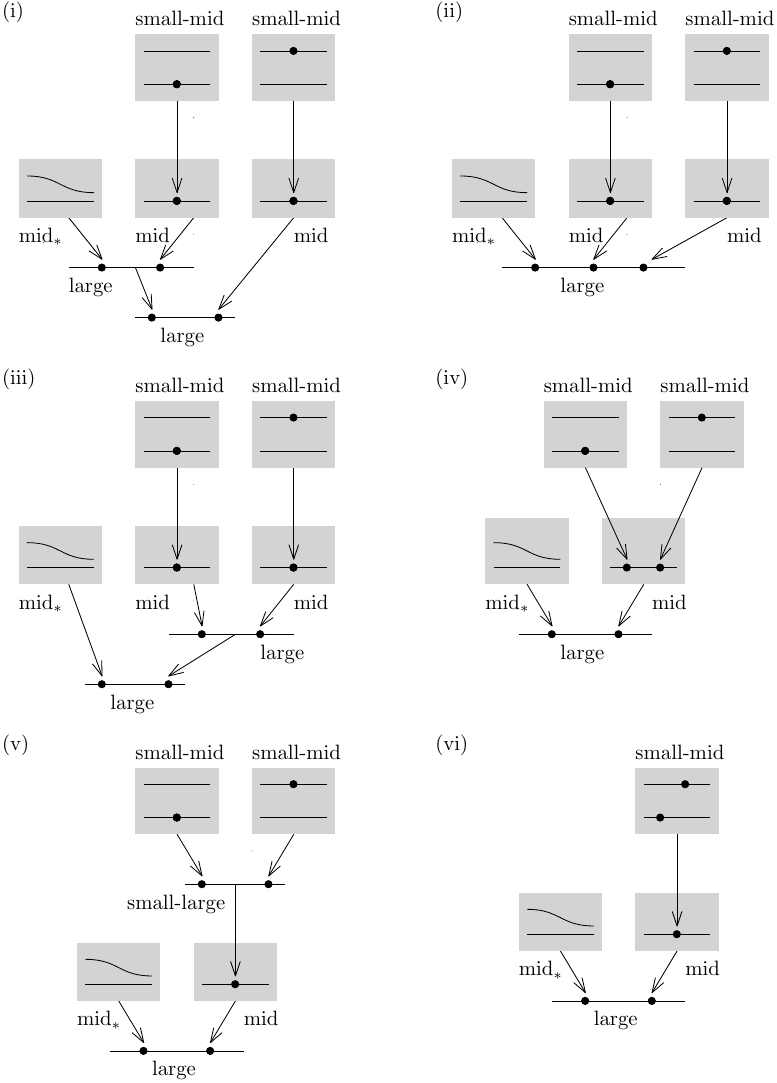}
\caption{\label{fig:convergent-screens}Some limits in $\mathit{SS}_{(1,1)}$, as discussed in Example \ref{th:110-convergent}. The $\ast$ marks the distinguished mid-scale screen.}
\end{centering}
\end{figure}%
\begin{figure}
\begin{centering}
\includegraphics{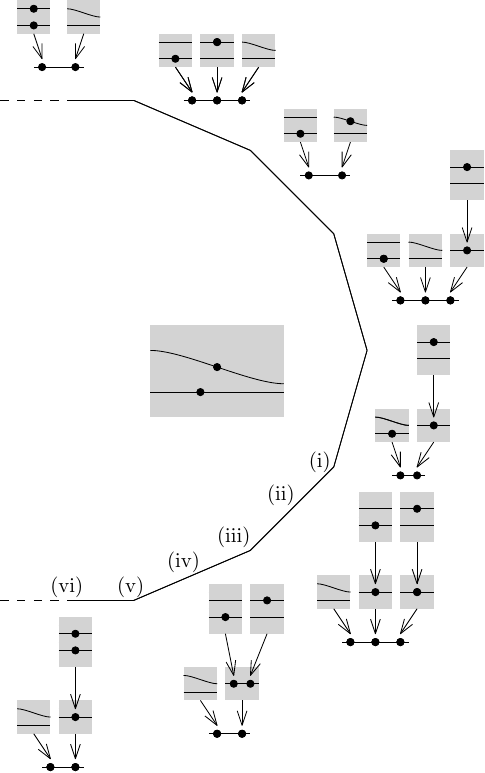}
\caption{\label{fig:14-gon}The space $\mathit{SS}_{(1,1)}$ is a $14$-gon, see Example \ref{th:110-convergent}. For space reasons, we have only drawn half of it. However, the drawing is arranged so that exchanging the first two colors corresponds to reflection along the vertical axis, and the missing half can be inferred from that.
}
\end{centering}
\end{figure}%

The structure of $\mathit{SS}$ as a compact topological space is relatively straightforward to obtain, following the model of \cite{bottman17}. It turns out that it is also a topological manifold with boundary, and in fact a differentiable manifold with generalized corners. The last-mentioned property deserves some discussion, since the required construction of coordinate charts, which borrows ideas from \cite{bottman-oblomkov19}, is instructive in its own right.

A boundary point in $\mathit{SS}_{d_1,\dots,(d_m,d_{m+1}),\dots,d_r}$ is given by a tree $T$ and associated screens carrying point configurations, as in \eqref{eq:glue-tree-6}. The gluing process which associates to this point a chart in the interior involves (small) gluing parameters $\lambda_e > 0$ for the finite edges of $T$, subject to constraints. Our main interest lies in those constraints, but let's first recall how to think of such gluing processes. This is made slightly more complicated in our case by the fact that the screens have different natures: the vertex $v_*$ carries a configuration of real numbers, without dividing by any group action; the mid-scale and small-mid scale vertices carry configurations which are given up to translation; and at all other scales we have configurations up to translation and rescaling. To deal with that, it is convenient to stabilize the configuration associated to the distinguished mid-scale vertex $v_*$ by adding two points $s_{\pm} = \pm 1$, thought of as belonging to their own new color, just as in \eqref{eq:stabilization}. To glue the screens together, we first choose specific representatives for those configurations which are defined only up to ambiguites. Then, given any finite edge $e$ of the tree, we take the screen associated to its source vertex, rescale the points in that configuration by $\lambda_e$, and then insert that into the target vertex by adding the real number that corresponds to the point where our configuration is being glued in (In abbreviated notation, gluing $s$ with scale $\lambda$ into a screen at point $r$ results in $r+\lambda s$.) After we have done that for all edges, we translate and rescale the resulting configuration to bring the points $s_{\pm}$ back to their original position (and then forget about those points). 

It may strike the reader that there are too many gluing parameters with respect to the codimension of the boundary strata; and indeed, the parameters are not independent, but subject to constraints. To formulate those, we can think in terms of the scales that the screens acquire after gluing. For any vertex $v$, let $\lambda_v$ be the product of the $\lambda_e^{\pm 1}$ along a path going from $v_*$ to $v$, with the sign $+1$ if the path follows the edge orientation, and $-1$ otherwise. We also need the following terminology:
\begin{equation}
\parbox{38em}{
 Given a small-mid scale vertex $v_-$, we say that a large scale vertex $v_+$ is a turning point for $v_-$ if there is a path from $v_*$ to $v_-$ which follows the orientation until it hits $v_+$, and then goes against the orientation to $v_-$.} 
 \end{equation}
Clearly, for any $v_-$ there is a unique turning point $v_+$. With that at hand, the relations are:
\begin{align}
& \label{eq:relation-1}
\text{if $v$ is a mid-scale vertex, $\lambda_v = 1$ (this is automatic for $v = v_*$);}
\\ & \label{eq:relation-2}
\text{if $v_+$ is a turning point for $v_-$ (which means $v_-$ is small-mid scale), $\lambda_{v_+}\lambda_{v_-} = 1$.}
\end{align}
It is easy to see that for a codimension one stratum, all the $\lambda_e$ therefore end up being the same.
%

\begin{example} \label{th:g1}
Consider gluing from the horizontal boundary edge at the top of Figure \ref{fig:14-gon}. Let's say that the large screen carries the configuration $(r_1,r_2)$; the mid-scale screen on the left carries a configuration $(s_1;s_2)$; the remaining screen, corresponding to $v_*$, is empty, but we add a third color and its points $s_\pm$ as explained above. The constraint
\eqref{eq:relation-1} says that the gluing parameters for both edges must be equal, so we effectively have a single parameter $\lambda$. In a first step, gluing with that parameter yields the configuration
\begin{equation} \label{eq:naive-glue}
(r_1+\lambda s_1; r_1+\lambda s_2; r_2-\lambda, r_2+\lambda).
\end{equation}
After that, we apply translation and rescaling which maps $r_2\pm\lambda $ back to $\pm 1$, that being $s \mapsto \lambda^{-1}(s-r_2)$; and (forgetting those points) end up with
\begin{equation} \label{eq:not-so-naive}
(\lambda^{-1}(r_1-r_2) + s_1; \lambda^{-1}(r_1-r_2) + s_2) \in \mathit{SS}_{(1,1)} \setminus \partial \mathit{SS}_{(1,1)}.
\end{equation}
This means that the gluing takes place in way which preserves the size of the mid-scale screens, even though that has been obscured a bit by writing it as rescaling by $\lambda$ and then its inverse.
\end{example}

\begin{example} \label{th:glue-example-0}
Consider the situation of the horizontal boundary edge at the bottom of Figure \ref{fig:14-gon}, which is also Figure \ref{fig:convergent-screens}(vi). Let's say that $\lambda_1,\lambda_2$ are the gluing parameters for the edges leading to the large-scale vertices, and $\lambda_3$ that for the remaining edge. Then, \eqref{eq:relation-2} says that $\lambda_1\lambda_2^{-1} = 1$, and \eqref{eq:relation-2} that $\lambda_{v_+}\lambda_{v_-} = \lambda_1 (\lambda_1\lambda_2^{-1}\lambda_3^{-1}) = \lambda_1^2 \lambda_2^{-1}\lambda_3^{-1} = 1$. As mentioned before, the end result is again that all gluing parameters are equal. Suppose that the large-scale screen carries $(r_1,r_2)$, the mid-scale screen carries $r$, and the small-mid scale screen carries $(s_1;s_2)$. The analogue of \eqref{eq:naive-glue} is
\begin{equation}
(r_2+\lambda^2 s_1; r_2 +\lambda^2 s_2; r_1 - \lambda, r_1 + \lambda)
\end{equation}
and that of \eqref{eq:not-so-naive} is obtained by applying $s \mapsto \lambda^{-1}(s-r_1)$, which gives
\begin{equation}
(\lambda^{-1}(r_2-r_1) + \lambda s_1; \lambda^{-1}(r_2 - r_1) + \lambda s_2).
\end{equation}
In the end, the two points at up at position $O(\lambda^{-1})$, and at distance $O(\lambda)$ from each other, which matches the description in Example \ref{th:110-convergent}(vi).
\end{example}

Allowing some of the parameters to become zero yields a partial gluing process, which extends the chart obtained by gluing to include boundary points. In order for the relations \eqref{eq:relation-1}, \eqref{eq:relation-2} to make sense in this context, one multiplies them by all $\lambda_e^{-1}$, so as to get equations between monomials with nonnegative coefficients. One can think of this completely as a limit of the previous gluing process.

%


\begin{example} \label{th:glue-example-2}
Take the example from Figure \ref{fig:glue-example-2}. After some preliminary simplifications, the relations between gluing parameters are $\lambda_1 = \lambda_2$, $\lambda_7 = \lambda_2\lambda_3$, $\lambda_4 = \lambda_5$, and more importantly 
\begin{equation} \label{eq:singularity}
\lambda_2\lambda_3 = \lambda_5\lambda_6. 
\end{equation}
Hence, this point is not a classical corner in its moduli space. After gluing, the position of the two rightmost points is of order $\lambda_2\lambda_3$, and the distance between them of order $\lambda_6^{-1}$. In the limit as all gluing parameters to go zero, $\lambda_2\lambda_3\lambda_6^{-1} = \lambda_5 \rightarrow 0$ by \eqref{eq:singularity}, as in the similar but simpler situation of Example \ref{th:110-convergent}(v).
\end{example}
\begin{figure}
\begin{centering}
\includegraphics{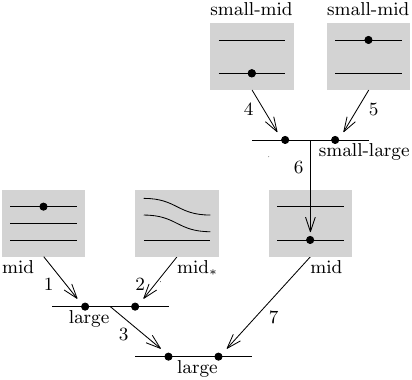}
\caption{\label{fig:glue-example-2}A generalized corner point, see Example \ref{th:glue-example-2}.}
\end{centering}
\end{figure}

It is convenient to pass from the multiplicative language of gluing parameters to the additive language of monoids. We define an abelian group $G_T$ as follows. There is one generator $g_e$ for each edge. For a vertex $v$, we define $g_v$ to be the signed sum of $g_e$ over a path from $v_*$ to $v$, with signs according to orientations. The additive relations corresponding to the ones above are: 
\begin{align} 
\label{eq:relation-3}
& \text{$g_v = 0$ for a mid-scale vertex $v$;} \\
& \label{eq:relation-4}
\text{$g_{v_+} + g_{v_-} = 0$ if $v_+$ is a turning point for $v_-$.}
\end{align}
Let $G_{T,\geq 0} \subset G_T$ be the sub-monoid generated by the $g_e$. The gluing parameters, including  the degenerate cases where some are set to zero, are elements of $\mathit{Hom}(G_T^{\geq 0}, \bR^{\geq 0})$, where $\bR^{\geq 0}$ is the multiplicative monoid. 

\begin{lemma} \label{th:corner-1}
$G_T$ is a free abelian group, whose rank is the number of vertices of $T$ which are neither mid-scale nor small-mid scale; in other words, the ``other scales'' in \eqref{eq:glue-tree-6}. 
\end{lemma}

\begin{proof}
Let $E_T$ be the set of finite edges, and $R_T$ be the set of relations. Our definition amounts to a short exact sequence 
\begin{equation} \label{eq:abelian-group-sequence}
0 \rightarrow \bZ^{R_T} \xrightarrow{\text{relations}} \bZ^{E_T} \longrightarrow G_T \rightarrow 0.
\end{equation}
Any relation has a distinguished finite edge associated to it: for \eqref{eq:relation-3}, the edge exiting $v$, and for \eqref{eq:relation-4}, the edge exiting $v_-$. Those edges are pairwise different. Given an element of $\bZ^{E_T}$, the coefficients for the distinguished edges give a splitting of the first map in \eqref{eq:abelian-group-sequence}, which implies freeness of the quotient.
 \end{proof}

\begin{lemma} \label{th:corner-2}
$G_{T,\geq 0}$ is saturated, meaning that if $g \in G_T$ satisfies $mg \in G_{T,\geq 0}$ for some $m \geq 2$, then $g \in G_{T,\geq 0}$.
\end{lemma}

%
%
%
%

\begin{proof}
For this, it is simpler to work exclusively in terms of the $g_v$, and use \eqref{eq:relation-3} to drop the mid-scale vertices. Hence, let $V_T$ be the set of all vertices which are not mid-scale. We start with $\bZ^{V_T}$, and define $G_T$ by quotienting out by \eqref{eq:relation-4}. An element
\begin{equation} \label{eq:vertex-sum}
\sum_{v \in V_T} m_v g_v \in \bZ^{V_T}
\end{equation}
is nonnegative if satisfies the following conditions. If $v$ lies above $v_*$ in our tree (meaning, the path from $v$ to the root goes through $v_*$), then $m_v \leq 0$. If $v_*$ lies above $v$, then $m_v \geq 0$. Finally, the $m_v$ increase as one goes towards the root. As before, $G_{T,\geq 0}$ is the image of the nonnegative elements in the quotient $G_T$. Here is an equivalent form of the desired statement:

{\bf Claim:} {\em Given some \eqref{eq:vertex-sum}, suppose that there are rational numbers $c_{v_-} \in [0,1]$, one for each small-mid-scale vertex $v_-$, such that 
\begin{equation} \label{eq:vertex-sum-2}
\sum_{v \in V_T} m_v g_v + \sum_{v_- \text{\rm small-mid}} c_{v_-} (g_{v_-} + g_{v_+}),
\end{equation}
satisfies the nonnegativity condition. Then, the same can be achieved with $c_{v_-} \in \{0,1\}$.}

To prove this, we take \eqref{eq:vertex-sum-2} and then gradually modify the $c_{v_-}$. Take a turning point $v_+$. There can in principle be several corresponding small-mid scale vertices $v_{-,1},\dots,v_{-,j}$. The coefficient of $v_+$ in \eqref{eq:vertex-sum-2} is then 
\begin{equation} \label{eq:plus-coefficient}
m_{v_+} + c_{v_+}, \;\; \text{where } c_{v_+} = \sum_{i=1}^j c_{v_{-,i}}.
\end{equation}
If this is an integer, we do nothing. Otherwise, we can increase (some of) the non-integer $c_{v_{-,i}}$ until the resulting expression \eqref{eq:plus-coefficient} becomes equal to the next larger integer. Let's apply this to all turning points. The outcome is that now, we have an expression \eqref{eq:vertex-sum-2} which still satisfies the nonnegativity condition, and where the coefficients of all turning points are integers. In a second pass, we change the coefficients of small-mid scale vertices again, but without affecting \eqref{eq:plus-coefficient}, to make all of them integers. The situation is, simplifying the notation, that we have non-integer $c_1,\dots,c_k \in [0,1]$ such that $c_1 + \cdots + c_k$ is an integer; and we then need to change them to be either $0$ or $1$, while preserving the sum, something that's clearly possible. Having done that, we have justified our claim.
\end{proof}

\begin{lemma} \label{th:corner-3}
$G_{T,\geq 0}$ is sharp, meaning that it contains no nontrivial pair of elements $\pm g$.
\end{lemma}

\begin{proof}
We know that $\mathit{Hom}(G_{T,\geq 0}, \bR^{\geq 0})$ recovers the space of gluing parameters, including degenerate ones. In particular, there is a distinguished point where all gluing parameters are set to zero, which is the zero map. Composing that with a homomorphism $\bZ \rightarrow G_{T, \geq 0}$ would mean that the zero map $\bZ \rightarrow \bR^{\geq 0}$ is a group homomorphism, which is nonsense.
\end{proof}

Lemmas \ref{th:corner-1}--\ref{th:corner-3} say that $G_T^{\geq 0}$ is a toric monoid (terminology as in \cite{joyce16}). For the space of gluing parameters, this is precisely what defines a generalized corner.

\section{The formal group structure\label{sec:group}}

This section carries out versions of our main constructions in an idealized context, where the technicalities of symplectic topology have been replaced by a general operadic framework (this degree of abstraction comes with its own occasional complications). The primary objects under consideration will be chain complexes which are algebras over the Fulton-MacPherson operad. Abstractly speaking, in view of \cite[Theorem 1.1]{mcclure-smith02}, this situation is not more general than the purely algebraic one mentioned in Remark \ref{th:braces}. However, that viewpoint lacks the geometric explicitness which is useful for applications to symplectic topology.

\subsection{Associahedra\label{subsec:stasheff}}
Consider the singular chain complexes of the associahedra, $C_*(S_d)$. These inherit the structure of a non-symmetric operad, using the maps induced by \eqref{eq:glue-tree-1} as well as the shuffle (Eilenberg-MacLane or Eilenberg-Zilber) product. One can inductively construct ``fundamental chains'' $[S_d] \in C_{d-2}(S_d)$ such that $[S_2] = [\mathit{point}]$, and
\begin{equation} \label{eq:fundamental-chain-1}
\partial [S_d] = \sum_{ij} (-1)^{(d-i-j)j + i}  T_{ij,*}( [S_{d-j+1}] \times [S_j] ).
\end{equation}
Here, the sum is over pairs $(i,j)$ corresponding to trees $T_{ij}$ with two vertices, of valence $j+1$ and $d-j+2$, respectively; and where the unique finite edge is the $(i+1)$-st incoming edge of the first vertex ($0 \leq i < d-j+1$), see Figure \ref{fig:ij}. We take the shuffle product (here just denoted by $\times$) of the fundamental chains $[S_{d-j+1}]$ and $[S_j]$, and then map that to $C_*(S_d)$ by the chain level map induced by \eqref{eq:glue-tree-1}, denoted here by $T_{ij,*}$. The sign takes into account the co-orientations of the boundary faces. In view of \eqref{eq:fundamental-chain-1}, $[S_d]$ has a preferred lift to a cycle for the pair $(S_d,\partial S_d)$, whose homology class is then a fundamental class in the standard sense, compatible with the orientations described in Section \ref{subsec:stasheff}.
\begin{figure}
\begin{centering}
\includegraphics{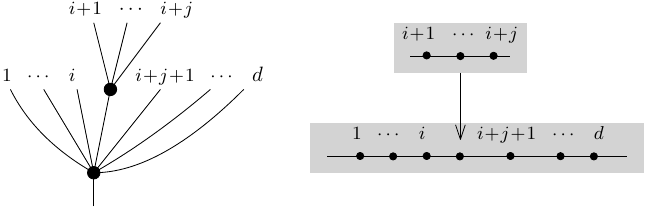}
\caption{The boundary faces of $S_d$ appearing in \eqref{eq:fundamental-chain-1}\label{fig:ij}, as trees and as strata of the compactified configuration space.}
\end{centering}
\end{figure}

Our standing convention is that chain complexes are cohomologically graded, hence we now switch to the grading-reversed version $C_{-*}(S_d)$. By an algebra over the chain level Stasheff operad, we mean a chain complex of free abelian groups $\scrA$, which comes with maps
\begin{equation} \label{eq:algebra-structure}
C_{-*}(S_d) \otimes \overbrace{\scrA \otimes \cdots \otimes \scrA}^d \longrightarrow \scrA,
\end{equation}
compatible with the composition maps induced by \eqref{eq:glue-tree-1}. Let's evaluate these maps at $[S_d] \otimes a_1 \otimes \cdots \otimes a_d$, multiply with a sign $(-1)^*$, where 
\begin{equation} \label{eq:weird-signs}
* = (d-1)|a_1| + (d-2)|a_2| + \cdots + |a_{d-1}|,
\end{equation}
and denote the outcome by $\mu^d_{\scrA}(a_1,\dots,a_d)$. These maps, together with $\mu^1_{\scrA} = -d_{\scrA}$,
make $\scrA$ into an $A_\infty$-ring. The associativity equations \eqref{eq:associativity} are a direct consequence of \eqref{eq:fundamental-chain-1}. Homological unitality is not part of this framework, hence has to be imposed as a separate property.

\begin{remark}
It is maybe appropriate to recall briefly how the signs work out. If we denote the operation \eqref{eq:algebra-structure} by $o_{\scrA}^d$, the starting point is its chain map property, which together with \eqref{eq:fundamental-chain-1} yields
\begin{equation}
\sum_{ij} (-1)^{(d+1)j + i(j+1)} o_{\scrA}^d(T_{ij,*} ([S_{d-j+1}] \times [S_j]) \otimes a_1 \otimes \cdots \otimes a_d)  + 
(\text{\it terms involving $d_{\scrA}$}) = 0.
\end{equation}
The operad property, not forgetting the Koszul signs, transforms this into
\begin{equation}
\begin{aligned}
& \sum_{ij} (-1)^{(d+1)j + i + j \maltese_i}
o_{\scrA}^{d-j+1}([S_{d-j+1}] \otimes a_1 \otimes \cdots \otimes
o_{\scrA}^j([S_j] \otimes a_{i+1} \otimes \cdots \otimes a_{i+j}) \otimes \cdots \otimes a_d) \\[-1em] &
\qquad \qquad \qquad
+ (\text{\it terms involving $d_{\scrA}$}) = 0;
\end{aligned}
\end{equation}
or in terms of the $A_\infty$-operations,
\begin{equation} \label{eq:almost-a-infinity}
(-1)^* \sum_{ij} (-1)^{\maltese_i} \mu^{d-j+1}_{\scrA}(a_1,\dots,\mu^j_{\scrA}(a_{i+1},\dots,a_{i+j}),\dots,a_d)
+ (\text{\it terms involving $d_{\scrA}$}) = 0,
\end{equation}
with $*$ as in \eqref{eq:weird-signs}. The sum in \eqref{eq:almost-a-infinity} is over $2 \leq j \leq d-1$, but only because we have omitted the differential terms, which are:
\begin{equation}
\begin{aligned}
& \sum_i (-1)^{d+\maltese_i+i} o_{\scrA}([S_d] \otimes a_1 \otimes \cdots \otimes d_{\scrA} a_{i+1} \otimes \cdots \otimes a_d)
- d_{\scrA}(o_{\scrA}^d([S_d] \otimes a_1 \otimes \cdots a_d)) \\[-.5em]
& \qquad \qquad = (-1)^* \sum_i (-1)^{\maltese_i} \mu_{\scrA}^d(a_1,\dots,\mu_{\scrA}^1(a_{i+1}),\dots,a_d) 
+ (-1)^* \mu^1_{\scrA}(\mu_{\scrA}^d(a_1,\dots,a_d)).
\end{aligned}
\end{equation}
\end{remark}

\subsection{Dependence on the fundamental chains\label{subsec:interval}}
Suppose we are given two sequences of fundamental chains $[S_d]$ and $[\tilde{S}_d]$, each of which separately satisfies \eqref{eq:fundamental-chain-1}. To relate them, we want to make further choices of fundamental chains, which have a mixed boundary property:
\begin{equation} \label{eq:f-relation}
\begin{aligned}
& f_{p,1,q} \in C_{d-2}(S_d), \quad \text{where $p,q \geq 0$ and $d = p+1+q$}, \\
& f_{p,1,0} = [S_{p+1}], \quad f_{0,1,q} = [\tilde{S}_{q+1}], \\
& \partial f_{p,1,q} = 
\sum_{ij}  (-1)^{(d-i-j)j + i} T_{ij,*} \begin{cases} 
 f_{p-j+1,1,q} \times [S_j] & \text{if $p \geq i+j$,} \\
 f_{i,1,p+q+1-i-j} \times f_{p-i,1,i+j-p-1} & \text{if $i \leq p < i+j$,} \\
 f_{p,1,q-j+1} \times [\tilde{S}_j] & \text{if $p<i$.}
 \end{cases}
\end{aligned}
\end{equation}
Graphically, one can think of \eqref{eq:f-relation} as follows. Let's mark the $(p+1)$-st leaf of our planar trees. Vertices that lie on the unique path connecting that leaf to the root correspond to factors carrying an appropriate $f$ chain, while the remaining ones always carry $[S]$ or $[\tilde{S}]$ chains, depending on whether they lie to the left or right of the path; see Figure \ref{fig:marked-tree}.
\begin{figure}
\begin{centering}
\includegraphics[scale=0.8]{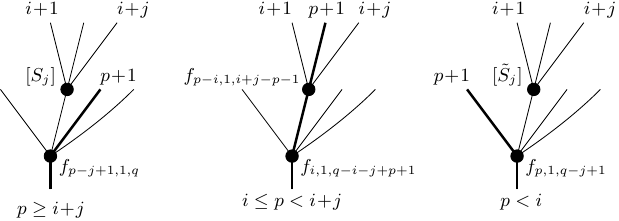}
\caption{\label{fig:marked-tree}Graphical representation of the terms in \eqref{eq:f-relation}.}
\end{centering}
\end{figure}%

Let $\mu_{\scrA}$ and $\tilde{\mu}_{\scrA}$ be the $A_\infty$-ring structures associated to $[S_d]$ and $[\tilde{S}_d]$.
In the same way, the action of $f_{p,1,q}$ gives rise to operations 
\begin{equation} \label{eq:fg-operations}
\begin{aligned} 
& 
\phi^{p,1,q}_{\scrA}: \scrA^{\otimes p+q+1} \longrightarrow \scrA[1-p-q], \;\; p+1+q \geq 2,\\
& \phi^{p,1,0}_{\scrA} = \mu_{\scrA}^{p+1}, \;\; \phi^{0,1,q}_{\scrA} = \tilde{\mu}_{\scrA}^{q+1},
\end{aligned}
\end{equation}
which, as before, we extend by setting $\phi^{0,1,0}_{\scrA} = -d_{\scrA}$. The relations inherited from \eqref{eq:f-relation} are
\begin{equation} \label{eq:f-relation-algebra}
\begin{aligned}
& \sum_{p \geq i+j} (-1)^{\maltese_i} \phi_{\scrA}^{p-j+1,1,q}(a_1,\dots,\mu_{\scrA}^j(a_{i+1},\dots,a_{i+j}),\dots,a_p;a_{p+1};a_{p+2},\dots,a_{p+q+1})  \\ &
+ \sum_{i \leq p < i+j} (-1)^{\maltese_i} \phi_{\scrA}^{i,1,p+q+1-i-j}(a_1,\dots;\phi_{\scrA}^{p-i,1,i+j-p-1}(a_{i+1},\dots,a_p;a_{p+1}; 
\\[-.5em] & \qquad \qquad \qquad \qquad \qquad \qquad \qquad \qquad \qquad
a_{p+2},\dots,a_{i+j});\dots,a_{p+q+1}) \\
& + \sum_{p<i} (-1)^{\maltese_i} \phi_{\scrA}^{p,1,q-j+1}(a_1,\dots,a_p;a_{p+1};a_{p+2},\dots,\tilde{\mu}_{\scrA}^j(a_{i+1},\dots,a_{i+j}),\dots,a_{p+q+1}) = 0.
\end{aligned}
\end{equation}

\begin{remark} \label{th:bimodule-1}
The operations \eqref{eq:fg-operations} equip the shifted space $\scrA[1]$ with the structure of an $A_\infty$-bimodule, where $\mu$ acts on the left and $\tilde{\mu}$ on the right (see e.g.\ \cite[Equation (2.5)]{seidel08}; the shift is there to match sign conventions).
\end{remark}

In a second step, we find fundamental chains
\begin{equation} \label{eq:g-relation}
\begin{aligned}
& g_{p,q} \in C_{d-1}([0,1] \times S_d), \quad \text{for $p,q > 0$ and $d = p+q$,} \\
& \partial g_{p,q} =  \{1\} \times f_{p-1,1,q} - \{0\} \times f_{p,1,q-1} \\
& + \sum_{ij} (-1)^{(d-i-j)j +i} T_{ij,*}  \begin{cases}
-g_{p-j+1,q} \times [S_j] & \text{if } p \geq i+j, \\
(-1)^{d-j} f_{i,1,q-i+p-j} \times g_{p-i,i+j-p} & \text{if } i \leq p < i+j, \\
-g_{p,q-j+1} \times [\tilde{S}_j] & \text{if } p <i.
\end{cases}
\end{aligned}
\end{equation}
When compared to \eqref{eq:fundamental-chain-1} and \eqref{eq:f-relation}, the spaces involved have acquired an additional $[0,1]$ factor: hence, we should really write $\mathit{id}_{[0,1]} \times T_{ij,*}$. The graphical representation involves drawing a dividing line between the first $p$ and last $q$ leaves of our trees. In the first two summands in \eqref{eq:g-relation}, we remove that dividing line and instead mark the leaves that are on either side of it, leading to the appearance of two $f$ terms. For the remaining summands, vertices to the left or right of the dividing line carry $[S]$ resp.\ $\tilde{[S]}$ chains (Figure \ref{fig:marked-tree-2}). If the dividing line ends at the top vertex (which is the middle case in both \eqref{eq:g-relation} and Figure \ref{fig:marked-tree-2}), the finite edge of the tree becomes the marked edge of the bottom vertex, which explains how that vertex carries an $f$ term.
\begin{figure}
\begin{centering}
\includegraphics[scale=0.8]{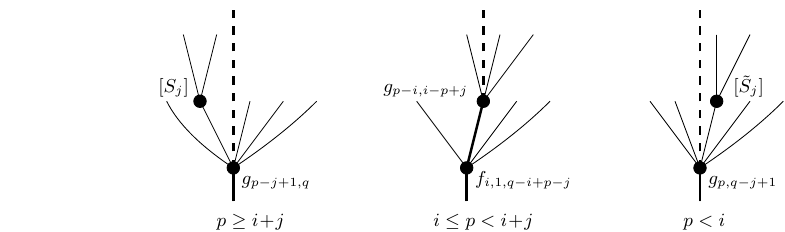}
\caption{\label{fig:marked-tree-2}Graphical representation of the $\sum_{ij}$ in \eqref{eq:g-relation}.}
\end{centering}
\end{figure}%

Let's take the image of \eqref{eq:g-relation} under projection to $S_d$. Its action under the operad structure, with additional signs inserted as in \eqref{eq:weird-signs}, gives operations
\begin{equation}
\label{eq:fg-operations-2}
\psi^{p,q}_{\scrA}: \scrA^{\otimes p+q} \longrightarrow \scrA[1-p-q], \;\; p,q > 0, \\
\end{equation}
which we complement by setting
$\psi^{1,0}_{\scrA} = -\mathit{id}_{\scrA}$, $\psi^{0,1}_{\scrA} = \mathit{id}_{\scrA}$.
These satisfy
\begin{equation} \label{eq:g-relation-algebra}
\begin{aligned}
& \sum_{p \geq i+j} 
(-1)^{\maltese_i}
\psi^{p-j+1,q}_{\scrA}(a_1,\dots,\mu_{\scrA}^j(a_{i+1},\dots,a_{i+j}),\dots,a_p;a_{p+1},\dots,a_{p+q}) 
\\
& - \sum_{i < p < i+j} \phi_{\scrA}^{i,1,p+q-i+j}(a_1,\dots,a_i;\psi_{\scrA}^{p-i,i+j-p}(a_{i+1},\dots,a_p;
\\[-1em] & \qquad \qquad\qquad\qquad\qquad\qquad\qquad\qquad
a_{p+1},\dots,a_{i+j});a_{i+j+1},\dots,a_{p+q}) \\[.5em]
& + \sum_{p \leq i} (-1)^{\maltese_i} \psi_{\scrA}^{p,q-j+1}(a_1,\dots,a_p;a_{p+1},\dots,\tilde{\mu}_{\scrA}^j(a_{i+1},\dots,a_{i+j}),\dots,a_{p+q}) = 0. 
\end{aligned}
\end{equation}
%
Note that \eqref{eq:g-relation-algebra} contains terms which correspond to the boundary faces $\{0\} \times S_d$ and $\{1\} \times S_d$:
\begin{equation} \label{eq:01-boundary-terms}
\begin{aligned}
&
- \phi_{\scrA}^{p-1,1,q}(a_1,\dots,a_{p-1}; \psi_{\scrA}^{1,0}(a_p); a_{p+1},\dots, a_{p+q}) -
\phi_{\scrA}^{p,1,q-1}(a_1,\dots,a_p; \psi_{\scrA}^{0,1}(a_{p+1}); \\
&
\qquad \qquad \qquad \qquad a_{p+2},\dots,a_{p+q})
\\ &
= \phi_{\scrA}^{p-1,1,q}(a_1,\dots,a_{p-1};a_p;a_{p+1},\dots,a_{p+q}) - \phi_{\scrA}^{p,1,q-1}(a_1,\dots,a_p;a_{p+1};a_{p+2},\dots,a_{p+q}). 
\end{aligned}
\end{equation}

\begin{example}
The simplest instance of \eqref{eq:g-relation-algebra}, bearing in mind the conventions for $\phi^{0,1,0}$, $\phi^{1,1,0}$ and $\phi^{0,1,1}$, is:
\begin{equation}
\psi^{1,1}_{\scrA}(\mu_{\scrA}^1(a_1);a_2) + (-1)^{\|a_1\|} \psi^{1,1}_{\scrA}(a_1;\mu_{\scrA}^1(a_2)) =
\mu^1_{\scrA}(\psi_{\scrA}^{1,1}(a_1;a_2))  + \mu_{\scrA}^2(a_1,a_2) - \tilde\mu_{\scrA}^2(a_1,a_2).
\end{equation}
This says that $(-1)^{|a_1|}\psi^{1,1}(a_1,a_2)$ is a chain homotopy relating the two versions of multiplication.
\end{example}

\begin{remark} \label{th:bimodule-2}
Following up on our last observation, one can give the following interpretation of \eqref{eq:g-relation-algebra}.
Recall from Remark \ref{th:bimodule-1} that the operations $\phi$ equip $\scrA$ (here, we undo the shift for simplicity) with an $A_\infty$-bimodule structure. By construction, this is isomorphic to $(\scrA,\mu_{\scrA})$ as a left module over itself, and to $(\scrA,\tilde{\mu}_{\scrA})$ as a right module. Correspondingly, one has two bimodule maps
\begin{equation} \label{eq:two-bimodule}
\xymatrix{
(\scrA,\mu_{\scrA}) \otimes_{\bZ} (\scrA,\tilde{\mu}_\scrA) \ar@<1ex>[rr]^-{\rho_{\scrA}} \ar@<-1ex>[rr]_-{\tilde{\rho}_{\scrA}} && (\scrA,\phi_{\scrA}),
}
\end{equation}
given by
\begin{equation} \label{eq:two-bimodule-2}
\begin{aligned}
&
\rho_{\scrA}^{p-1,1,q-1}(a_1,\dots ;a_p \otimes a_{p+1}; \dots,a_{p+q}) = \pm \phi^{p,1,q-1}_{\scrA}(a_1,\dots;a_{p+1};\dots,a_{p+q}),
\\ & 
\tilde{\rho}_{\scrA}^{p-1,1,q-1}(a_1,\dots;a_p \otimes a_{p+1};\dots,a_{p+q}) = \pm \phi^{p-1,1,q}_{\scrA}(a_1,\dots;a_p;\dots,a_{p+q}).
\end{aligned}
\end{equation}
In these terms, \eqref{eq:g-relation-algebra} says that $\psi$ provides a homotopy between $\rho$ and $\tilde{\rho}$.

It is worth noting that homological unitality, when it holds, can be used to simplify the picture. Namely, suppose that $\mu_{\scrA}$ and $\tilde{\mu}_{\scrA}$ are both homologically unital, with a priori different units $e_{\scrA}$ and $\tilde{e}_{\scrA}$. Then, a bimodule map as in \eqref{eq:two-bimodule} is determined up to homotopy by the image of $[e_{\scrA} \otimes \tilde{e}_{\scrA}]$ in $H^0(\scrA)$. In our situation, these two classes are 
\begin{equation}
\begin{aligned}
& [\mu^2_{\scrA}(e_{\scrA},\tilde{e}_{\scrA})] = [\tilde{e}_{\scrA}], \\
& [\tilde\mu^2_{\scrA}(e_{\scrA},\tilde{e}_{\scrA})] = [e_{\scrA}],
\end{aligned}
\end{equation}
so the existence of a homotopy $\psi$ just amounts to saying that the two units are, after all, cohomologous. Similarly, the different choices of $\psi$ form an affine space over $H^{-1}(\scrA)$. 
\end{remark}


Let's define an $A_\infty$-ring structure on 
\begin{equation}
\scrH = \scrA \otimes \scrI = \scrA u \oplus \scrA \tilde{u} \oplus \scrA v,
\end{equation}
where $\scrI$ is the noncommutative interval \eqref{eq:nc-interval}, as follows. The differential $\mu^1_{\scrH}$ is as in \eqref{eq:tensor-product}. The nonzero higher $A_\infty$-operations are
\begin{equation} \label{eq:tensor-structure}
\begin{aligned}
& 
\mu^d_{\scrH}(a_1 \otimes u, \dots, a_d \otimes u) = \mu^d_{\scrA}(a_1,\dots,a_d) \otimes u,
\\
& 
\mu^{p+q}_{\scrH}(a_1 \otimes u, \dots, a_p \otimes u, a_{p+1} \otimes \tilde{u}, \dots, a_{p+q} \otimes \tilde{u}) \\ & \quad = (-1)^{\maltese_{p+q}} \psi^{p,q}_{\scrA}(a_1,\dots,a_p;a_{p+1},\dots,a_{p+q}) \otimes v \quad \text{for $p,q>0$}, \\
&
\mu^d_{\scrH}(a_1 \otimes \tilde{u},\dots,a_d \otimes \tilde{u}) = \tilde{\mu}^d_{\scrA}(a_1,\dots,a_d) \otimes \tilde{u}, 
\\
&
\mu^{p+q+1}_{\scrH}(a_1 \otimes u, \dots, a_p \otimes u, a_{p+1} \otimes v, a_{p+2} \otimes \tilde{u}, \dots, a_{p+q+1} \otimes \tilde{u}) \\ & \quad = (-1)^{\maltese_{p+q+1} - \maltese_{p+1}} \phi^{p,1,q}(a_1,\dots,a_p; a_{p+1}; a_{p+2},\dots, a_{p+q+1}) \otimes v.
\end{aligned}
\end{equation}
(This generalizes the previous \eqref{eq:tensor-product}, which corresponds to the diagonal $A_\infty$-bimodule structure and vanishing $\psi$). The $A_\infty$-associativity relations follow directly from \eqref{eq:f-relation-algebra}, \eqref{eq:g-relation-algebra}.

\begin{remark}
As a check on the signs, consider the associativity relation for $(a_1 \otimes u, \dots, a_p \otimes u, a_{p+1} \otimes \tilde{u},\dots,a_{p+q} \otimes \tilde{u})$, and more specifically the $v$-component of that relation. This turns out to be exactly \eqref{eq:g-relation-algebra} multiplied by $(-1)^{\maltese_{p+q}+1}$. The crucial terms, compare \eqref{eq:01-boundary-terms}, are
\begin{equation}
\begin{aligned}
& (-1)^{\maltese_{p-1}} \mu_{\scrH}^{p+q}(a_1 \otimes u, \dots, \text{\it $v$-component of } \mu^1_{\scrH}(a_p \otimes u), a_{p+1} \otimes \tilde{u}, \dots) \\ & \qquad =
(-1)^{\maltese_p+1} \mu_{\scrH}^{p+q}(a_1 \otimes u, \dots, a_p \otimes v, a_{p+1} \otimes \tilde{u}, \dots) 
\\ & \qquad =
(-1)^{\maltese_{p+q}+1} \phi_{\scrA}^{p-1,1,q}(a_1,\dots,a_{p-1};a_p;a_{p+1},\dots,a_{p+q}) \otimes v
\end{aligned}
\end{equation}
and
\begin{equation}
\begin{aligned}
& (-1)^{\maltese_p} \mu_{\scrH}^{p+q}(a_1 \otimes u, \dots, a_p \otimes u, \text{\it $v$-component of } \mu^1_{\scrH}(a_{p+1} \otimes \tilde{u}), \dots) \\ & \qquad =
(-1)^{\maltese_{p+1}} \mu_{\scrH}^{p+q}(a_1 \otimes u, \dots, a_p \otimes u, a_{p+1} \otimes v, \dots) \\ & \qquad =
(-1)^{\maltese_{p+q}} \phi_{\scrA}^{p,1,q-1}(a_1,\dots,a_p;a_{p+1};a_{p+2},\dots,a_{p+q}) \otimes v.
\end{aligned}
\end{equation}
\end{remark}

By construction, the projections  \eqref{eq:two-projections} are $A_\infty$-homomorphisms from $\mu_{\scrH}$ to $\mu_{\scrA}$ and $\tilde{\mu}_{\scrA}$, respectively, and also chain homotopy equivalences. 
By taking a homotopy inverse (Lemma \ref{th:homotopy-inverse}) of one projection, and composing with the other projection, we get an $A_\infty$-homomorphism
\begin{equation} \label{eq:a-tilde-a}
(\scrA,\mu_{\scrA}) \longrightarrow (\scrA,\tilde{\mu}_{\scrA}),
\end{equation}
whose linear part is homotopic to the identity (one can achieve that it's exactly the identity). For a completely satisfactory statement, one would need to prove that \eqref{eq:a-tilde-a} is itself independent of the choice of \eqref{eq:fg-operations}, \eqref{eq:fg-operations-2} up to homotopy of $A_\infty$-homomorphisms; and also, that the composition of two maps \eqref{eq:a-tilde-a} is again a map of the same type, up to homotopy. This would use higher analogues of $\scrI$. For the sake of brevity, we will not carry it out here. 

\subsection{Fulton-MacPherson spaces and colored multiplihedra\label{subsec:apply-fm}}
One defines the structure of an algebra over $C_{-*}(\mathit{FM}_d)$ on a chain complex $\scrC$ by maps analogous to \eqref{eq:algebra-structure}, with the additional stipulation of $\mathit{Sym}_d$-invariance. On the cohomology level, $H^*(\scrC)$ becomes a Gerstenhaber algebra. The chain level structure is a classical topic in algebraic topology ($E_2$-algebras; see e.g.\ \cite{may72, cohen76, mcclure-smith02, sinha13}). For our purpose, only part of that structure is relevant (that part, maybe surprisingly, does not include the fundamental chains $[\mathit{FM}_d] \in C_{2d-3}(\mathit{FM}_d)$ and the resulting $L_\infty$-structure; in fact, the chains relevant for us have dimension $\leq d$).
\begin{figure}
\begin{centering}
\includegraphics[scale=0.85]{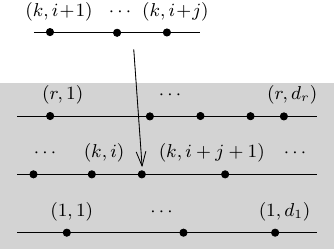}
\caption{\label{fig:codim-1-trees}A summand in the first sum in \eqref{eq:boundary-mww}.}
\end{centering}
\end{figure}
\begin{figure}
\begin{centering}
\includegraphics{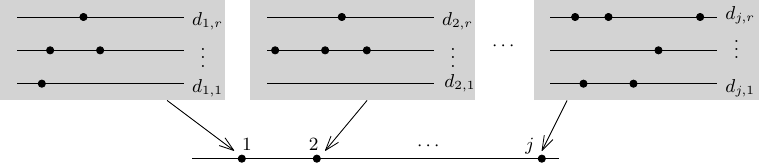}
\caption{\label{fig:codim-1-trees-2}A summand in the second sum in \eqref{eq:boundary-mww}.}
\end{centering}
\end{figure}

First of all, having chosen fundamental chains $[S_d]$ for the Stasheff associahedra, one can map them to $C_{d-2}(\mathit{FM}_d)$ via \eqref{eq:s-into-fm}, and their action turns $\scrC$ into an $A_\infty$-ring. As before, one has to require homological unitality separately. Next, choose fundamental chains $[\mathit{MWW}_{d_1,\dots,d_r}] \in C_{d-1}(\mathit{MWW}_{d_1,\dots,d_r})$ for the colored multiplihedra, which satisfy the analogue of \eqref{eq:fundamental-chain-1}. It is worth while writing this down:
\begin{equation} \label{eq:boundary-mww}
\begin{aligned}
& \partial [\mathit{MWW}_{d_1,\dots,d_r}] = \\ &
\sum_{ijk} (-1)^{(d_k-i-j+d_{k+1}+\cdots+d_r)j + (d_1+\cdots+d_{k-1}+i + 1)} T_{ijk,*}( [\mathit{MWW}_{d_1,\dots,d_k-j+1,\dots,d_r}] \times [S_j] ) 
\\ &
\!\!\! + \sum_{\text{partitions}} \!\!\! (-1)^\diamondsuit\, T_{d_{1,1},\dots,d_{j,r},*}(
 [S_j] \times [\mathit{MWW}_{d_{1,1},\dots,d_{1,r}}] \times \cdots \times [\mathit{MWW}_{d_{j,1},\dots,d_{j,r}}]).
\end{aligned}
\end{equation}
The second sum is over all $j \geq 2$ and partitions $d_1 = d_{1,1} + \cdots + d_{j,1}$, $\dots$, $d_r = d_{1,r} + \cdots + d_{j,r}$, such that $d_{i,1} + \cdots + d_{i,r} > 0$ for each $i = 1,\dots,j$. The sign there is given by
\begin{equation}
\begin{aligned}
& \diamondsuit = \sum_{\substack{i_1 < i_2 \\ k_1>k_2}} d_{i_1,k_1} d_{i_2,k_2} 
\\[-1em] & \qquad \qquad \quad
+ (j-1)(d_{1,1}+\cdots+d_{1,r}-1) + (j-2)(d_{2,1} + \cdots + d_{2,r}-1) + \cdots
\end{aligned}
\end{equation}

\begin{example}
Examples of the degenerate configurations corresponding to the terms in \eqref{eq:boundary-mww} are shown in Figures \ref{fig:codim-1-trees} and \ref{fig:codim-1-trees-2} (the trees $T_{ijk}$ and $T_{d_{1,1},\dots,d_{j,r}}$ can be inferred from looking at those, so we will not define them explicitly). Figures \ref{fig:12-gon} and \ref{fig:8-gon} illustrate the orientation issues: in both of them, the actual moduli space has the standard orientation of the plane, and the arrows show the orientations of the boundary strata arising from \eqref{eq:glue-tree-3}.
\end{example}

Choose maps \eqref{eq:mww-into-fm}, take the images of the fundamental chains under those maps, and let them act on $\scrC$. The outcome are operations
\begin{equation} \label{eq:beta-operations}
\beta^{d_1,\dots,d_r}_{\scrC}: \scrC^{\otimes d} \longrightarrow \scrC[1-d].
\end{equation}
In their definition, we insert signs as in \eqref{eq:weird-signs}; for $\beta_{\scrC}^{d_1,\dots,d_r}(c_{1,1},\dots,c_{1,d_1};\dots;c_{r,1},\dots,c_{r,d_r})$ this means $(-1)^\ast$ with
\begin{equation}
\begin{aligned}
\ast = & \,(d_1 + \cdots + d_r -1) |c_{1,1}| + (d_1+ \cdots + d_r - 2)|c_{1,2}| + \cdots + (d_1 + \cdots + d_{r-1}) |c_{1,d_1}| \\
& + (d_1 + \cdots + d_{r-1}-1) |c_{2,1}| + \cdots
\end{aligned}
\end{equation}
For $\mathit{MWW}_{0,\dots,0,1,0,\dots,0} = \mathit{point}$, where there is no corresponding Fulton-MacPherson space, we artificially set
\begin{equation} \label{eq:010}
\beta^{0,\dots,0,1,0,\dots,0}_{\scrC} = \mathit{id}_{\scrC}. 
\end{equation}
As a consequence of \eqref{eq:boundary-mww}, 
\begin{equation} \label{eq:property-of-beta}
\begin{aligned}
& \sum_{ijk} (-1)^{\maltese_{k,i}} \beta_{\scrC}^{d_1,\dots,d_k-j+1,\dots,d_r}\big(c_{1,1},\dots,c_{1,d_1};  \dots; 
c_{k,1},\dots, \mu_{\scrC}^j(c_{k,i+1},\dots,c_{k,i+j}),
\\[-1em] & \qquad \qquad \qquad \qquad \dots,c_{k,d_k} ; \dots; c_{r,1},\dots,c_{r,d_r}\big) \\[.5em]
& = 
\sum_{\text{partitions}} (-1)^\heartsuit \, \mu_{\scrC}^j\big(\beta_{\scrC}^{d_{1,1},\dots,d_{1,r}}(c_{1,1},\dots,c_{1,d_{1,1}};\dots;c_{r,1},\dots,c_{r,d_{1,r}}),  \dots,
\\[-.5em] & \qquad \qquad \qquad \qquad
\beta_{\scrC}^{d_{j,1},\dots,d_{j,r}}(c_{1,d_1-d_{j,1}+1},\dots,c_{1,d_1};\dots;c_{d_r-d_{j,r}+1},\dots,c_{r,d_r})\big).
\end{aligned}
\end{equation}
Here, the sums are over indexing sets as in \eqref{eq:boundary-mww}, except that we now additionally allow the differential $\mu^1_{\scrC} = d_{\scrC}$. Recall that by construction, the map \eqref{eq:mww-into-fm} forgets factors of $\mathit{MWW}_{0,\dots,0,1,0,\dots,0} = \mathit{point}$. Algebraically, this corresponds to the places where \eqref{eq:010} appears in \eqref{eq:property-of-beta}. The $\maltese$ symbol is the sum of reduced degrees of all $c$ which precede the $\mu$; and $\heartsuit$ yields the Koszul sign that corresponds to permuting the $c_{k,i}$ from their original order into the order in which they appear on the right hand side of \eqref{eq:property-of-beta}, but using reduced degrees $\|c_{k,i}\|$.

\begin{remark}
The operations \eqref{eq:beta-operations} constitute an $A_\infty$-multihomomorphism with $r$ entries (the single-object version of an $A_\infty$-multifunctor; see \cite[Definition 8.8]{bespalov}, or closer to our context, the discussion of the $r =2$ case in \cite[Section 4.5]{bottman-wehrheim18}) $\scrC \times \cdots \times \scrC\rightarrow \scrC$. What we will study later on amounts to the action of those $A_\infty$-multihomomorphisms on Maurer-Cartan elements. One can argue that the homomorphisms themselves should be at the center of attention (meaning that Proposition \ref{th:semi-associativity} should be understood as a consequence of a composition property of the $A_\infty$-multihomomorphisms up to homotopy; and similarly that Corollary \ref{th:r-1-trivial} should be true because for $r = 1$, one gets an $A_\infty$-endomorphism of $\scrC$ which is homotopy equivalent to the identity); in the interest of keeping the discussion concrete, we have chosen not to take that route.
\end{remark}

\begin{example} \label{th:beta-2}
In view of \eqref{eq:010} and \eqref{eq:property-of-beta}, $\beta^2_{\scrC}$ satisfies
\begin{equation}
\begin{aligned} &
\mu^1_{\scrC}(\beta_{\scrC}^{2}(c_1,c_2)) - \beta_{\scrC}^{2}(\mu^1_{\scrC}(c_1),c_2) - (-1)^{\|c_1\|}
\beta_{\scrC}^{2}(c_1,\mu^1_{\scrC}(c_2)) \\ & \quad = \beta_{\scrC}^1(\mu^2_{\scrC}(c_1,c_2)) -
\mu^2_{\scrC}(\beta^1_{\scrC}(c_1), \beta^1_{\scrC}(c_2)) = 0,
\end{aligned}
\end{equation}
which means that $(-1)^{|c_1|} \beta_{\scrC}^2(c_1,c_2)$ is a chain map of degree $-1$. Geometrically, the reason is that the image of the fundamental chain under $\mathit{MWW}_2 \rightarrow \mathit{FM}_2$ is a one-cycle. However, this cycle is supported at a single point of $\mathit{FM}_2 \iso S^1$, hence is necessarily nullhomologous. This implies that $\beta_{\scrC}^2$ is chain homotopic to zero.
\end{example}

\begin{example} \label{th:beta-1-1}
The first substantially nontrivial case is $\beta^{1,1}_{\scrC}$, which satisfies
\begin{equation} \label{eq:beta-1-1}
\begin{aligned} &
\mu^1_{\scrC}\beta_{\scrC}^{1,1}(c_1;c_2) - \beta_{\scrC}^{1,1}(\mu^1_{\scrC}(c_1);c_2) -
(-1)^{\|c_1\|}\beta_{\scrC}^{1,1}(c_1;\mu^1_{\scrC}(c_2)) \\ & \qquad = -\mu^2_{\scrC}(c_1,c_2) -
(-1)^{\|c_1\| \cdot \|c_2\|} \mu^2_{\scrC}(c_2,c_1).
\end{aligned}
\end{equation}
In more conventional terminology, $(-1)^{|c_1|} \beta_{\scrC}^{1,1}(c_1;c_2)$ is the $\circ$ operation which shows homotopy commutativity of the product on $H^*(\scrC)$.
\end{example}

\begin{definition} \label{th:def-multiproduct}
Fix an adic ring $N$ (Definition \ref{th:adic-ring}). Given $\gamma_1,\dots,\gamma_r \in \scrC^1 \hat\otimes N$, define
\begin{equation} \label{eq:multiproduct}
\Pi^r_{\scrC}(\gamma_1,\dots,\gamma_r) = 
\sum_{\substack{d_1,\dots,d_r\geq 0 \\ d_1+\cdots+d_r > 0}} \beta_{\scrC}^{d_1,\dots,d_r}(\overbrace{\gamma_1,\dots,\gamma_1}^{d_1};\dots;\overbrace{\gamma_r,\dots,\gamma_r}^{d_r}).
\end{equation}
\end{definition}

Suppose that we have $(\gamma_1,\dots,\gamma_r)$ as well as, for some $k$, another element $\tilde{\gamma}_k$ (the basic case is $\gamma_k = \tilde{\gamma}_k$, but for some applications, the freedom to choose a general $\tilde{\gamma}_k$ is important). Then, define a linear endomorphism of $\scrC \hat\otimes N$ by a generalization of \eqref{eq:mc-differential}:
\begin{equation}
\begin{aligned}
P_{\scrC}^{r,k}(\xi) & = \sum_{d_1,\dots,d_r} \sum_{p+q+1=d_r} \beta_{\scrC}^{d_1,\dots,d_k,\dots,d_r}(\overbrace{\gamma_1,\dots,\gamma_1}^{d_1};\dots;\overbrace{\gamma_k,\dots,\gamma_k}^p,\xi,
\\[-.5em] & \qquad \qquad \qquad \qquad \qquad \qquad \qquad \qquad \qquad
\underbrace{\tilde\gamma_k,\dots,\tilde\gamma_k}_q;\dots;\underbrace{\gamma_r,\dots,\gamma_r}_{d_r}).
\end{aligned}
\end{equation}
The definitions, taking \eqref{eq:010} into account, have the following immediate consequences:
\begin{align}
& \Pi^r_{\scrC}(\gamma_1,\dots,\gamma_r) = \gamma_1 + \cdots + \gamma_r \; \text{mod } N^2, \\
& 
\Pi^r_{\scrC}(\gamma_1,\dots,\gamma_k + \xi,\dots,\gamma_r) = \Pi_{\scrC}(x) + 
P^{r,k}_{\scrC}(\xi) + (\text{\it order $\geq 2$ terms in $\xi$}), \label{eq:derivative-of-pi}
\\ \label{eq:derivative-of-pi-2}
& \xi \in \scrC \hat\otimes N^m \;\; \Longrightarrow P^{r,k}_{\scrC}(\xi) = \xi \;\text{mod } N^{m+1}, \\
& \Pi^r_{\scrC}(\gamma_1,\dots,\gamma_r) - 
\Pi^r_{\scrC}(\gamma_1,\dots,\tilde{\gamma}_k,\dots,\gamma_r) = P^{r,k}_{\scrC}(\gamma_k - \tilde{\gamma}_k). \label{eq:forgot}
\end{align}
In \eqref{eq:derivative-of-pi}, the endomorphism $P^{r,k}_{\scrC}$ is with respect to $\tilde\gamma_k = \gamma_k$. The two subsequent equations, in contrast, use a general $\tilde{\gamma}_k$.

\begin{lemma} \label{th:multiproduct}
If $\gamma_1,\dots,\gamma_r$ are Maurer-Cartan elements \eqref{eq:mc}, then so is $\gamma = \Pi^r_{\scrC}(\gamma_1,\dots,\gamma_r)$. Moreover, the equivalence class of $\gamma$ depends only on those of $\gamma_1,\dots,\gamma_r$.
\end{lemma}

\begin{proof}
From \eqref{eq:property-of-beta} one gets
\begin{equation} \label{eq:inherit-mc}
\sum_d \mu_{\scrC}^d(\gamma,\dots,\gamma) = \sum_k P^{r,k}_{\scrC}\big( \sum_j \mu^j_{\scrC}(\gamma_k,\dots,\gamma_k) \big).
\end{equation}
Here, the $P$ operations are defined using $\tilde{\gamma}_k = \gamma_k$. This shows that the Maurer-Cartan property is preserved. Similarly,  suppose that for some $1 \leq k \leq r$, we have another Maurer-Cartan solution $\tilde{\gamma}_k$. Then, for the associated $\gamma$ and $\tilde{\gamma} = \Pi^r_{\scrC}(\gamma_1,\dots,\tilde{\gamma}_k,\dots,\gamma_r)$,
\begin{equation} \label{eq:inherit-equivalence}
\sum_{p,q} \mu_{\scrC}^{p+q+1}(\overbrace{\gamma,\dots,\gamma}^p, P^{r,k}_{\scrC}(x),
\overbrace{\tilde{\gamma},\dots,\tilde{\gamma}}^q) = P^{r,k}_{\scrC}\big(\sum_{p,q} \mu_{\scrC}^{p+q+1}(\overbrace{\gamma_k,\dots,\gamma_k}^p,x,\overbrace{\tilde{\gamma}_k,\dots,\tilde{\gamma}_k}^q)\big).
\end{equation}
In particular, if we have an element $h_k$ which provides an equivalence between $\gamma_k$ and $\tilde{\gamma}_k$, then $h = P^{r,k}_{\scrC}(h_k)$ provides an equivalence between $\gamma$ and $\tilde{\gamma}$, by \eqref{eq:forgot}.
\end{proof}

We want to mention a few elementary statements which, taken together, stand in a converse relation of sorts to Lemma \ref{th:multiproduct}. 

\begin{lemma} \label{th:inverse-1}
Suppose that we have $\gamma_1,\dots,\gamma_{k-1},\gamma_{k+1},\dots,\gamma_r \in \scrC^1 \hat\otimes N$. Then, for each $\gamma \in \scrC^1 \hat\otimes N$ there is exactly one $\gamma_k$ such that $\Pi^r_{\scrC}(\gamma_1,\dots,\gamma_r) = \gamma$.
\end{lemma}

\begin{proof}
By \eqref{eq:derivative-of-pi} and \eqref{eq:derivative-of-pi-2}, if $\xi \in \scrC^1 \hat \otimes N^m$, then $\Pi^r_{\scrC}(\gamma_1,\dots,\gamma_k + \xi,\dots,\gamma_r) = \Pi^r_{\scrC}(\gamma_1,\dots,\gamma_r) + \xi$ mod $N^{m+1}$. This allows one to solve for $\gamma_k$ order by order, and to show uniqueness of the solution in the same way.
\end{proof}

\begin{lemma} \label{th:inverse-2}
Suppose that we have $\gamma_1,\dots,\gamma_r \in \scrC^1 \hat\otimes N$. If all but $\gamma_k$ are Maurer-Cartan elements, and $\gamma = \Pi^r_{\scrC}(\gamma_1,\dots,\gamma_r)$ is Maurer-Cartan as well, then $\gamma_k$ must also be Maurer-Cartan.
\end{lemma}

\begin{proof}
From \eqref{eq:inherit-mc} and the assumptions, one sees that $P^{r,k}_{\scrC}(\sum_j \mu_{\scrC}^j(\gamma_k,\dots,\gamma_k)) = 0$. On the other hand, by \eqref{eq:derivative-of-pi-2}, $P^{r,k}_{\scrC}$ is clearly invertible. 
\end{proof}

\begin{lemma} \label{th:inverse-3}
Given Maurer-Cartan elements $\gamma_1,\dots,\gamma_{k-1},\gamma_{k+1},\dots,\gamma_r$ and $\gamma$, there is a unique Maurer-Cartan element $\gamma_k$ such that $\Pi^r_{\scrC}(\gamma_1,\dots,\gamma_r) = \gamma$.
\end{lemma}

\begin{proof}
This is simply a combination of Lemmas \ref{th:inverse-1} and \ref{th:inverse-2}.
\end{proof}

\begin{lemma} \label{th:inverse-4}
Suppose that we have Maurer-Cartan elements $\gamma_1,\dots,\gamma_r$ and $\tilde{\gamma}_k$, for some $1 \leq k \leq r$. If $\gamma = \Pi^r_{\scrC}(\gamma_1,\dots,\gamma_r)$ and $\tilde{\gamma} = \Pi^r_{\scrC}(\gamma_1\,\dots,\gamma_{k-1},\tilde{\gamma}_k,\gamma_{k+1},\dots,\gamma_r)$ are equivalent, then so are $\gamma_k$ and $\tilde{\gamma}_k$.
\end{lemma}

\begin{proof}
This is a consequence of \eqref{eq:inherit-equivalence} and the fact that $P^{r,k}_{\scrC}$ is an automorphism.
\end{proof}

Take the case $r = 1$ of \eqref{eq:multiproduct}. Then \eqref{eq:property-of-beta} says that $(\beta^1_{\scrC} = \mathit{id}, \beta^2_{\scrC},\dots)$ form an $A_\infty$-homomorphism from $\scrC$ to itself (which is not surprising, since the underlying spaces $\mathit{MWW}_d$ are the multiplihedra). The corresponding operation \eqref{eq:multiproduct} is just the action of the $A_\infty$-homomorphism on Maurer-Cartan elements. One can show that this $A_\infty$-homomorphism is always homotopic to the identity, and hence $\Pi^1_{\scrC}(\gamma)$ is equivalent to $\gamma$. (The first piece of the statement about the $A_\infty$-homomorphism is  Example \ref{th:beta-2}, but we won't explain the rest here; as for the action on Maurer-Cartan elements, we will give an indirect argument in Corollary \ref{th:r-1-trivial}). Therefore, that case is essentially trivial. With that in mind, the first nontrivial instance of \eqref{eq:multiproduct} is $r = 2$, which we will denote by
\begin{equation} \label{eq:bullet}
\gamma_1 \bullet \gamma_2 = \Pi^2_{\scrC}(\gamma_1,\gamma_2).
\end{equation}
It will eventually turn out that the $r>2$ cases can be reduced to an $(r-1)$-fold application of this product (Corollary \ref{th:r-at-least-3}), and hence are in a sense redundant. 

\subsection{Well-definedness\label{subsec:hideous}}
Proving that \eqref{eq:multiproduct} is well-defined involves comparing different choices of the underlying $A_\infty$-structures $\mu_{\scrC}$, as well as of the operations $\beta_{\scrC}$. Since the details are lengthy, and the outcome overall not surprising, we will provide only a sketch of the argument.

One can generalize the construction of the operations \eqref{eq:beta-operations}, by allowing the use of different versions of the $A_\infty$-structure (in fact, a different version for each color of input, and another one for the output). Concretely, suppose that we have $(r+1)$ choices of fundamental chains for the Stasheff associahedra, with their associated $A_\infty$-structures $\mu_{\scrC,0}, \dots, \mu_{\scrC,r}$. By choosing fundamental chains on the colored multiplihedra which satisfy an appopriately modified version of \eqref{eq:boundary-mww}, we get generalized operations \eqref{eq:beta-operations}, which then lead to a map
\begin{equation} \label{eq:multiproduct-2}
\Pi^r_{\scrC}: \mathit{MC}(\scrC,\mu_{\scrC,1};N) \times \cdots \times \mathit{MC}(\scrC,\mu_{\scrC,r};N) \longrightarrow \mathit{MC}(\scrC,\mu_{\scrC,0};N).
\end{equation}
For instance, let's look at $r = 1$. Then, what we get from the modified operations \eqref{eq:beta-operations} is an $A_\infty$-homomorphism between two choices of $A_\infty$-structures on $\scrC$, whose linear part is the identity. That gives an alternative proof of the uniqueness result from Section \ref{subsec:interval} (in spite of that, it made sense for us to include the original proof; the reason will become clear shortly).

In \eqref{eq:multiproduct-2}, we want to understand the effect of simultaneously changing $\mu_{\scrC,0}$, one of the other $\mu_{\scrC,k+1}$, $k \geq 0$, and correspondingly also \eqref{eq:multiproduct-2}. Namely, suppose that we have alternative versions $\tilde{\mu}_{\scrC,0}$ and $\tilde{\mu}_{\scrC,k+1}$. Alongside \eqref{eq:multiproduct-2}, we also have another operation which uses the alternative $A_\infty$-structures, as well as different choices of functions \eqref{eq:imaginary-part-is-tau} and fundamental chains on the $\mathit{MWW}$ spaces. Let's denote that version by $\tilde{\Pi}^r_{\scrC}$. The construction from Section \ref{subsec:interval} yields $A_\infty$-structures $\mu_{\scrH,0}$ and $\mu_{\scrH,k+1}$, where $\scrH = \scrC \otimes \scrI$. One can then construct a new operation $\Pi_{\scrH}^{k,1,l}$, where $k+1+l = r$, which fits into the following diagram, with vertical arrows induced by \eqref{eq:two-projections}:
\begin{equation} \label{eq:interpolate}
\xymatrix{
\txt{
$\mathit{MC}(\scrC,\mu_{\scrC,1};N) \times \cdots \times \mathit{MC}(\scrC,\mu_{\scrC,k+1};N)$ \\ 
$ \times \cdots \times \mathit{MC}(\scrC,\mu_{\scrC,r};N)$ 
}
\ar[r]^-{\Pi^r_{\scrC}}
& \mathit{MC}(\scrC,\mu_{\scrC,0};N) \\
\ar[u]^-{\iso} \ar[d]_-{\iso}
\txt{$\mathit{MC}(\scrC,\mu_{\scrC,1};N) \times \cdots \times \mathit{MC}(\scrH, \mu_{\scrH,k+1};N)$ \\ $\times \cdots \times \mathit{MC}(\scrC,\mu_{\scrC,r};N)$
} \ar[r]^-{\Pi_{\scrH}^{k,1,l}}  
& \mathit{MC}(\scrH,\mu_{\scrH,0};N) \ar[u]_-{\iso} \ar[d]^-{\iso} \\
\txt{
$\mathit{MC}(\scrC,\mu_{\scrC,1};N) \times \cdots \times \mathit{MC}(\scrC,\tilde\mu_{\scrC,k+1};N)$
\\ $\times \cdots \times \mathit{MC}(\scrC,\mu_{\scrC,r};N)$} \ar[r]_-{\tilde\Pi^r_{\scrC}}
& \mathit{MC}(\scrC,\tilde\mu_{\scrC,0};N).
}
\end{equation}

Rather than giving the general construction of \eqref{eq:interpolate}, we will only look the $r = 1$ case. This is not terribly interesting in itself, but contains the main complications of the general situation, while allowing us to couch the discussion in more familiar terms. The setup for $r = 1$ is that we are given the following data: 
\begin{itemize} \itemsep.5em
\item four $A_\infty$-structures on $\scrC$, namely $\mu_{\scrC,k}$ and $\tilde{\mu}_{\scrC,k}$ for $k = 0,1$; 
\item two $A_\infty$-structures on $\scrH$, namely $\mu_{\scrH,k}$ for $k = 0,1$, which are constructed with the aim of interpolating between $\mu_{\scrC,k}$ and $\tilde{\mu}_{\scrC,k}$. Their definition, following \eqref{eq:tensor-structure}, involves operations $\phi_{\scrC,k}$ and $\psi_{\scrC,k}$ as in \eqref{eq:fg-operations}, \eqref{eq:fg-operations-2}. 
\item Finally, we have two versions of \eqref{eq:beta-operations}, which are $A_\infty$-homo\-mor\-phisms $\beta_{\scrC}: (\scrC,\mu_{\scrC,1}) \rightarrow (\scrC,\mu_{\scrC,0})$ and $\tilde\beta_{\scrC}: (\scrC,\tilde{\mu}_{\scrC,1}) \rightarrow (\scrC,\tilde{\mu}_{\scrC,0})$. 
\end{itemize}
The aim is to define an $A_\infty$-homomorphism $\beta_{\scrH}$, again having the identity as its linear term, which fits into a commutative diagram
\begin{equation} \label{eq:interpolating-functor}
\xymatrix{
(\scrC,\mu_{\scrC,1}) \ar[rr]^-{\beta_{\scrC}} && (\scrC,\tilde\mu_{\scrC,0}) 
\\
\ar[u] \ar[d]
(\scrH, \mu_{\scrH,1}) \ar[rr]^-{\beta_{\scrH}}
&& (\scrH, \mu_{\scrH,0}) \ar[u] \ar[d]
\\
(\scrC,\tilde\mu_{\scrC,1}) \ar[rr]^-{\tilde\beta_{\scrC}} && (\scrC,\tilde\mu_{\scrC,0}). 
}
\end{equation}
The corresponding special case of \eqref{eq:interpolate} is then defined through the action of $\beta_{\scrH}$ on Maurer-Cartan elements. The definition of $\beta_{\scrH}$ involves two kinds of operations:
\begin{align} \label{eq:sigma-map}
&
\begin{aligned}
& \sigma^{p,1,q}_{\scrC}: \scrC^{\otimes p+1+q} \longrightarrow \scrC[-p-q], \;\; p+q \geq 0, \\
& \qquad \sigma^{p,1,0}_{\scrC} = \beta_{\scrC}^{p+1}, \;\; \sigma^{0,1,q}_{\scrC} = \tilde\beta_{\scrC}^{q+1}, 
\; \text{and in particular $\sigma^{0,1,0}_{\scrC} = \mathit{id}_{\scrC}$},
\end{aligned}
\\
&
\label{eq:tau-map}
\tau^{p,q}_{\scrC}: \scrC^{\otimes p+q} \longrightarrow \scrC[-p-q-1], \;\; p,q > 0, 
\end{align}
These enter into a formula parallel to \eqref{eq:tensor-structure}: the nonzero terms of our $A_\infty$-homomorphism are
\begin{equation} \label{eq:mixed-functor}
\begin{aligned}
& \beta_{\scrH}^d(c_1 \otimes u, \dots, c_d \otimes u) = 
\beta_{\scrC}^d(c_1,\dots,c_d) \otimes u 
, \\
& \beta_{\scrH}^{p+q}(c_1 \otimes u, \dots, c_p \otimes u, c_{p+1} \otimes \tilde{u}, \dots, c_{p+q} \otimes \tilde{u}) \\ & \qquad \qquad = (-1)^{\maltese_{p+q}} \tau^{p,q}_{\scrC}(c_1,\dots,c_p;c_{p+1},\dots,c_{p+q}) \otimes v 
\quad\text{for $p,q>0$}, \\
& \beta_{\scrH}^d(c_1 \otimes \tilde{u},\dots, c_d \otimes \tilde{u}) = 
\tilde\beta_{\scrC}^d(c_1,\dots,c_d) \otimes \tilde{u} 
, \\
& \beta_{\scrH}^{p+q+1}(c_1 \otimes u,\dots, c_p \otimes u, c_{p+1} \otimes v, 
c_{p+2} \otimes \tilde{u}, \dots, c_{p+q+1} \otimes \tilde{u}) \\ & \qquad \qquad
 = (-1)^{\maltese_{p+q+1}-\maltese_{p+1}} \sigma^{p,1,q}_{\scrC}(c_1, \dots, c_p; c_{p+1}; c_{p+2},\dots, c_{p+q+1}) \otimes v.
\end{aligned}
\end{equation}

The fact that \eqref{eq:mixed-functor} satisfies the $A_\infty$-homomorphism relations reduces to certain properties of \eqref{eq:sigma-map}, \eqref{eq:tau-map}. Those for \eqref{eq:sigma-map} are
\begin{equation}
\begin{aligned}
& \sum_{p \geq i+j} (-1)^{\maltese_i} \sigma_{\scrC}^{p-j+1,1,q}(c_1, \dots, c_i, \mu_{\scrC,1}^j(c_{i+1},\dots,c_{i+j}), \dots;c_{p+1}; \dots,c_{p+q+1}) 
\\ & + \sum_{i \leq p \leq i+j} (-1)^{\maltese_i} \sigma_{\scrC}^{i,1,p+q-i+j}(c_1,\dots,c_i; \phi_{\scrC,1}^{p-i,1,i+j-p-1}(c_{i+1},\dots ;c_{p+1}; \dots,c_{i+j}); \\[-.5em]
& \qquad \qquad \qquad \qquad \qquad \qquad \qquad \qquad c_{i+j+1},\dots,c_{p+q+1}) 
\\ & + \;\; \sum_{p<i} (-1)^{\maltese_i} \sigma_{\scrC}^{p,1,q-j+1}(c_1,\dots ; c_{p+1}; \dots, \tilde\mu^j_{\scrC,1}(c_{i+1},\dots,c_{i+j}), \dots, c_{p+q+1})
\\ &
= \sum_{\text{partitions}} \phi^{s,1,t}_{\scrC,0}\big(
\beta_{\scrC}^{d_1}(c_1,\dots,c_{d_1}), \dots, 
\beta_{\scrC}^{d_s}(c_{d_1+\cdots+d_{s-1}+1}, \dots, c_{d_1+\cdots+d_s}); 
\\[-.5em] & \qquad \qquad \qquad
\sigma_{\scrC}^{p-d_1-\cdots-d_s,1,d_1+\cdots+d_{s+1}-p-1}(c_{d_1 + \cdots+ d_s+1}, \dots ; c_{p+1}; \dots, c_{d_1+\cdots + d_{s+1}}); \\
& \qquad \qquad \qquad\cdots,
 \tilde\beta_{\scrC}^{d_{s+1+t}}(c_{p+q+2-d_{s+t+1}},\dots,c_{p+q+1}) \big).
\end{aligned}
\end{equation}
On the right hand side, the sum is over all $(s,t)$ and partitions $d_1 + \cdots + d_{s+1+t} = p+1+q$ such that $d_1 + \cdots +  d_s  < p+1$ and $d_1 + \cdots + d_{s+1} \geq p+1$. In spite of the apparently larger number of terms which appear, this is formally parallel to the $A_\infty$-homomorphism equation, and in fact the nontrivial operations \eqref{eq:sigma-map} are obtained from a choice of fundamental chains on $\mathit{MWW}_{p+q+1}$, as well as functions \eqref{eq:imaginary-part-is-tau}. The trick is that the boundary behaviour of these data is partially determined by the choices underlying $\beta_{\scrC}$ and $\tilde{\beta}_{\scrC}$, just as in our previous discussion of \eqref{eq:f-relation}. The relations for \eqref{eq:tau-map} are:
\begin{equation} \label{eq:omg}
\begin{aligned}
& -\sum_{p \geq i+j} (-1)^{\maltese_i} \tau_{\scrC}^{p-j+1,q}(c_1, \dots, c_i, \mu_{\scrC,1}^j(c_{i+1},\dots,c_{i+j}), \dots ;c_{p+1}, \dots,c_{p+q}) 
\\ & + \sum_{i < p < i+j} \sigma_{\scrC}^{i,1,p+q-i-j}(c_1,\dots,c_i; \psi_{\scrC,1}^{p-i,i+j-p}(c_{i+1},\dots ;c_{p+1}, \dots,c_{i+j}); \\[-.5em]
& \qquad \qquad \qquad \qquad \qquad \qquad \qquad \qquad c_{i+j+1},\dots,c_{p+q}) \\
& - \sum_{p \leq i} (-1)^{\maltese_i} \tau_{\scrC}^{p,q-j+1}(c_1,\dots , c_p; \dots, \tilde\mu^j_{\scrC,1}(c_{i+1},\dots,c_{i+j}), \dots, c_{p+q})
\\
& =  \sum_{\text{partitions}} (-1)^{\maltese_{d_1+\cdots+d_{s}}} \phi_{\scrC,0}^{s,1,t}\big(\beta_{\scrC}^{d_1}(c_1,\dots,c_{d_1}), \dots, 
\beta_{\scrC}^{d_s}(c_{d_1+\cdots+d_{s-1}+1}, \dots, c_{d_1+\cdots+d_s}); 
\\[-.5em] & \qquad \qquad \qquad
\tau_{\scrC}^{p-d_1-\cdots-d_s,d_1+\cdots+d_{s+1}-p}(c_{d_1 + \cdots+ d_s+1}, \dots, c_p; c_{p+1}, \dots, c_{d_1+\cdots + d_{s+1}}); \\
& \qquad \qquad \qquad\dots,
 \tilde\beta_{\scrC}^{d_{s+1+t}}(c_{p+q+2-d_{s+t+1}},\dots,c_{p+q}) \big)
\\ &
+ \sum_{\text{partitions}} \psi^{s,t}_{\scrC,0}\big(
\beta_{\scrC}^{d_1}(c_1,\dots,c_{d_1}), \dots, 
\beta_{\scrC}^{d_s}(c_{d_1+\cdots+d_{s-1}+1}, \dots, c_{d_1+\cdots+d_s}); 
\\[-.5em] & \qquad \qquad \qquad\dots,
 \tilde\beta_{\scrC}^{d_{s+t}}(c_{p+q+2-d_{s+t+1}},\dots,c_{p+q}) \big).
\end{aligned}
\end{equation}
Combinatorially, the difference between the two terms on the right hand side of \eqref{eq:omg} is where the dividing semicolon between the first $p$ and last $q$ inputs comes to lie: in the first case, we require that $d_1 + \cdots + d_s < p < d_1 + \cdots + d_{s+1}$, so that semicolon is inside one of the innermost operations, which becomes a $\tau$ operation; in the second case, we require that $d_1 + \cdots + d_s = p$, so that semicolon separates the two kinds of inputs for the $\psi$ operation. Topologically one realizes \eqref{eq:tau-map} by choosing suitable fundamental chains on $[0,1] \times \mathit{MWW}_{p+q}$, and analogues of \eqref{eq:imaginary-part-is-tau} on that product space.  The second sum in \eqref{eq:omg} contains terms which correspond to the boundary faces $\{0,1\} \times \mathit{MWW}_{p+q}$, just as in \eqref{eq:01-boundary-terms}:
\begin{equation}
\begin{aligned}
& \sigma_{\scrC}^{p-1,1,q}(c_1,\dots,c_{p-1};\psi^{1,0}_{\scrC,1}(c_p);c_{p+1},\dots,c_{p+q+1}) 
\\ & \qquad + \sigma_{\scrC}^{p,1,q-1}(c_1,\dots,c_p;\psi^{0,1}_{\scrC,1}(c_p);c_{p+1},\dots,c_{p+q+1}) \\
& = -\sigma_{\scrC}^{p-1,1,q}(c_1,\dots,c_{p-1};c_p;c_{p+1},\dots,c_{p+q})  + \sigma_{\scrC}^{p,1,q-1}(c_1,\dots,c_p;c_{p+1};c_{p+2},\dots,c_{p+q}). \\
\end{aligned}
\end{equation}

\begin{example}
The first new operation $\tau^{1,1}_{\scrC}$ satisfies (bearing in mind that all versions of the $A_\infty$-structure on $\scrC$ share the same differential $-d = \mu^1_{\scrC} = \phi^{0,1,0}_{\scrC,0} = \phi^{0,1,0}_{\scrC,1}$)
\begin{equation}
\begin{aligned}
& \mu^1_{\scrC}(\tau_{\scrC}^{1,1}(c_1;c_2)) + \tau_{\scrC}^{1,1}(\mu^1_{\scrC}(c_1);c_2) + (-1)^{\|c_1\|} \tau_{\scrC}^{1,1}(c_1;\mu^1_{\scrC}(c_2)) \\
& \quad = \beta_{\scrC}^2(c_1,c_2) - \tilde{\beta}^2_{\scrC}(c_1,c_2) + \psi_{\scrC,1}^{1,1}(c_1;c_2) - \psi_{\scrC,0}^{1,1}(c_1;c_2).
\end{aligned}
\end{equation}
This is exactly what's required for the first nontrivial $A_\infty$-functor equation on $\scrH$: one has
\begin{equation}
\begin{aligned}
& \mu^1_{\scrH,0}\big( \beta_{\scrH}^2(c_1 \otimes u, c_2 \otimes \tilde{u}) \big)
+ \mu^2_{\scrH,0}\big( \beta_{\scrH}^1(c_1 \otimes u), \beta_{\scrH}^1(c_2 \otimes \tilde{u})\big)
\\ & = (-1)^{\|c_1\|+\|c_2\|} \big( \mu^1_{\scrC}(\tau_{\scrC}^{1,1}(c_1;c_2)) + 
  \psi^{1,1}_{\scrC,0}(c_1;c_2) \big) \otimes v, 
\end{aligned}
\end{equation}
while
\begin{equation}
\begin{aligned}
& \beta_{\scrH}^2( \mu^1_{\scrH,1}(c_1 \otimes u), c_2 \otimes \tilde{u})
+ (-1)^{\|c_1\|} \beta_{\scrH}^2( c_1 \otimes u, \mu^1_{\scrH,1}(c_2 \otimes \tilde{u}))  
\\ & \qquad
+ \beta_{\scrH}^1( \mu^2_{\scrH,1}(c_1 \otimes u, c_2 \otimes \tilde{u})) 
\\
& = (-1)^{\|c_1\|+\|c_2\|} \big( -
\tau^{1,1}_{\scrC}(\mu^1_\scrC (c_1); c_2)
+ \beta^2_{\scrC}(c_1,c_2)
\\ & \qquad
- (-1)^{\|c_1\|} \tau_{\scrC}^{1,1}(c_1; \mu^1_{\scrC}(c_2))
- \tilde{\beta}^2_{\scrC}(c_1,c_2)
 + \psi^{1,1}_{\scrC,1}(c_1;c_2)
\big) \otimes v.
\end{aligned}
\end{equation}
\end{example}

To conclude our discussion, let's return to the general context (arbitrary $r$), and note that then, repeated application of \eqref{eq:interpolate} allows one to change all the choices involved. We record the outcome:

\begin{corollary} \label{th:hideous}
Suppose that we have two different choices of fundamental chains on the associahedra and colored multiplihedra, as well as of functions \eqref{eq:imaginary-part-is-tau}, leading to two version of the $A_\infty$-structure and operations \eqref{eq:multiproduct}. These fit into a commutative diagram
\begin{equation}
\xymatrix{
\ar[d]_-{\iso}
\mathit{MC}(\scrC,\mu_{\scrC};N)^r \ar[rr]^-{\Pi^r_{\scrC}} && \ar[d]^-{\iso} \mathit{MC}(\scrC,\mu_{\scrC};N)
\\
\mathit{MC}(\scrC,\tilde\mu_{\scrC};N)^r \ar[rr]^-{\tilde\Pi^r_{\scrC}} && \mathit{MC}(\scrC,\tilde\mu_{\scrC};N)
}
\end{equation}
Here, we have related our $A_\infty$-structures using functors as in \eqref{eq:a-tilde-a}, and the vertical arrows are the induced maps on Maurer-Cartan elements.
\end{corollary}

\subsection{The $p$-th power operation\label{subsec:multiproduct-p}}
When defining \eqref{eq:beta-operations}, suppose now that we choose our functions \eqref{eq:tau-functions} so that they satisfy \eqref{eq:drop-color-2}. For the fundamental chains, we may also assume that they are chosen to be compatible with the identifications \eqref{eq:drop-color}. In algebraic terms, the outcome is a cancellation property, which allows one to forget colors that do not carry any marked points:
\begin{equation} \label{eq:cancel-0}
\beta_{\scrC}^{d_1,\dots,d_{k-1},0,d_{k+1},\dots,d_r} = \beta_{\scrC}^{d_1,\dots,d_{k-1},d_{k+1},\dots,d_r}.
\end{equation}
Assuming that such a choice has been adopted, we have:

\begin{lemma} \label{th:p-power-beta}
Take a prime number $p$, and the coefficient ring $N = q\bF_p[[q]]/q^{p+1}$. Then, for $\gamma = qc + O(q^2)$, one has
\begin{equation}
\Pi^p_{\scrC}(\gamma,\dots,\gamma) = \beta^{1,\dots,1}_{\scrC}(c;\dots;c) q^p.
\end{equation}
\end{lemma}

\begin{proof}
This is elementary, along the same lines as in Lemma \ref{th:xi-algebra}. Applying \eqref{eq:cancel-0} allows one to rewrite \eqref{eq:multiproduct} as
\begin{equation} \label{eq:r-product-with-insertion}
\Pi^r_{\scrC}(\gamma,\dots,\gamma) = \sum_{1 \leq k \leq r} {\textstyle \left(\begin{matrix} r \\ k \end{matrix}\right)} \sum_{\substack{ d_1,\dots,d_k > 0 }} \beta_{\scrC}^{d_1,\dots,d_k}(\overbrace{\gamma,\dots,\gamma}^{d_1};\dots;\overbrace{\gamma,\dots,\gamma}^{d_k}),
\end{equation}
where the combinatorial factor reflects the possibilities of inserting $0$ superscripts into each $\beta$ operation. Suppose that our coefficient ring is $N = q\bF_p[[q]]$, and set $r = p$. Then \eqref{eq:r-product-with-insertion} becomes
\begin{equation}
\Pi^p_{\scrC}(\gamma,\dots,\gamma) = \sum_{d_1,\dots,d_p>0} 
\beta_{\scrC}^{d_1,\dots,d_k}(\overbrace{\gamma,\dots,\gamma}^{d_1};\dots;\overbrace{\gamma,\dots,\gamma}^{d_p}).
\end{equation}
Truncating mod $q^{p+1}$ leaves $\beta_{\scrC}^{1,\dots,1}(\gamma;\cdots;\gamma) = \beta_{\scrC}^{1,\dots,1}(c;\dots;c) q^p$ as the only nonzero term.
\end{proof}

\subsection{Deligne-Mumford spaces and commutativity\label{subsec:commutativity}}
Let's consider the question of commutativity of the product \eqref{eq:bullet}. Concretely, this hypothetical commutativity would mean that there is an $h \in \scrC^0 \hat\otimes N$ such that 
\begin{equation} \label{eq:commutative-h}
\sum_{p,q} \mu^{p+q+1}_{\scrC}(\overbrace{\gamma_1 \bullet \gamma_2,\dots,\gamma_1 \bullet \gamma_2}^p, h, \overbrace{\gamma_2 \bullet \gamma_1,\dots, \gamma_2 \bullet \gamma_1}^q) = \gamma_1 \bullet \gamma_2 - \gamma_2 \bullet \gamma_1.
\end{equation}
Let's suppose, to simplify the exposition, that the coefficient ring is $N = q\bZ[[q]]$. Moreover, we choose to define the operations \eqref{eq:beta-operations} as in Section \ref{subsec:multiproduct-p}, so that \eqref{eq:cancel-0} holds. That entails some convenient (but not essential, of course) cancellations in our formulae. Given two Maurer-Cartan elements 
\begin{equation} \label{eq:k-maurer-cartan}
\gamma_k = qc_k + O(q^2)\quad (k = 1,2),
\end{equation}
with leading terms $c_k$ which are cocycles in $\scrC^1$, we have by definition
\begin{equation} \label{eq:gamma12-21}
\gamma_1 \bullet \gamma_2 - \gamma_2 \bullet \gamma_1 = q^2\big(
\beta_{\scrC}^{1,1}(c_1;c_2)
- \beta_{\scrC}^{1,1}(c_2;c_1) 
\big)  + O(q^3);
\end{equation}
In writing down this formula, we have exploited the fact that, due to our choices, $\beta^{2,0}_{\scrC}(c_k,c_k) = \beta_{\scrC}^{0,2}(c_k,c_k)$.
It follows from Example \ref{th:beta-1-1} that 
\begin{equation} \label{eq:beta-difference}
(c_1,c_2) \in \scrC \longmapsto (-1)^{\|c_1\|} \beta_{\scrC}^{1,1}(c_1;c_2) - (-1)^{\|c_1\|\,|c_2|}\beta_{\scrC}^{1,1}(c_2;c_1)
\end{equation}
 is a chain map of degree $-1$. On cohomology, it defines the Lie bracket $[\cdot,\cdot]: H^*(\scrC) \otimes H^*(\scrC) \rightarrow H^{*-1}(\scrC)$ which is part of the Gerstenhaber algebra structure. Geometrically, $\beta^{1,1}$ arises from a one-dimensional chain in $\mathit{FM}_2$ whose boundary points are exchanged by the $\bZ/2$-action. The sum of this chain and its image under the nontrivial element of $\bZ/2$ is a cycle, which generates $H_1(\mathit{FM}_2) \iso H_1(S^1) = \bZ$. If we similarly write $h = qb + O(q^2)$, then \eqref{eq:commutative-h} taken modulo $q^3$ says that $b$ is a cocycle, and that
\begin{equation} \label{eq:commutative-mod-q3}
\mu^2_{\scrC}(c_1+c_2,b) + \mu^2_{\scrC}(b,c_1+c_2) + (\text{\it coboundaries}) = 
\beta_{\scrC}^{1,1}(c_1;c_2) - \beta_{\scrC}^{1,1}(c_2;c_1).
\end{equation}
By \eqref{eq:beta-1-1}, the left hand side of \eqref{eq:commutative-mod-q3} is nullhomologous. Hence, for \eqref{eq:commutative-mod-q3} to be satisfied, the Lie bracket of $[c_1]$ and $[c_2]$ must be zero, which means that commutativity does not hold in this level of generality.

We now switch from Fulton-MacPherson to Deligne-Mumford spaces. One could define the structure of an algebra over $C_{-*}(\mathit{DM}_d)$ on a chain complex $\scrC$ in the same way as before. However, that notion is not well-behaved. For instance, the action of $\mathit{DM}_2 = \mathit{point}$ would yield a strictly commutative product on $\scrC$. The underlying problem is that the $\mathit{Sym}_d$-action on $\mathit{DM}_d$ is not free (from an algebraic viewpoint, $C_{-*}(\mathit{DM}_d)$ is not a projective $\bZ[\mathit{Sym}_d]$-module). There is a simple workaround, by ``freeing up'' the action. Namely, let $(E_d)$ be an $E_\infty$-operad, which means that the spaces $E_d$ are contractible and freely acted on by $\mathit{Sym}_d$. Let's adopt a concrete choice, namely, the analogue of Fulton-MacPherson space for point configurations in $\bR^\infty$. Then 
\begin{equation} \label{eq:freeing-up}
\EuScript{DM}_d = \mathit{DM}_d \times E_d
\end{equation}
is again an operad, which is homotopy equivalent to $\mathit{DM}_d$ but carries a free action of $\mathit{Sym}_d$. The maps \eqref{eq:fm-to-dm} admit lifts
\begin{equation} \label{eq:fm-to-dm-2}
\mathit{FM}_d \longrightarrow \EuScript{DM}_d,
\end{equation}
which are compatible with the operad structure, including the action of $\mathit{Sym}_d$. As an existence statement, this is a consequence of the properties of $E_d$; but for our specific choice, such lifts can be defined explicitly by taking \eqref{eq:fm-to-dm} together with the natural inclusion $\mathit{FM}_d \rightarrow E_d$.

Assume from now on that $\scrC$ carries the structure of an operad over $C_{-*}(\EuScript{DM}_d)$, and hence inherits one over the Fulton-MacPherson operad by \eqref{eq:fm-to-dm-2}. Take the one-cycle in $\mathit{FM}_2$ underlying \eqref{eq:gamma12-21} and map it to (the contractible space) $\EuScript{DM}_2$. Choosing a bounding cochain (which is itself unique up to coboundaries) yields a nullhomotopy
\begin{equation} \label{eq:first-kappa}
\begin{aligned}
& \kappa_{\scrC}^{1,1}: \scrC^{\otimes 2} \longrightarrow \scrC[-2], \\
& \mu^1_{\scrC}(\kappa^{1,1}_{\scrC}(c_1;c_2)) + \kappa^{1,1}_{\scrC}(\mu^1_{\scrC}(c_1);c_2) + (-1)^{\|c_1\|}\kappa^{1,1}_{\scrC}(c_1;\mu^1_{\scrC}(c_2)) 
\\ & \qquad = \beta_{\scrC}^{1,1}(c_1;c_2) - (-1)^{\|c_1\|\,\|c_2\|}\beta_{\scrC}^{1,1}(c_2;c_1).
\end{aligned}
\end{equation}
As a consequence, if we set
\begin{equation} \label{eq:guess-0}
h = q^2\kappa_{\scrC}^{1,1}(c_1;c_2) \in \scrC^0 \hat\otimes q^2\bZ[[q]], 
\end{equation}
then \eqref{eq:commutative-h} is satisfied modulo $q^3$ (on the left hand side, only the $p = q = 0$ term matters at this point). Hence, if we reduce coefficients to $q\bZ[[q]]/q^3$, then \eqref{eq:bullet} is commutative. Nothing we have said so far is in any way surprising: the vanishing of the Lie bracket in the case where the operations come from Deligne-Mumford space is a well-known fact (if one uses the framed little disc operad as an intermediate object, it follows from vanishing of the BV operator).

Let's push our investigation a little further. As special cases of \eqref{eq:property-of-beta}, we have 
\begin{equation} \label{eq:21-relation}
\begin{aligned} &
\mu^1_{\scrC}(\beta_{\scrC}^{2,1}(c_1,c_2;c_3)) -  \beta_{\scrC}^{2,1}(\mu_{\scrC}^1(c_1),c_2;c_3) 
\\ & \qquad
- (-1)^{\|c_1\|} \beta_{\scrC}^{2,1}(c_1,\mu^1_{\scrC}(c_2);c_3) - (-1)^{\|c_1\|+\|c_2\|} \beta_{\scrC}^{2,1}(c_1,c_2;\mu^1_{\scrC}(c_3))
\\ & =
 \beta_{\scrC}^{1,1}(\mu_{\scrC}^2(c_1,c_2);c_3) -
(-1)^{\|c_2\|\,\|c_3\|} \mu_{\scrC}^2(\beta_{\scrC}^{1,1}(c_1;c_3),c_2) - \mu_{\scrC}^2(c_1,\beta_{\scrC}^{1,1}(c_2;c_3)) 
\\ & \qquad - \mu^2_{\scrC}(\beta_{\scrC}^2(c_1,c_2),c_3) - (-1)^{\|c_3\| (\|c_1\|+\|c_2\|)}\mu^2_{\scrC}(c_3,\beta_{\scrC}^2(c_1,c_2)) 
\\ & \qquad  - \mu_{\scrC}^3(c_1,c_2,c_3) - (-1)^{\|c_2\|\,\|c_3\|}\mu_{\scrC}^3(c_1,c_3,c_2) - (-1)^{(\|c_1\|+\|c_2\|)\|c_3\|} \mu_{\scrC}^3(c_3,c_1,c_2);
\end{aligned}
\end{equation}
respectively 
\begin{equation} \label{eq:12-relation}
\begin{aligned} &
\mu^1_{\scrC}( \beta_{\scrC}^{1,2}(c_3;c_1,c_2) )
- \beta_{\scrC}^{1,2}(\mu_{\scrC}^1(c_3);c_1,c_2) 
\\ & \qquad
- (-1)^{\|c_3\|} \beta_{\scrC}^{1,2}(c_3;\mu_{\scrC}^1(c_1),c_2) 
- (-1)^{\|c_3\|+\|c_1\|} \beta_{\scrC}^{1,2}(c_3;c_1,\mu_{\scrC}^1(c_2))
\\ & =
(-1)^{\|c_3\|} \beta_{\scrC}^{1,1}(c_3;\mu^2_{\scrC}(c_1,c_2)) 
- \mu_{\scrC}^2(\beta_{\scrC}^{1,1}(c_3;c_1),c_2) 
- (-1)^{\|c_1\|\,\|c_3\|} \mu_{\scrC}^2(c_1,\beta_{\scrC}^{1,1}(c_3;c_2))
\\ & \qquad 
- \mu^2_{\scrC}(c_3,\beta_{\scrC}^2(c_1,c_2)) 
- (-1)^{\|c_3\|(\|c_1\|+\|c_2\|)}\mu^2_{\scrC}(\beta_{\scrC}^2(c_1,c_2),c_3) 
\\ &  \qquad
- \mu_{\scrC}^3(c_3,c_1,c_2) - (-1)^{\|c_3\|\,\|c_1\|} \mu_{\scrC}^3(c_1,c_3,c_2)
- (-1)^{\|c_3\|(\|c_1\|+\|c_2\|)} \mu_{\scrC}^3(c_1,c_2,c_3).
\end{aligned}
\end{equation}
Therefore, if we consider the map $K_{\scrC}^{2,1}: \scrC^{\otimes 3} \rightarrow \scrC[-2]$ given by
\begin{equation} \label{eq:k21}
\begin{aligned}
& K_{\scrC}^{2,1}(c_1,c_2;c_3) = \beta^{2,1}_{\scrC}(c_1,c_2;c_3) - (-1)^{\|c_3\| (\|c_1\|+\|c_2\|)}\beta_{\scrC}^{1,2}(c_3;c_1,c_2)
\\ & \quad - \kappa_{\scrC}^{1,1}(\mu^2_{\scrC}(c_1,c_2);c_3)  - (-1)^{\|c_2\|\,\|c_3\|}
\mu^2_{\scrC}(\kappa_{\scrC}^{1,1}(c_1;c_3),c_2) - 
(-1)^{\|c_1\|}
\mu^2_{\scrC}(c_1,\kappa_{\scrC}^{1,1}(c_2;c_3)),
\end{aligned}
\end{equation}
then that satisfies
\begin{equation}
\begin{aligned}
&
\mu^1_{\scrC}(K_{\scrC}^{2,1}(c_1,c_2;c_3)) - K_{\scrC}^{2,1}(\mu^1_{\scrC}(c_1),c_2;c_3) \\ & \qquad -
(-1)^{\|c_1\|} K_{\scrC}^{2,1}(c_1,\mu^1_{\scrC}(c_2);c_3) - (-1)^{\|c_1\|+\|c_2\|} 
K_{\scrC}^{2,1}(c_1,c_2;\mu^1_{\scrC}(c_3)) = 0,
\end{aligned}
\end{equation}
which means that $(-1)^{|c_2|} K_{\scrC}^{2,1}(c_1,c_2;c_3)$ is a chain map of degree $-2$. The same observation applies to 
\begin{equation} \label{eq:k12}
\begin{aligned}
& K_{\scrC}^{1,2}(c_1;c_2,c_3) = \beta^{1,2}_{\scrC}(c_1;c_2,c_3) - (-1)^{\|c_1\|(\|c_2\|+\|c_3\|}\beta_{\scrC}^{2,1}(c_2,c_3;c_1) 
\\ & \quad
-  (-1)^{\|c_1\|} \kappa_{\scrC}^{1,1}(c_1;\mu^2_{\scrC}(c_2,c_3)) -
\mu^2_{\scrC}(\kappa_{\scrC}^{1,1}(c_1;c_2),c_3) - (-1)^{\|c_1\|\,\|c_2\|} \mu^2_{\scrC}(c_2,\kappa_{\scrC}^{1,1}(c_1;c_3)).
\end{aligned}
\end{equation}

We now return to the original situation \eqref{eq:k-maurer-cartan}. Suppose that the cocycles $K_{\scrC}^{2,1}(c_1,c_1;c_2)$ and $K_{\scrC}^{1,2}(c_1;c_2,c_2)$ are trivial in $H^*(\scrC)$, and that we have chosen bounding cochains for them:
\begin{equation}
\begin{aligned}
& K_{\scrC}^{2,1}(c_1,c_1;c_2) = \mu^1_{\scrC}(b^{2,1}), \\
& K_{\scrC}^{1,2}(c_1;c_2,c_2) = \mu^1_{\scrC}(b^{1,2}).
\end{aligned}
\end{equation}
Then, \eqref{eq:commutative-h} is satisfied modulo $q^4$ by the following refinement of \eqref{eq:guess-0}:
\begin{equation} \label{eq:guess-1}
h = \kappa_{\scrC}^{1,1}(\gamma_1;\gamma_2) + q^3 (b^{2,1} + b^{1,2}) \in \scrC^0 \hat\otimes q\bZ[[q]]/q^4.
\end{equation}
It remains to look at the geometry underlying \eqref{eq:k21}, \eqref{eq:k12}. Both cases are parallel, so let's focus on $K^{2,1}_{\scrC}$. For  $\beta^{2,1}_{\scrC}(c_1,c_2;c_3)$, take the relevant map and project it to actual Deligne-Mumford space for simplicity, which means considering the composition
\begin{equation}
\mathit{MWW}_{2,1} \longrightarrow \mathit{FM}_3 \longrightarrow \EuScript{DM}_3 \longrightarrow \mathit{DM}_3 \iso S^2.
\end{equation}
Looking at Figure \ref{fig:8-gon}, we see that three boundary sides of the octagon $\mathit{MWW}_{2,1}$, corresponding to the $\mu^3$ terms in \eqref{eq:k21}, are mapped to paths in $\mathit{DM}_3$ which are images of the canonical map $S_3 \rightarrow \mathit{DM}_3$ and two of its permuted versions (those which preserve the ordering between the first and second point in the configuration). The remaining sides are collapsed to ``special points'', meaning the images of the maps $\mathit{DM}_2 \times \mathit{DM}_2 \rightarrow \mathit{DM}_3$. Altogether, we get a relative cycle, whose homology class in $H_2(\mathit{DM}_3, (\mathit{DM}_3)_{\bR}) = H_2(S^2,S^1)$ is independent of all choices involved in the construction. From the assumption \eqref{eq:traffic-rules}, one sees easily that this cycle corresponds to one of the two discs in $S^2$ bounding $S^1$. Correspondingly, for $\beta^{1,2}_{\scrC}(c_3;c_1,c_2)$, we get a relative cycle corresponding to the other disc. The outcome of this discussion is that $K_{\scrC}^{2,1}$ is constructed from a cycle which represents a generator of $H_2(\mathit{DM}_3) = H_2(S^2)$ (the difference between the two discs bounding the same $S^1$, roughly speaking). Since the action of $\mathit{Sym}_3$ on $H_2(\mathit{DM}_3)$ is trivial, the induced map
\begin{equation} \label{eq:induced-k}
([c_1],[c_2],[c_3]) \mapsto [(-1)^{|c_2|}K_{\scrC}^{2,1}(c_1,c_2;c_3)]: H^*(\scrC)^{\otimes 3} \longrightarrow H^{*-2}(\scrC)
\end{equation}
is graded symmetric. In particular, $[K_{\scrC}^{2,1}(c_1,c_1;c_2)] \in H^1(\scrC)$ must be $2$-torsion. If we can rule out such torsion, then the class must necessarily vanish, and similarly for $[K_{\scrC}^{1,2}(c_1;c_2,c_2)]$. We can carry over the argument to other coefficients:

\begin{proposition} \label{th:commutative-2}
The product $\bullet $ on $\mathit{MC}(\scrC;N)$ is commutative if $N^3 = 0$. It is  also commutative if $N^4 = 0$, and additionally, $H^*(\scrC)$ is a free abelian group.
\end{proposition}

\begin{proof}
The first part is as in \eqref{eq:guess-0}. For the second part, the obstruction is now $[K_{\scrC}^{2,1}(c_1,c_1;c_2)] \in H^{*-2}(\scrC;N^3)$, and similarly for $K_{\scrC}^{1,2}$. From our assumption, it follows that $H^*(\scrC;G) = H^*(\scrC) \otimes G$ for any abelian group $G$. Hence, the symmetry argument that ensures vanishing of \eqref{eq:induced-k} carries over to arbitrary coefficients.
\end{proof}

\subsection{Strip-shrinking spaces and associativity}
Let's assume that $\scrC$ is homologically unital (this assumption is used in the context of geometric stabilization arguments, which add extra marked points).  Our aim is to show:

\begin{proposition} \label{th:semi-associativity}
For any $r \geq 2$ and any $1 \leq m \leq r-1$, $\Pi^r_{\scrC}(\gamma_1,\dots,\gamma_r)$ is equivalent to $\Pi^{r-1}_{\scrC}(\gamma_1,\dots,\gamma_m \bullet \gamma_{m+1},\dots,\gamma_r) = \Pi^{r-1}_{\scrC}(\gamma,\dots,\Pi^2_{\scrC}(\gamma_m,\gamma_{m+1}),\dots,\gamma_r)$.
\end{proposition}

Before getting into the proof, let's draw some immediate consequences.

\begin{corollary} \label{th:cor-as}
The product $\bullet$ is associative.
\end{corollary}

This is because, by the $r = 3$ case of Proposition \ref{th:semi-associativity}, $\Pi^3_{\scrC}(\gamma_1,\gamma_2,\gamma_3)$ is equivalent to both $\Pi^2_{\scrC}(\gamma_1 \bullet \gamma_2, \gamma_3) = (\gamma_1 \bullet \gamma_2) \bullet \gamma_3$ and $\Pi^2_{\scrC}(\gamma_1, \gamma_2 \bullet \gamma_3) = \gamma_1 \bullet (\gamma_2 \bullet \gamma_3)$.

\begin{corollary}
$\gamma = 0$ is the neutral element, meaning that $\gamma \bullet 0$ and $0 \bullet \gamma$ are equivalent to $\gamma$.
\end{corollary}

The definition says that
\begin{equation} \label{eq:0-gamma}
\gamma \bullet 0 = \sum_d \beta_{\scrC}^{d,0}(\gamma,\dots,\gamma),
\end{equation}
so the statement is not immediately obvious. However, it is obvious that $0 \bullet 0 = 0$. By Lemmas \ref{th:inverse-3} and \ref{th:inverse-4}, $\gamma \mapsto \gamma \bullet 0$ is a bijective map from $\mathit{MC}(\scrC;N)$ to itself. By associativity, that bijective map is idempotent, and therefore the identity. (There is a more direct geometric argument which shows that \eqref{eq:0-gamma} is equivalent to $\gamma$,  along the lines of Remark \ref{th:beta-2}, but we have chosen to avoid it.)

\begin{corollary}
$(\mathit{MC}(\scrC;N), \bullet)$ is a group.
\end{corollary}

This is a combination of the previous two corollaries and Lemma \ref{th:inverse-3}.

\begin{corollary} \label{th:r-at-least-3}
For any $r \geq 3$, $\Pi^r_{\scrC}(\gamma_1,\dots,\gamma_r)$ is equivalent to $\gamma_1 \bullet \cdots \bullet \gamma_r$.
\end{corollary}

This follows by induction: if $\Pi^{r-1}_{\scrC}(\gamma_1,\dots,\gamma_{r-1})$ is equivalent to $\gamma_1 \bullet \cdots \bullet \gamma_{r-1}$, then $\Pi^r_{\scrC}(\gamma_1,\dots,\gamma_r)$ is equivalent to $\Pi^{r-1}_{\scrC}(\gamma_1 \bullet \gamma_2,\dots,\gamma_r)$, hence to $(\gamma_1 \bullet \gamma_2) \bullet \cdots \bullet \gamma_r$.

\begin{corollary} \label{th:r-1-trivial}
$\Pi^1_{\scrC}(\gamma)$ is always equivalent to $\gamma$.
\end{corollary}

Proposition \ref{th:semi-associativity}, with $r = 2$, says that $\Pi^2_{\scrC}(\gamma_1,\gamma_2) = \gamma_1 \bullet \gamma_2$ is equivalent to $\Pi^1_{\scrC}(\gamma_1 \bullet \gamma_2)$, which implies the desired statement by specializing to $\gamma_1 = 0$ (again, this is a workaround which avoids a direct geometric argument).

\begin{proof}[Proof of Proposition \ref{th:semi-associativity}]
This uses the strip-shrinking moduli spaces from Section \ref{subsec:strip-shrinking}, with their maps \eqref{eq:double-stabilization}. The codimension one boundary faces are images of maps \eqref{eq:glue-tree-4} defined on the following spaces. First, in parallel with the first term of \eqref{eq:boundary-mww}, we have 
\begin{equation} \label{eq:ss-boundary}
\mathit{SS}_{d_1,\dots,d_k-j+1,\dots,d_r} \times S_j.
\end{equation}
where $j$ and $k$ can be arbitrary (in particular, the latter can be $m$ or $m+1$, something that's not entirely reflected in our notation). The analogue of the second term in \eqref{eq:boundary-mww} is less obvious:
\begin{equation} \label{eq:double-partition}
\begin{aligned}
& S_j \times \textstyle \prod_{i=1}^{k-1} \overbrace{\mathit{MWW}_{d_{i,1},\dots,d_{i,r}} \times 
\mathit{SS}_{d_{k,1}, \dots, (d_{k,m}, d_{k,m+1}),\dots,d_{k,r}}}^{\text{$r$ colors}} \\
& \qquad \qquad \textstyle \prod_{i=k+1}^j \big( \underbrace{\mathit{MWW}_{d_{i,1},\dots,d_{i,m-1},a_i,d_{i,m+2},\dots,d_{i,r}}}_{\text{$(r-1)$ colors}} 
\times \prod_{l = 1}^{a_i} \underbrace{\mathit{MWW}_{d_{i,l,1}, d_{i,l,2}}}_{\text{$2$ colors}} \big).
\end{aligned}
\end{equation}
Here, the last kind of $\mathit{MWW}$ factor corresponds to the small-mid vertices in the terminology of \eqref{eq:glue-tree-4}.  Such boundary faces are parametrized by ``double partitions''. One first chooses $j \geq 2$ and a partition $d_1 = d_{1,1} + \cdots + d_{j,1}$, \dots, $d_r = d_{1,r} + \cdots d_{j,r}$. Additionally, there is a distinguished $k \in \{1,\dots,j\}$, for which $(d_{k,1},\dots,d_{k,r})$ can be $(0,\dots,0)$. Finally, for each $i>k$ one chooses a further $a_i$ and partitions $d_{i,m} = d_{i,1,1} + \cdots + d_{i,a_i,1}$, $d_{i,m+1} = d_{i,1,2} + \cdots + d_{i,a_i,2}$.

We fix fundamental chains on the $\mathit{SS}$ spaces, compatible with the boundary structure in the usual sense. We then insert those chains into our operadic structure through \eqref{eq:double-stabilization}, with the additional convention that at the stabilizing marked points \eqref{eq:stabilization}, we will always apply a fixed homology unit $e_{\scrC}$. Denote the resulting operations by 
\begin{equation} \label{eq:alpha-operations}
\alpha_{\scrC}^{d_1,\dots,(d_m,d_{m+1}),\dots,d_r}: \scrC^{\otimes d} \longrightarrow \scrC[-d], \;\;
d = d_1 + \dots + d_r \geq 0.
\end{equation}
The simplest example is $\mathit{SS}_{0,\dots,0} = \mathit{point}$, which is mapped to $\mathit{FM}_2 = S^1$ by taking the configuration \eqref{eq:stabilization}. This coincides with the map $S_2 \rightarrow \mathit{FM}_2$ which is part of our $A_\infty$-structure, and therefore
\begin{equation} \label{eq:trivial-alpha}
\alpha^{0,\dots,(0,0),\dots,0}_{\scrC} = \mu^2_{\scrC}(e_{\scrC},e_{\scrC}) = e_{\scrC} + \text{\it (coboundary)}.
\end{equation}
Generally, the operations \eqref{eq:alpha-operations} satisfy the equation obtained from setting the sum of boundary contributions \eqref{eq:ss-boundary}, \eqref{eq:double-partition} equal to zero:
\begin{equation} \label{eq:property-of-alpha}
\begin{aligned}
& \sum_{ijk} \pm \alpha_{\scrC}^{d_1,\dots,d_k-j+1,\dots,d_r}\big(c_{1,1},\dots,c_{1,d_1};  \dots; 
c_{k,1},\dots, \mu_{\scrC}^j(c_{k,i+1},\dots,c_{k,i+j}),
\\[-1em] & \qquad \qquad \qquad \qquad \dots,c_{k,d_k} ; \dots; c_{r,1},\dots,c_{r,d_r}\big) \\[.5em]
= &
\sum_{\text{double partitions}} \pm \mu_{\scrC}^j\big(\beta_{\scrC}^{d_{1,1},\dots,d_{1,r}}(c_{1,1},\dots,c_{1,d_{1,1}};\dots;c_{r,1},\dots,c_{r,d_{1,r}}),  \dots,
\\ & \qquad 
\alpha_{\scrC}^{d_{k,1},\dots,d_{k,r}}(c_{1,d_{1,1} + \cdots + d_{k-1,1}+1},\dots,c_{1,d_{1,1}+\cdots+d_{k,1}};\dots),\dots,
\\ & \qquad 
\beta_{\scrC}^{d_{j,1},\dots,d_{j,m-1},a_j,d_{j,m+2},\dots,d_{j,r}}(c_{1,d_1-d_{j,1}+1},\dots,c_{1,d_1};
\\ & \qquad \qquad
\beta_{\scrC}^{d_{j,1,1},d_{j,1,2}}(c_{m,d_m-d_{j,m}+1},\dots;
c_{m+1,d_{m+1}-d_{j,m+1}+1},\dots), \dots,
\\ & \qquad \qquad
\beta_{\scrC}^{d_{j,a_j,1},d_{j,a_j,2}}(\dots,c_{m,d_m}; 
\dots,c_{m+1,d_{m+1}});
\\ & \qquad \qquad
c_{m+2,d_{m+2}-d_{m+2,j}+1},\dots,c_{m+2,d_{m+2}};\dots) \big),
\end{aligned}
\end{equation}
with double partitions as in \eqref{eq:double-partition}, the only difference being that we have an additional $\mu^1_{\scrC}(\beta_{\scrC}^{d_1,\dots,d_r})$ term. Given Maurer-Cartan elements $\gamma_1,\dots,\gamma_r$, set 
\begin{equation}
g = \sum_{d_1,\dots,d_r \geq 0} \alpha^{d_1,\dots,(d_m,d_{m+1}),\dots,d_r}_{\scrC}(\overbrace{\gamma_1,\dots,\gamma_1}^{d_1}; \dots; \overbrace{\gamma_r,\dots,\gamma_r}^{d_r}) \in \scrC^0 \hat\otimes (\bZ 1 \oplus N).
\end{equation}
As a direct consequence of \eqref{eq:property-of-alpha}, this satisfies
\begin{equation}
\begin{aligned}
&
\sum_{p,q} \mu^{p+q+1}_{\scrC}\big(\overbrace{\Pi^r_{\scrC}(\gamma_1,\dots,\gamma_r),\dots,\Pi^r_{\scrC}(\gamma_1,\dots,\gamma_r)}^p, g, \\[-.5em] & \qquad \qquad
\underbrace{\Pi^{r-1}_{\scrC}(\gamma_1,\dots,\gamma_m \bullet \gamma_{m+1},\dots,\gamma_r), \dots,
\Pi^{r-1}_{\scrC}(\gamma_1,\dots,\gamma_m \bullet \gamma_{m+1},\dots,\gamma_r)}_q \big) = 0.
\end{aligned}
\end{equation}
In view of \eqref{eq:trivial-alpha} and Lemma \ref{th:alternative-mc}, this is exactly what we need to prove the equivalence of the two Maurer-Cartan elements in question.
\end{proof}

\section{Cohomology operations\label{sec:operations}}

Following Steenrod and (in a more abstract context) May, reduced power operations arise from homotopy symmetries. This general principle can be applied to configuration spaces, as in Cohen's classical work, and also to Deligne-Mumford spaces. After a brief review of the underlying homological algebra, we discuss those two instances, and their relationship.

\subsection{Equivariant (co)homology\label{subsec:equivariant}}
Let $\scrC$ be a complex of vector spaces over a field $\bF$. Given an action of a group $G$ on $\scrC$, one can consider the group cochain complex $C^*_G(\scrC)$ and its cohomology $H^*_G(\scrC)$ (this is mild generalization of the classical concept of equivariant cohomology, where the coefficients lie in a $G$-module, see e.g.\ \cite{brown} for a general account, and \cite[p.\ 115]{maclane} or \cite[p.\ 179]{weibel} for the traditional choice of group cochain complex). If $C^*(X)$ is the cochain complex of a space $X$ carrying a $G$-action, and $V$ is a representation of $G$ over $\bF$, then setting $\scrC = C^*(X) \otimes V$ recovers
\begin{equation} \label{eq:borel-construction-1}
H^*_G(C^*(X) \otimes V) = H^*_G(X;V) = H^*(X \times_G EG;V),
\end{equation}
the equivariant cohomology in the classical sense (with coefficients in the local system over the Borel construction $X \times_G EG$ determined by $V$). A variant of the construction yields the group chain complex $C_*^G(\scrC)$ and group homology $H_*^G(\scrC)$. Recall that in our convention, all chain complexes are cohomologically graded. If we start with the chain complex $C_*(X)$ of a space, and a representation $V$, as before, then setting $\scrC = C_{-*}(X) \otimes V$ gives
\begin{equation} \label{eq:borel-construction-2}
H_*^G(C_{-*}(X) \otimes V) = H^G_{-*}(X;V) = H_{-*}(X \times_G EG;V).
\end{equation}
Group homology and cohomology carry exterior cup and cap products (see e.g.\ \cite[Section V.3]{brown})
\begin{align}
\label{eq:group-cup-product}
& H^*_G(\scrC_1) \otimes H^*_G(\scrC_2) \longrightarrow H^*_G(\scrC_1 \otimes \scrC_2), \\
\label{eq:group-cap-product}
& H_*^G(\scrC_1) \otimes H^*_G(\scrC_2) \longrightarrow H_*^G(\scrC_1 \otimes \scrC_2).
\end{align}

The cases relevant for our purpose are where $G$ is a permutation group $\mathit{Sym}_p$ of prime order, or its cyclic subgroup $\bZ/p$, and $\bF = \bF_p$. For the cyclic group, there are particularly simple complexes computing equivariant (co)homology. The cohomology version is
\begin{equation}
\label{eq:cyclic-cochains}
\left\{
\begin{aligned}
& C^*_{\bZ/p}(\scrC) = \scrC[[t]] \oplus \theta \scrC[[t]], \quad |t| = 2, \; |\theta| = 1, \\
& d_{\bZ/p}(t^k c) = t^k\, dc + \theta t^k(T c - c), \\
& d_{\bZ/p}(\theta t^k c) = -\theta t^k\,dc + t^{k+1}(c + T c + \cdots + T^{p-1} c),
\end{aligned}
\right.
\end{equation}
where $T: \scrC \rightarrow \scrC$ is the generator of the $\bZ/p$-action. In the case of trivial coefficients $\scrC = \bF_p$, the differential vanishes. The ring structure \eqref{eq:group-cup-product}, for $\scrC_1 = \bF_p$ and general $\scrC_2 = \scrC$, satisfies $[t] \cdot [t^k c] = [t^{k+1} c]$ and $[t] \cdot [t^k \theta c] = [t^{k+1} \theta c]$. However, for $p = 2$ and $\scrC_1 = \scrC_2 = \bF_p$ one has $[\theta] \cdot [\theta] = [t]$, while for $p>2$ that expression would be zero. Indeed, in the case $p = 2$ it is more convenient to write $\theta = t^{1/2}$ (for a more precise discussion of the choices of generators used here in relation to topology, we refer to Section \ref{subsec:sign-of-t}). The group homology version is
\begin{equation}
\label{eq:cyclic-chains}
\left\{
\begin{aligned}
& C_*^{\bZ/p}(\scrC) = \scrC[s] \oplus \sigma \scrC[s], \quad |s| = -2, \; |\sigma| = -1, \\
& d^{\bZ/p}(s^k c) = s^k \, dc - \sigma s^{k-1} (c + Tc + \cdots + T^{p-1} c), \\
& d^{\bZ/p}(\sigma s^k c) = -\sigma s^k \, dc - s^k (Tc - c),
\end{aligned}
\right.
\end{equation}
and under \eqref{eq:group-cap-product}, $t$ acts on equivariant homology by cancelling one power of $s$ in \eqref{eq:cyclic-chains} (by convention, $s^{-1}$ is set to zero). Because the index of $\bZ/p \subset \mathit{Sym}_p$ is coprime to $p$, 
\begin{align}
& H^*_{\mathit{Sym}_p}(\scrC) \longrightarrow H^*_{\bZ/p}(\scrC) \quad \text{is injective,}
\label{eq:restrict-to-cyclic} \\
\label{eq:project-to-symmetric}
& H_*^{\bZ/p}(\scrC) \longrightarrow H_*^{\mathit{Sym}_p}(\scrC) \quad \text{is surjective}
\end{align}
for every complex $\scrC$ with $\mathit{Sym}_p$-action. Let $\bF_p(l)$, $l \in \bZ$, be the one-dimensional representations which are: trivial if $l$ is even; and associated to $\mathrm{sign}: \mathit{Sym}_p \rightarrow \{\pm 1\} \subset \bF_p^\times$ if $l$ is odd. Note that the restriction of the $\mathrm{sign}$ homomorphism to $\bZ/p$ is trivial. The relevant special case of \eqref{eq:restrict-to-cyclic} can be made explicit as follows (see e.g. \cite[Lemma 1.4]{may70}):
\begin{align} 
\label{eq:symmetric-group-cohomology-1}
& 
H^*_{\mathit{Sym_p}}(\bF_p) = \bF_p[t^{p-1}] \oplus \theta t^{p-2} \bF_p[t^{p-1}]  \subset H^*_{\bZ/p}(\bF_p), 
\\
\label{eq:symmetric-group-cohomology-2}
&
H^*_{\mathit{Sym_p}}(\bF_p(1)) = t^{(p-1)/2} \bF_p[t^{p-1}] \oplus \theta t^{(p-3)/2} \bF_p[t^{p-1}]
\subset H^*_{\bZ/p}(\bF_p).
\end{align}

Let $\scrC$ be a general complex of $\bF_p$-vector spaces, and consider $\mathit{Sym}_p$ acting on its tensor power $\scrC^{\otimes p}$, by permuting the factors with Koszul signs. In this situation, there is a canonical equivariant diagonal map
\begin{equation} \label{eq:equivariant-diagonal}
H^l(\scrC) \longrightarrow H^{pl}_{\mathit{Sym}_p}(\scrC^{\otimes p} \otimes \bF_p(l)),
\end{equation}
which lifts the standard diagonal $H^l(\scrC) \rightarrow H^{pl}(\scrC^{\otimes p})$ (see e.g.\ \cite[Lemma 1.1(iv)]{may70} for its well-definedness). The equivariant diagonal is compatible with multiplication by elements of $\bF_p$, but not additive. 
It is sometimes convenient to simplify the discussion of \eqref{eq:equivariant-diagonal} by restricting to the cyclic subgroup,
\begin{equation}
\label{eq:equivariant-diagonal-2}
\xymatrix{
& H^{pl}_{\mathit{Sym}_p}(\scrC^{\otimes p} \otimes \bF_p(l)) \ar@{^{(}->}[d], \\
H^k(\scrC) \ar[ur]^-{\eqref{eq:equivariant-diagonal}} \ar[r] 
& H^{pl}_{\bZ/p}(\scrC^{\otimes p}).
}
\end{equation}
By explicit computation in \eqref{eq:cyclic-cochains}, one sees that the equivariant diagonal for $\bZ/p$ becomes additive after multiplying with $t$. It follows that \eqref{eq:equivariant-diagonal} becomes additive after multiplying with $t^{p-1}$, since that lies in the subgroup \eqref{eq:symmetric-group-cohomology-1}.

\subsection{Cohen operations\label{subsec:cohen}}
Let $\scrC$ be a complex of $\bF_p$-vector spaces, which has the structure of an algebra over the $\bF_p$-coefficient version of the Fulton-MacPherson operad. Recall that the action of $\mathit{Sym}_p$ on $\mathit{FM}_p$ is free. The associated Cohen operation is a map
\begin{equation} \label{eq:cohen-operations} 
\mathit{Coh}_p: H^l(\scrC) \longrightarrow (H^*(\mathit{FM}_p/\mathit{Sym}_p;\bF_p(l)) \otimes H^*(\scrC))^{pl},
\end{equation}
defined as follows:
\begin{equation} \label{eq:cohen-operations-2} 
\xymatrix{
\ar@{-->}[dddd]
H^l(\scrC) \ar[r]^-{\eqref{eq:equivariant-diagonal}}
& 
H^{pl}_{\mathit{Sym}_p}(\scrC^{\otimes p} \otimes \bF_p(l))  \ar[d]_-{\text{operad structure}}
\\
& 
H^{pl}_{\mathit{Sym}_p}(\mathit{Hom}(C_{-*}(\mathit{FM}_p), \scrC) \otimes \bF_p(l)) \ar@{=}[d] 
\\
&
H^{pl}_{\mathit{Sym}_p}(\mathit{Hom}(C_{-*}(\mathit{FM}_p) \otimes \bF_p(l), \scrC)) \ar@{=}[d]_-{\text{Kunneth}}
\\ &
\mathit{Hom}^{pl}(H_{-*}^{\mathit{Sym}_p}(\mathit{FM}_p;\bF_p(l)), H^*(\scrC))
\ar@{=}[d]_-{\text{freeness of the action}}
\\
\big( H^*(\scrC) \otimes H^*(\mathit{FM}_p/\mathit{Sym}_p;\bF_p(l)) \big)^{pl} \ar@{=}[r]
& 
\mathit{Hom}^{pl}(H_{-*}(\mathit{FM}_p/\mathit{Sym}_p;\bF_p(l)), H^*(\scrC)).
}
\end{equation}
On the middle lines, the $\mathit{Sym}_p$-action is trivial on the $\scrC$-factor. Because their definition involves \eqref{eq:equivariant-diagonal}, these operations are not expected to be additive. Note that we could also have defined our operations using $\bZ/p$, but of course, the resulting operations would still lie in the subspace $H^*(\mathit{FM}_p/\mathit{Sym}_p;\bF_p(l)) \subset H^*(\mathit{FM}_p/(\bZ/p);\bF_p)$. For computational purposes, let's spell out what happens when one decodes \eqref{eq:cohen-operations-2}:

\begin{lemma} \label{th:explicit-cohen}
Suppose that we have an cycle $c \in \scrC$ of degree $l$. Then $c^{\otimes p} \in \scrC^{\otimes p}$ is a cycle which is $\mathit{Sym}_p$-invariant up to an $\bF_p(l)$-twist, and which therefore represents a class in $H^*_{\mathit{Sym}_p}(\scrC^{\otimes p} \otimes \bF_p(l))$. Similarly, suppose that we have a chain $B \in C_*(\mathit{FM}_p;\bF_p)$ with the property that $\partial B$ goes to zero in $C_*(\mathit{FM}_p) \otimes_{\mathit{Sym}_p} \bF_p(l)$. Such a chain represents a class $[B] \in H_*(\mathit{FM}_p/\mathit{Sym}_p;\bF_p(l))$. As a consequence of the properties of $B$ and $c$, the image of $B \otimes c^{\otimes p}$ under the operad action is a cycle in $\scrC$. That cycle represents the image of $[c]$ under \eqref{eq:cohen-operations}, paired with $[B]$.
\end{lemma}

The structure of Cohen operations was determined in \cite[Theorems 5.2 and 5.3]{cohen76}. The group relevant for operations on the even degree cohomology of $\scrC$ is
\begin{equation}  \label{eq:configuration-space-untwisted}
H^*(\mathit{FM}_p/\mathit{Sym}_p;\bF_p) \iso 
\begin{cases} \bF_p & \ast = 0, \, 1, \\ 0 & \text{otherwise.} 
\end{cases}
\end{equation}
For $\ast = 0$, that just recovers the $p$-fold power for the product that is part of the Gerstenhaber algebra structure on $H^*(\scrC)$. If we suppose that $p>2$, the operation obtained from the $\ast = 1$ group can again be described as part of the Gerstenhaber structure, as $x \mapsto [x,x] x^{p-2}$. The twisted counterpart is more interesting:
\begin{equation}
\label{eq:configuration-space-twisted}
H^*(\mathit{FM}_p/\mathit{Sym}_p;\bF_p(1)) \iso 
\begin{cases} \bF_p & \ast = p-1, \, p-2, \\ 0 & \text{otherwise.} \end{cases}
\end{equation}
Moreover, still as part of \cite[Theorem 5.3]{cohen76}, the pullback map
\begin{equation} \label{eq:pullback-from-point}
H^*_{\mathit{Sym}_p}(\bF_p(1)) = H^*(B\mathit{Sym}_p;\bF_p(1)) \longrightarrow H^*(\mathit{FM}_p/\mathit{Sym}_p;\bF_p(1))
\end{equation}
is onto. Therefore, the groups \eqref{eq:configuration-space-twisted} can be thought of as generated by $t^{(p-1)/2}$ and $\theta t^{(p-3)/2}$, the lowest degree generators in \eqref{eq:symmetric-group-cohomology-2}. Note that for $p>2$, $t^{(p-1)/2}$ is the image of $\theta t^{(p-3)/2}$ under the Bockstein $\beta$.

\subsection{Quantum Steenrod operations\label{subsec:quantum-steenrod}}
The same idea works for any operad, and in particular, Deligne-Mumford spaces. Concretely, this means that we consider the action of $\mathit{Sym}_p$ on $\mathit{DM}_p$ which keeps the marked point $z_0$ fixed, and permutes $(z_1,\dots,z_p)$. Given an algebra $\scrC$ over the $\bF_p$-coefficient Deligne-Mumford operad, one gets operations analogous to \eqref{eq:cohen-operations},
\begin{equation} \label{eq:equivariant-dm}
H^l(\scrC) \longrightarrow \big( H^*(\scrC) \otimes H^*_{\mathit{Sym}_p}(\mathit{DM}_p;\bF_p(l))\big)^{pl}. 
\end{equation}
In principle, the same caveat as in Section \ref{subsec:commutativity} applies, which means that we should replace $\mathit{DM}_p$ by a homotopy equivalent space \eqref{eq:freeing-up}. However, that makes no difference for the present discussion, since only the equivariant cohomology of the space will be involved.

Unfortunately, the equivariant mod $p$ cohomology of Deligne-Mumford space is not known (to this author, at least), but there are simplified versions of this construction which are easier to understand. Because $p$ is assumed to be prime, the $\mathit{Sym}_p$-action on $\mathit{DM}_p$ has a unique orbit $O_p$ with isotropy subgroups isomorphic to $\bZ/p$ (all other isotropy subgroups have orders not divisible by $p$). For concreteness, we just look at one specific point $\diamond \in O_p$, whose  isotropy subgroup is the standard cyclic subgroup $\bZ/p \subset \mathit{Sym}_p$: 
\begin{equation} \label{eq:fixed-point}
\diamond = \big(C = \bar{\bC} = \bC \cup \{\infty\}, \, z_0 = \infty, \, z_k = e^{2\pi \sqrt{-1} k/p} \;\; \text{for $k = 1,\dots,p$}\big).
\end{equation}

\begin{lemma} \label{th:orbit}
Restriction to $\diamond \in O_p$ yields isomorphisms
\begin{equation} \label{eq:include-isomorphism}
\xymatrix{
\ar@/_1.5pc/[rr]_-{\iso}
H^*_{\mathit{Sym}_p}(O_p;\bF_p(l)) \ar[r] & H^*_{\bZ/p}(O_p;\bF_p) \ar[r] & H^*_{\bZ/p}(\diamond;\bF_p) = H^*_{\bZ/p}(\bF_p).
}
\end{equation}
\end{lemma}

\begin{proof}
This is elementary: the underlying map of Borel constructions, obtained by composing
\begin{equation}
E\mathit{Sym}_p \times_{\bZ/p} \diamond \hookrightarrow E\mathit{Sym}_p \times_{\bZ/p} O_p \twoheadrightarrow E\mathit{Sym}_p \times_{\mathit{Sym}_p} O_p
\end{equation}
is a homeomorphism. Moreover, the local system on $E\mathit{Sym}_p \times_{\mathit{Sym}_p} O_p$ associated to $\bF_p(1)$ is canonically trivial.
\end{proof}
 
%
Quantum Steenrod operations are obtained by replacing $\mathit{DM}_p$ in \eqref{eq:equivariant-dm} by its subspace $O_p$. In view of Lemma \ref{th:orbit}, we can equivalently define them using the $\bZ/p$-equivariant cohomology of a point. Written in that way, they have the form
\begin{equation} \label{eq:quantum-steenrod-2}
\mathit{QSt}_p: H^l(\scrC) \longrightarrow (H^*(\scrC) \otimes H^*_{\bZ/p}(\bF_p))^{pl}.
\end{equation}
Following our discussion of \eqref{eq:equivariant-diagonal}, we know that \eqref{eq:quantum-steenrod-2} becomes additive after multiplication with $t^{p-1}$. Since that multiplication acts injectively on $H^*_{\bZ/p}(\bF_p)$, one sees that \eqref{eq:quantum-steenrod-2} is already additive. 

As an intermediate object between the two spaces considered so far, take $\mathit{DM}^\circ_p$ be the moduli space of smooth genus zero curves with $(p+1)$ marked points, or equivalently \eqref{eq:configuration-space-3}, which is an open subset of $\mathit{DM}_p$ containing $O_p$. Similarly, let $\mathit{FM}_p^\circ$ be the configuration space \eqref{eq:configuration-space-1}, which is the interior of $\mathit{FM}_p$ and hence homotopy equivalent to the whole space. The forgetful map $\mathit{FM}_p \rightarrow \mathit{DM}_p$ restricts to a circle bundle $\mathit{FM}_p^\circ \rightarrow \mathit{DM}^{\circ}_p$.  

\begin{lemma} \label{th:open-moduli-space}
Restriction to $O_p \subset \mathit{DM}_p^\circ$, together with Lemma \ref{th:orbit}, yields isomorphisms
\begin{align} \label{eq:include-fixed-locus}
& H^*_{\mathit{Sym}_p}(\mathit{DM}^\circ_p;\bF_p) \iso H^*_{\bZ/p}(\mathit{point};\bF_p), \\
\label{eq:include-fixed-locus-2}
& H^*_{\mathit{Sym}_p}(\mathit{DM}^\circ_p;\bF_p(1)) \iso
\begin{cases} 
0 & * < p-2, \\
H^*_{\bZ/p}(\mathit{point};\bF_p) & * \geq p-2. 
\end{cases}
\end{align}
Moreover, the pullback map is an isomorphism 
\begin{equation} \label{eq:two-degrees}
H^*_{\mathit{Sym}_p}(\mathit{DM}_p^\circ;\bF_p(1)) \longrightarrow H^*_{\mathit{Sym}}(\mathit{FM}^\circ_p;\bF_p(1)),
\quad * = p-2, \, p-1.
\end{equation}
\end{lemma}

\begin{proof}
Consider the Gysin sequence and its restriction to \eqref{eq:fixed-point}:
\begin{equation} \label{eq:gysin}
\xymatrix{
\cdots \rightarrow
H_{\mathit{Sym}_p}^{*-2}(\mathit{DM}_p^\circ;\bF_p(l)) \ar[r] \ar[d]
&
H^*_{\mathit{Sym}_p}(\mathit{DM}^\circ_p;\bF_p(l)) \ar[r] \ar[d]
&
H^*_{\mathit{Sym}_p}(\mathit{FM}^\circ_p;\bF_p(l)) \rightarrow \cdots
\ar[d]
\\
\cdots \rightarrow H^{*-2}_{\bZ/p}(\diamond;\bF_p) \ar[r]^-{-t}
&
H^*_{\bZ/p}(\diamond;\bF_p) \ar[r] 
& 
H^*_{\bZ/p}(S^1;\bF_p) \rightarrow \cdots
}
\end{equation}
Over $\diamond$, the fibre of the circle bundle $\mathit{FM}_p^\circ \rightarrow \mathit{DM}_p^\circ$ can be identified with the representation of $\bZ/p$ with weight $-1$. In other words, the $S^1$ in \eqref{eq:gysin} carries the action of $\bZ/p$ by clockwise rotation. The $-t$ appearing in the sequence is the associated equivariant Euler class. For $l = 0$, inspection of \eqref{eq:configuration-space-untwisted} shows that the rightmost $\downarrow$ in \eqref{eq:gysin} is always an isomorphism. One can therefore prove \eqref{eq:include-fixed-locus} by upwards induction on degree. 

For $l = 1$, we use a variant of the same argument. The Gysin sequence and \eqref{eq:configuration-space-twisted} imply that
\begin{equation} \label{eq:use-gysin-0}
H^*_{\mathit{Sym}_p}(\mathit{DM}^\circ_p;\bF_p(1)) \iso \begin{cases}
0 & * < p-2, \\
H^*_{\mathit{Sym}_p}(\mathit{FM}^\circ_p;\bF_p(1)) 
& * = p-2.
\end{cases}
\end{equation}
Let's look at the first nontrivial degree, and the maps
\begin{equation}
H^{p-2}_{\mathit{Sym}_p}(\mathit{point};\bF_p(1)) \xrightarrow{\text{pullback}} H^{p-2}_{\mathit{Sym}_p}(\mathit{DM}^\circ_p;\bF_p(1)) \xrightarrow{\text{restriction}} H^{p-2}_{\bF_p}(\diamond;\bF_p).
\end{equation}
From \eqref{eq:use-gysin-0} and \eqref{eq:pullback-from-point}, it follows that the first map is an isomorphism. The composition of the two maps is just \eqref{eq:restrict-to-cyclic}, hence an isomorphism. It follows that the second map must be an isomorphism as well, which is part of \eqref{eq:include-fixed-locus-2}. On the other hand, since the $\mathit{Sym}_p$-action has isotropy groups of order coprime to $p$ outside $O_p$,
\begin{equation} \label{eq:borel-vs-quotient}
H^*_{\mathit{Sym}_p}(\mathit{DM}^\circ_p,O_p;\bF_p(1)) \iso H^*(\mathit{DM}^\circ_p/\mathit{Sym}_p, O_p/\mathit{Sym}_p;\bF_p(1)),
\end{equation}
and the right hand side vanishes in high degrees. From that and Lemma \ref{th:orbit}, one sees that 
the restriction map $H^*_{\mathit{Sym}_p}(\mathit{DM}^\circ_p;\bF_p(1)) \rightarrow H^*_{\bZ/p}(\mathit{point};\bF_p)$ is an isomorphism in high degrees. By downward induction on degree, using \eqref{eq:configuration-space-twisted} and \eqref{eq:gysin}, one obtains the degree $\geq p-1$ part of \eqref{eq:include-fixed-locus-2}. From \eqref{eq:include-fixed-locus-2}, it also follows that the map $H_{\mathit{Sym}_p}^{p-2}(\mathit{DM}^\circ_p;\bF_p(1)) \rightarrow H_{\mathit{Sym}_p}^p(\mathit{DM}^\circ_p;\bF_p(1))$ in the top row of \eqref{eq:gysin} is an isomorphism, which then implies \eqref{eq:two-degrees} in degree $(p-1)$; the degree $(p-2)$ part of the same statement has been derived before, in \eqref{eq:use-gysin-0}.
\end{proof}

As an immediate consequence, suppose that $\scrC$ is an algebra over the chain level Deligne-Mumford operad. Consider its induced structure as an algebra over the Fulton-MacPherson operad. By definition, the associated operations \eqref{eq:cohen-operations}, \eqref{eq:quantum-steenrod-2} fit into a commutative diagram
\begin{equation} \label{eq:comparison}
\xymatrix{
& \big(H^*(\scrC) \otimes H^*(\mathit{FM}_p^\circ/\mathit{Sym}_p;\bF_p(l)) \big)^{pl} \\
H^l(\scrC) \ar[ur]^-{\text{Cohen}\;\;} \ar[r] \ar[dr]_-{\text{quantum Steenrod}\;\;\;\;\;\;\;\;\;\;}
& \big(H^*(\scrC) \otimes H^*_{\mathit{Sym}_p}(\mathit{DM}_p^\circ;\bF_p(l)) \big)^{pl} 
\ar[u] \ar[d]
\\
& \big(H^*(\scrC) \otimes H^*_{\bZ/p}(\bF_p) \big)^{pl}
}
\end{equation}
If $l$ is odd, then Lemma \ref{th:open-moduli-space} shows that both vertical arrows are isomorphisms on the degree $(p-2)$ or $(p-1)$ cohomology groups of the moduli spaces. Those cohomology groups are one-dimensional, and their generators can be identified with $\theta t^{(p-3)/2}$ and $t^{(p-1)/2}$, respectively. To put it more succinctly:

\begin{lemma} \label{th:cohen-and-steenrod}
The Cohen and quantum Steenrod operations 
\begin{equation}
H^l(\scrC) \longrightarrow H^{pl-k}(\scrC), \quad \text{for $l$ odd, and $k = p-2$ or $k = p-1$}
\end{equation}
coincide.
\end{lemma}

\section{Prime power maps\label{sec:power}}

This section brings together the lines of thought from Sections \ref{sec:group} (formal group structure) and \ref{sec:operations} (cohomology operations). Our first task is to make part of the discussion in Section \ref{subsec:cohen} more concrete, by introducing an explicit cocycle which generates $H^{p-1}(\mathit{FM}_p/\mathit{Sym}_p;\bF_p(1))$. By looking at the relation between that cocycle and the map $\mathit{MWW}_{1,\dots,1} \rightarrow \mathit{FM}_p$, we obtain an abstract analogue of Theorem \ref{th:p-power} in the operadic context.

\subsection{A cocycle in unordered Fulton-MacPherson space} 
Take \eqref{eq:s-into-fm}, modify it by rotation by $i$ so as to put the resulting configurations on the imaginary axis in $\bC$, and then compose that with projection to $\mathit{FM}_p/\mathit{Sym}_p$ (recall that the $\mathit{Sym}_p$-action is free, so the quotient is again a smooth manifold with corners, or topologically a manifold with boundary). The outcome is a submanifold (a copy of the associahedron $S_p$)
\begin{equation} \label{eq:embed-s}
Z_p \subset \mathit{FM}_p/\mathit{Sym}_p,
\end{equation}
with $\partial Z_p \subset \partial \mathit{FM}_p/\mathit{Sym}_p$. By definition, \eqref{eq:embed-s} has a preferred lift to $\mathit{FM}_p$, and therefore, the local system $\bF_p(1)|Z_p$ has a canonical trivialization. Using that and the orientations of $S_p$ and $\mathit{FM}_p$, we get a class 
\begin{equation} \label{eq:dual-cocycle}
[Z_p] \in H_{p-2}(\mathit{FM}_p/\mathit{Sym}_p, \partial \mathit{FM}_p/\mathit{Sym}_p; \bF_p(1))  \iso
H^{p-1}(\mathit{FM}_p/\mathit{Sym}_p;\bF_p(1)).
\end{equation}
In terms of the previous computations \eqref{eq:symmetric-group-cohomology-2}, \eqref{eq:configuration-space-twisted}, this can be expressed as follows:

\begin{lemma} \label{th:dual-cocycle}
For $p = 2$, \eqref{eq:dual-cocycle} is the image of $\theta = t^{1/2}$ under \eqref{eq:pullback-from-point}; for $p>2$, it is the image of 
\begin{equation}
(-1)^{\frac{p-1}{2}} {\textstyle \big(\frac{p-1}{2}!\big)} t^{\frac{p-1}{2}} \in H^{p-1}(B\mathit{Sym}_p;\bF_p(1)).
\end{equation}
\end{lemma}

\begin{proof}
Let's consider the more interesting case $p>2$ first. Take the map
\begin{equation} \label{eq:real-part}
\mathit{Conf}_p(\bC) \longrightarrow \bR^p/\bR = \bR^{p-1}
\end{equation}
which projects ordered configurations to their real part, and then quotients out by the diagonal $\bR$ subspace. This map is $\mathit{Sym}_p$-equivariant, and the fibre at $0$ is the subspace $\tilde{Z}_p = \bR \times \mathit{Conf}_p(\bR) \subset \mathit{Conf}_p(\bC)$ of configurations with common real part, $(z_1 = s+\sqrt{-1}\,t_1, \dots, z_p = s+\sqrt{-1}\,t_p)$. Let's orient that by using the coordinates $(s,t_1,\dots,t_p)$ in this order. This differs from its orientation as a fibre of \eqref{eq:real-part} by a Koszul sign $(-1)^{p(p-1)/2} = (-1)^{(p-1)/2}$. On the other hand, the fibre at $0$ represents the pullback via \eqref{eq:real-part} of the equivariant Euler class of the $\bZ/p$-representation $\bR^p/\bR$. From this and \eqref{eq:cyclic-euler}, we get
\begin{equation} \label{eq:preimage-0}
[\tilde{Z}_p] = (-1)^{\frac{p-1}{2}} {\textstyle \big( \frac{p-1}{2}! \big)} t^{\frac{p-1}{2}} \in H^{p-1}_{\bZ/p}(\mathit{Conf}_p(\bC);\bF_p).
\end{equation}
The corresponding relation must hold in $H^{p-1}_{\mathit{Sym}_p}(\mathit{Conf}_p(\bC);\bF_p(1))$ as well, since both classes involved live in that group, and the map from there to $\bZ/p$-equivariant cohomology is injective. Note that $\tilde{Z}_p$ is the preimage of $Z_p$ under the quotient map $\mathit{Conf}_p(\bC) \rightarrow \mathit{FM}_p$. Moreover, inspection of \eqref{eq:split-orientations} shows that the orientations of $\mathit{FM}_p$, $Z_p$ and $\tilde{Z}_p$ we have used are compatible with that relation. Since the quotient map is equivariant and a homotopy equivalence, \eqref{eq:preimage-0} implies the corresponding property for $[Z_p]$.

One could follow the same strategy for $p = 2$, but we can be even more explicit. The generator of $H_1(\mathit{FM}_2/\mathit{Sym}_2;\bZ/2) \iso \bZ/2$ consists of a loop of configurations where two points rotate around each other, and its image in $H_1(B\mathit{Sym}_2;\bZ/2) \iso \bZ/2$ is obviously nontrivial. On the other hand, that loop intersects $Z_p$ transversally at exactly one point, which proves the desired statement.
\end{proof}

\subsection{A cycle in unordered Fulton-MacPherson space\label{subsec:unordered}}
Consider the space $\mathit{MWW}_{1,\dots,1}$ with $d$ colors, denoted here by $\mathit{MWW}_{\howmany{d}}$ for the sake of brevity. As a special case of \eqref{eq:glue-tree-3}, its codimension $1$ boundary faces are of the form
\begin{equation} \label{eq:111-faces}
\mathit{MWW}_{\howmany{d_1}} \times \cdots \times \mathit{MWW}_{\howmany{d_r}} \times S_r \xrightarrow{T_{I_1,\dots,I_r}} \mathit{MWW}_{\howmany{d}};
\end{equation}
there is one such face for each decomposition of $\{1,\dots,d\}$ into $r \geq 2$ nonempty subsets $(I_1,\dots,I_r)$, with $d_k = |I_k|$. In describing the boundary faces, we have used the identifications \eqref{eq:drop-color}. Suppose that we choose maps \eqref{eq:mww-into-fm} so as to be compatible with \eqref{eq:drop-color}, as in Section \ref{subsec:multiproduct-p}. Consider two decompositions $(I_1,\dots,I_d)$ and $(\tilde{I}_1,\dots,\tilde{I}_d)$, which correspond to the same ordered partition $d_k = |I_k| = |\tilde{I}_k|$. Then, the associated maps \eqref{eq:111-faces} and \eqref{eq:mww-into-fm} fit into a commutative diagram
\begin{equation} \label{eq:two-decompositions}
\xymatrix{
\mathit{MWW}_{\howmany{d}} \ar[d]
&& \mathit{MWW}_{\howmany{d_1}} \times \cdots \times \mathit{MWW}_{\howmany{d_r}} \times S_r 
\ar[rr]^-{T_{\tilde{I}_1,\dots,\tilde{I}_r}}
\ar[ll]_-{T_{I_1,\dots,I_r}} &&
\mathit{MWW}_{\howmany{d}} \ar[d]
\\
\mathit{FM}_d && \mathit{FM}_d \ar[ll]_-{\sigma_{I_1,\dots,I_r}}^-{\iso} \ar[rr]^-{\sigma_{\tilde{I}_1,\dots,\tilde{I}_r}}_-{\iso} && \mathit{FM}_d.
}
\end{equation}
Here, $\sigma_{I_1,\dots,I_r} \in \mathit{Sym}_d$ is the unique permutation which maps $\{1,\dots,d_1\}$ order-preservingly to $I_1$, $\{d_1+1,\dots,d_1+d_2\}$ order-preservingly to $I_2$, and so on; and correspondingly for $\sigma_{\tilde{I}_1,\dots,\tilde{I}_r}$. Suppose that we choose fundamental chains to be also compatible with \eqref{eq:drop-color}. Then, \eqref{eq:boundary-mww} simplifies to
\begin{equation} \label{eq:boundary-howmany}
\partial [\mathit{MWW}_{\howmany{d}}] = \sum_{(I_1,\dots,I_r)} \pm T_{I_1,\dots,I_r,*} ([\mathit{MWW}_{\howmany{d_1}}] \times \cdots \times [\mathit{MWW}_{\howmany{d_r}}] \times [S_r]). 
\end{equation}
Thinking of the homology of $\mathit{FM}_p/\mathit{Sym}_p$ as in Lemma \ref{th:explicit-cohen}, we get:

\begin{lemma} \label{th:quotient-cycle}
Suppose that $d = p$ is prime. Then, the image of $[\mathit{MWW}_{\howmany{p}}]$ under \eqref{eq:mww-into-fm} is a chain, denoted here by $B_p \in C_*(\mathit{FM}_p;\bF_p)$, with the property that $\partial B_p$ goes to zero in $C_*(\mathit{FM}_p;\bF_p) \otimes _{\mathit{Sym}_p} \bF_p(1)$. Therefore, it represents a class $[B_p] \in H_{p-1}(\mathit{FM}_p/\mathit{Sym}_p;\bF_p(1))$.
\end{lemma}

\begin{proof}
Consider two codimension $1$ boundary faces as in \eqref{eq:two-decompositions}. The resulting chains in $\mathit{FM}_d$ differ by applying the permutation $\sigma_{\tilde{I}_1,\dots,\tilde{I}_r}\sigma_{I_1,\dots,I_r}^{-1}$. Hence, when mapped to $C_*(\mathit{FM}_d) \otimes_{\mathit{Sym}_p} \bF_p(1)$, they differ by the sign of that permutation. On the other hand, their entries in \eqref{eq:boundary-howmany} differ by the same sign. When computing $\partial B_p$ in $C_*(\mathit{FM}_p) \otimes_{\mathit{Sym}_p} \bF_p(1)$, the two kinds of signs cancel, which means that the contributions are the same. Now, the cyclic group $\bZ/p \subset \mathit{Sym}_p$ acts freely on ordered  decompositons corresponding to the same ordered partition, and this provides the required cancellation mod $p$ for the terms of $\partial B_p$.
\end{proof}


\begin{lemma} \label{th:one-point}
The canonical pairing between the cohomology class of \eqref{eq:dual-cocycle} and the homology class from Lemma \eqref{th:quotient-cycle} is $[Z_p] \cdot [B_p] = (-1)^{p(p-1)/2}$.
\end{lemma}

\begin{proof}
We think of this Poincar{\'e}-dually as an intersection number. The relevant cycles intersect at exactly one point of $\mathit{FM}_p$, which is a configuration $(z_1;\dots;z_p)$ with $\mathrm{re}(z_1) = \cdots = \mathrm{re}(z_d) = 0$ and $\mathrm{im}(z_1) < \cdots < \mathrm{im}(z_d)$. The tangent space of $Z_d \subset \mathit{FM}_d$ at that point can be thought of as keeping $(z_1,z_2)$ fixed, and moving $(z_3,\dots,z_d)$ infinitesimally in imaginary direction. The tangent space to the image of $\mathit{MWW}_{\howmany{p}}$ at the same point consists of keeping $z_1$ fixed, but moving $(z_2,\dots,z_d)$ infinitesimally in real direction. Note that positive horizontal motion of $z_2$ yields a clockwise motion of the angular component of $z_2-z_1$. This observation, when combined with standard Koszul signs, yields the desired local intersection number.
\end{proof}
%

To see the implications of Lemma \ref{th:dual-cocycle} and \ref{th:one-point}, note that by the dual of \eqref{eq:configuration-space-twisted} and \eqref{eq:pullback-from-point},
\begin{equation} \label{eq:hp-1}
H_{p-1}(\mathit{FM}_p/\mathit{Sym}_p;\bF_p(1)) \iso H_{p-1}(\mathit{BSym}_p;\bF_p(1)) \iso \bF_p.
\end{equation}
In those terms, what we have shown is:

\begin{lemma} \label{th:cycle-found}
For $p = 2$, the homology class of the cycle from Lemma \ref{th:quotient-cycle} is the unique nontrivial element in \eqref{eq:hp-1}. For $p>2$, it is $(\frac{p-1}{2}!)^{-1}$ times the standard generator (dual to $t^{(p-1)/2}$) .
\end{lemma}

Take $[c] \in H^{\mathrm{odd}}(\scrC)$ and our $B_p$, and consider the image of $B_p \otimes c^{\otimes p}$ under the operad action. This defines a map $H^{\mathrm{odd}}(\scrC) \rightarrow H^{\mathrm{odd}}(\scrC)$, which by Lemma \ref{th:explicit-cohen} is a certain component of the Cohen operation applied to $[c]$. Lemma \ref{th:cycle-found} tells us exactly what it is:
\begin{equation} \label{eq:which-cohen}
\begin{cases} \text{the $t^{\frac12}$ (or $\theta$) component of \eqref{eq:cohen-operations-2}} & p = 2, \\
\textstyle{\big(\frac{p-1}{2} !\big)}^{-1} \,\, \text{times the $t^{\frac{p-1}{2}}$-component of \eqref{eq:cohen-operations-2}}
& p > 2.
\end{cases}
\end{equation}
If $c$ has degree $1$, the same process computes $\beta_{\scrC}^{1,\dots,1}(c;\dots;c)$, by definition of \eqref{eq:beta-operations}. Lemma \ref{th:p-power-beta} shows that this is the leading order term of $\Pi^p_{\scrC}(\gamma,\dots,\gamma)$ for a Maurer-Cartan element $\gamma = qc + O(q^2)$, and Corollary \ref{th:r-at-least-3} identifies that with the $p$-fold product of $\gamma$ under our formal group law. The consequence, under the assumption of homological unitality inherited from the proof of Proposition \ref{th:semi-associativity}, is:

\begin{theorem} \label{th:abstract-counterpart}
Take the group law $\bullet$ on $\mathit{MC}(\scrC;q\bF_p[q]/q^{p+1})$. The $p$-th power map for that group fits into a commutative diagram like \eqref{eq:p-power-diagram}, with the operation \eqref{eq:which-cohen} at the bottom.
\end{theorem}

To conclude our discussion, note that if the operad structure is induced from one over Deligne-Mumford spaces, as in \eqref{eq:fm-to-dm-2}, then the relevant operation \eqref{eq:which-cohen} can also be written as a quantum Steenrod operation, by Lemma \ref{th:cohen-and-steenrod}.

\section{Constructions using pseudo-holomorphic curves\label{sec:floer}}
We will now translate the previous arguments into more specifically symplectic terms. The choice of singular chains on parameter spaces, and its application to a general operadic structure, is replaced by a choice of perturbations which make the moduli spaces regular, followed by counting-of-solutions to extract the algebraic operations. For technical convenience, we use Hamiltonian Floer theory (with a small time-independent Hamiltonian) as a model for cochains on our symplectic manifold. Correspondingly, all the operations are defined using inhomogeneous Cauchy-Riemann equations on punctured surfaces. This makes no difference with respect to Theorem \ref{th:p-power}, since the Floer-theoretic version of quantum Steenrod operations agrees with that defined using ordinary pseudo-holomorphic curves (for $p = 2$, see \cite{wilkins18b}; the same strategy works for general $p$).

\subsection{Floer-theoretic setup\label{subsec:floer-setup}}
Let $(X,\omega_X)$ be a closed symplectic manifold, which is monotone \eqref{eq:monotone}. We fix a $C^2$-small Morse function $H \in \smooth(X,\bR)$, with its Hamiltonian vector field $Z_H$. We also fix a compatible almost complex structure $J$, and the associated metric $g_J$. We require:

\begin{properties} \label{th:floer-properties}
(i) All spaces of Morse flow lines for $(H,g_J)$ are regular.

(ii) All one-periodic orbits of $Z_H$ are constant, hence critical points of $H$. Moreover, the linearized flow at each such point $x$ is nondegenerate for all times $t \in (0,1]$, which implies that the Conley-Zehnder index is equal to the Morse index $\mu(x)$.

(iii) No $J$-holomorphic sphere $v$ with $c_1(v) = 1$ passes through a critical point of $H$ or intersects an isolated Morse flow line. Here, $c_1(v)$ is the usual shorthand for $c_1(X)$ integrated over $[v] \in H_2(X)$.
\end{properties}

Consider the autonomous Floer equation, where as usual $S^1 = \bR/\bZ$:
\begin{equation} \label{eq:floer}
\left\{
\begin{aligned}
& u: \bR \times S^1 \longrightarrow X, \\
& \partial_s u + J(\partial_t u - Z_H) = 0, \\
& \textstyle \lim_{s \rightarrow \pm\infty} u(s,t) = x_{\pm},
\end{aligned}
\right.
\end{equation}
where the limits $x_{\pm}$ are critical points of $H$. This equation has an $(\bR \times S^1)$-symmetry by translation in both directions. For $S^1$-invariant solutions, meaning $t$-independent maps $u = u(s)$, it reduces to the negative gradient flow equation $du/ds + \nabla_{g_J} H = 0$. Denote by $D_u$ the linearized operator at a solution of \eqref{eq:floer}. Its Fredholm index can be computed as
\begin{equation}
\mathrm{index}(D_u) = \mu(x_-) - \mu(x_+) + 2 c_1(u),
\end{equation}
where in the last term, we have extended $u$ to the compactification $(\bR \times S^1) \cup \{\pm \infty\} = S^2$.  As a consequence of transversality results in \cite{hofer-salamon95, floer-hofer-salamon94} (see in particular \cite[Theorem 7.3]{hofer-salamon95} and \cite[Theorem 7.4]{floer-hofer-salamon94}), we may further require:

\begin{properties} \label{th:floer-properties-2}
(i) All solutions of Floer's equation with $\mathrm{index}(D_u) \leq 1$ are independent of $t$. Concretely, there are none with negative index; the only ones with index zero are constant; and those with $\mathrm{index}(D_u) = 1$ are isolated 
Morse flow lines for $(H,g_J)$. 

(ii) For the last-mentioned $u$, all solutions of $D_u \xi = 0$ are independent of $t$, hence lie in the kernel of the corresponding linearized operator from Morse theory. In view of Property \ref{th:floer-properties}(i) and (ii), this implies that such Floer solutions are regular.
\end{properties}

Define $\scrC$ to be the standard Morse complex for $(H,g_J)$, meaning that
\begin{equation}
\scrC = \bigoplus_x \bZ_x[-\mathrm{\mu(x)}],
\end{equation}
where $\bZ_x$ is the orientation line (the rank one free abelian group whose two generators correspond to orientations of the descending manifold of $x$), with a differential $d_{\scrC}$ that counts isolated gradient flow lines. When considered as a $\bZ/2$-graded space, this is equal to the Floer complex of $(H,J)$, thanks to the properties above. Our conventions are cohomological, meaning that with notation as in \eqref{eq:floer}, $d_{\scrC}$ takes ``$x_+$ to $x_-$''.

\subsection{Operations}
Take $\bar{\bC} = \bC P^1$, with marked points $z_0 = \infty$ and $z_1,\dots,z_d \in \bC$. Consider the resulting punctured surface, 
\begin{equation} \label{eq:surface}
C = \bar{\bC} \setminus \{z_0,\dots,z_d\} = \bC \setminus \{z_1,\dots,z_d\}. 
\end{equation}
An inhomogeneous term $\nu_C$ is a $(0,1)$-form on $C$ with values on vector fields on $X$:
\begin{equation} \label{eq:inhomogeneous-term}
\nu_C \in \Omega^{0,1}(C, \smooth(X,TX)) = \smooth(C \times X, \mathit{Hom}_{\bC}(\overline{TC},TX)),
\end{equation}
where the $(0,1)$ part is taken with respect to $J$. We require that our inhomogeneous terms should have a special structure near the marked points:
\begin{equation} \label{eq:kappa-at-zk}
\nu_C = 
\begin{cases}
\big(Z_H \otimes \mathrm{re} (d\log(z-z_k)/2\pi \sqrt{-1}) \big)^{0,1} & \text{ near $z_k$ for $k>0$,} \\
\big(Z_H \otimes \mathrm{re} (d\log(z-\xi_C)/2\pi \sqrt{-1}) \big)^{0,1} & \text{ near $z_0 = \infty$,}
\end{cases}
\end{equation}
where $\xi_C \in \bC$ is an auxiliary datum that we consider as part of $\nu_C$. In cylindrical coordinates 
\begin{equation} \label{eq:cylindrical}
z = \begin{cases}
z_k + \exp(-2\pi (s+\sqrt{-1}t)) & \text{ near $z_k$ for $k>0$, where $(s,t) \in \bR^{\geq 0} \times S^1$}, \\
\xi_C + \exp(-2\pi (s+\sqrt{-1}t)) & \text{ near $z_0 =\infty$, where $(s,t) \in \bR^{\leq 0} \times S^1$},
\end{cases}
\end{equation}
what \eqref{eq:kappa-at-zk} says is $\nu_C = (Z_H \otimes \mathit{dt})^{0,1}$. Consider the inhomogenous Cauchy-Riemann equation
\begin{equation} \label{eq:floer-2}
\left\{
\begin{aligned}
& u: C \longrightarrow X, \\
& \bar\partial u = \nu_C(u), \\
& \textstyle \lim_{z \rightarrow z_k} u(z) = x_k,
\end{aligned}
\right.
\end{equation}
where the limits $x_k$ are again critical points of $H$. When written in coordinates \eqref{eq:cylindrical} near the marked points, \eqref{eq:floer-2} reduces to \eqref{eq:floer}, explaining why this convergence condition makes sense. The linearization of \eqref{eq:floer-2} has 
\begin{equation}
\mathrm{index}(D_u) = \mu(x_0) - \textstyle \sum_{j=1}^d \mu(x_j) + 2c_1(u).
\end{equation}
We will not explain the compactness and transversality theory for moduli spaces of solutions of \eqref{eq:floer-2}, both being standard (the first due to monotonicity, the second because we have complete freedom in choosing $\nu_C$ on a compact part of $C$). For a single surface $C$ and a generic choice of $\nu_C$, counting solutions of \eqref{eq:floer-2} will give rise to a chain map $\scrC^{\otimes d} \rightarrow \scrC$ which preserves the $\bZ/2$-degree (and which represents the $d$-fold pair-of-pants product).

We need to review briefly the gluing process for surfaces, to see how it fits in with inhomogeneous terms. Suppose that we have two surfaces $C_k = \bC \setminus \{z_{k,1},\dots,z_{k,d_k}\}$, $k = 1,2$, which also come with inhomogeneous terms $\nu_{C_k}$, and in particular $\xi_{C_k} \in \bC$. Fix some $0 \leq i < d_1$. The gluing process produces a family of surfaces $C_{\lambda} = \bC \setminus \{z_{\lambda,1},\dots,z_{\lambda,d}\}$, where $d = d_1+d_2  - 1$, depending on a sufficiently small parameter $\lambda>0$. Namely, take the affine transformation
\begin{equation}
\phi_\lambda(z) = \lambda (z - \xi_{C_2}) + z_{1,i+1};
\end{equation}
then
\begin{equation} \label{eq:new-points}
z_{\lambda,k} = \begin{cases} 
z_{1,k} & k \leq i, \\
\phi_\lambda(z_{2,k-i}) & i < k \leq i+d_2, \\
z_{1,k-d_2+1} & k > i+d_2.
\end{cases} 
\end{equation} 
We want to equip the glued surfaces with inhomogeneous terms $\nu_{C_\lambda}$ which are smoothly dependent on $\lambda$ and have the following property. Fix some sufficiently small $r>0$ and large $R>0$. First,
\begin{equation} \label{eq:glue-nu}
\left\{
\begin{aligned} 
&
\nu_{C_\lambda} = \nu_{C_1} \;\; \text{where $|z| \geq R$, and also where $|z - z_{1,k}| \leq r$, for $0 < k \neq i+1$}; 
\\ &
\nu_{C_\lambda} = \phi_{\lambda,*}\nu_{C_2} \;\; \text{where $\lambda R \leq |z-z_{1,i+1}| \leq r$}; 
\\ &
\phi_\lambda^* \nu_{C_\lambda} = \nu_{C_2} \;\; \text{where $|z-z_{2,k}| \leq r$, for any $k>0$}.
\end{aligned}
\right.
\end{equation}
The upshot is that $\nu_{C_\lambda}$ is completely prescribed in certain (partly $\lambda$-dependent) neighbourhoods of the marked points on $C_\lambda$, as well as on an annular ``gluing region'', see Figure \ref{fig:glue-nu}. On each such region, $\nu_{C_\lambda}$ is given by a similar expression as in \eqref{eq:kappa-at-zk}. In particular, the middle line of \eqref{eq:glue-nu} really says that
\begin{equation}
\begin{aligned}
& \nu_{C_\lambda} = \big(Z_H \otimes \mathrm{re} (d\log(z-z_{1,i+1})/2\pi \sqrt{-1}) \big)^{0,1} 
\;\;  \text{where $\lambda R \leq |z-z_{1,i+1}| \leq r$} \\
 \Longleftrightarrow \;\; &
\phi_\lambda^*\nu_{C_\lambda} = \big( Z_H \otimes \mathrm{re}(d\log(z-\xi_{C_2})/2\pi \sqrt{-1}) \big)^{0,1}
\;\; \text{where $R \leq |z-\xi_{C_2}| \leq \lambda^{-1} r$.}
\end{aligned}
\end{equation}
\begin{figure}
\begin{centering}
\includegraphics{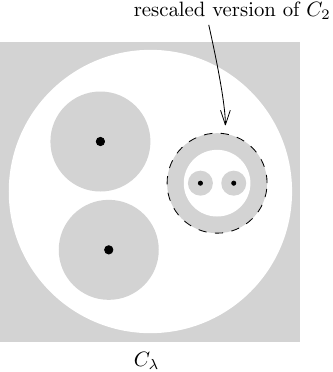}
\caption{\label{fig:glue-nu} Gluing punctured planes equipped with inhomogeneous terms, as in \eqref{eq:glue-nu}. The shaded regions are where the inhomogeneous term on the glued surface is prescribed.}
\end{centering}
\end{figure}
Additionally, there are asymptotic conditions as $\lambda \rightarrow 0$:
\begin{equation}
\left\{
\begin{aligned}
&
\text{on $|z-z_{1,i+1}| \geq r$, the family $\nu_{C_\lambda}$ can be smoothly extended to $\lambda = 0$, }
\\
& \hspace{20em} \text{by setting that extension equal to $\nu_{C_1}$;} \\
&
\text{on $|z| \leq R$, the family $\phi_{\lambda}^*\nu_{C_\lambda}$ can be smoothly extended to $\lambda =0$,}
\\
& \hspace{20em} \text{by setting that extension equal to $\nu_{C_2}$.}
\end{aligned}
\right.
\end{equation}
Given that, it makes sense for a sequence of solutions of \eqref{eq:floer-2} on $C_{\lambda_k}$, $\lambda_k \rightarrow 0$, to converge to a ``broken solution'' which consists of corresponding solutions on $C_1$ and $C_2$; and conversely, the gluing process for broken solutions applies; as used, for instance, in proving associativity of the pair-of-pants product. Again, we omit the details, which are standard. Thanks to our use of a autonomous Hamiltonian, there is also a version where $C_2$ is rotated before being glued in, meaning that we use a small $\lambda \in \bC^*$ (inserting absolute values wherever the size of $\lambda$ appears in the formulae above). 

A process such as \eqref{eq:new-points}, in which the $\lambda$-dependence of the marked points follows a specific pattern, is simple to describe, but far more rigid than the analytic arguments require. Here is a more appropriate formulation, where the first part describes the ingredient for compactness arguments, and the second part addresses gluing of solutions:

\begin{definition} \label{th:smoothing}
Take $C_1$ and $C_2$ as before. Choose arbitrary families of surfaces with inhomogeneous terms $C_{1,r}$ and $C_{2,r}$, depending on $r \in \bR^m$ for some $m$, and which reduce to the given ones for $r = 0$. Apply the previously described notion of gluing in a parametrized way, which means that we have a family $C_{\lambda,r}$. 

(i) Suppose that $C_k$ is a sequence of surfaces with inhomogeneous terms $\nu_{C_k}$, which for $k \gg 0$ are isomorphic to $C_{\lambda_k,r_k}$ for $\lambda_k > 0$ and $(\lambda_k,r_k) \rightarrow (0,0)$. We then say that the $C_k$ degenerate to $(C_1,C_2)$.

(ii) Suppose that $C_\sigma$ is a smooth family of surfaces with inhomogeneous terms $\nu_{C_\sigma}$, depending on a parameter $\sigma>0$. Suppose that for small $\sigma$, these are isomorphic to $C_{\lambda(\sigma),r(\sigma)}$, where $(\lambda(\sigma),r(\sigma))$ satisfies $(\lambda(0),r(0)) = (0,0)$ and $\lambda'(0) > 0$. We then say that the family $C_\sigma$ is obtained by smoothing $(C_1,C_2)$.
\end{definition}

To clarify the notation, is might be useful to look slightly ahead to our first application.
When defining an $A_\infty$-structure, one deals with $C_1$ and $C_2$ which depend, respectively, on moduli in $S_{d_1} \setminus \partial S_{d_1}$ and $S_{d_2} \setminus \partial S_{d_2}$. In these terms, $r$ is a local coordinate on the product of those spaces, while $\lambda$ is the transverse coordinate to $S_{d_1} \times S_{d_2} \subset \partial S_d$, $d = d_1+d_2-1$. As this example shows, our discussion has been limited to the simplest process of gluing two surfaces together; a complete description would include the generalization to arbitrarily many surfaces.

To round off the discussion of inhomogeneous terms, let's mention an obvious generalization, which is to equip \eqref{eq:surface} with a family of compatible almost complex structures $(J_z)_{z \in C}$, which reduce to the given $J$ outside a compact subset. When defining the associated notion of inhomogeneous term, one uses those structures to define the $(0,1)$ part, and similarly for \eqref{eq:floer-2}. It is straightforward to extend the gluing process to this situation. Usually, this generalization is not required, since the freedom to choose $\nu_C$ is already enough to achieve transversality of moduli spaces. However, there are situations such as the construction of continuation maps, where varying almost complex structures occur necessarily (because one is trying to relate different choices of $J$).

\subsection{The quantum $A_\infty$-structure\label{subsec:quantum-ainfty}}
This is the most familiar application. Given $(s_1,\dots,s_d)$ as in \eqref{eq:configuration-space-1}, we think of them as complex points $z_k = s_k$, and then equip the resulting surface \eqref{eq:surface} with an inhomogeneous term $\nu_C$, which should vary smoothly in dependence on the points, and be invariant under the symmetries in \eqref{eq:configuration-space-1}; one can think of this as a fibrewise inhomogeneous term on the universal family of surfaces over $S_d \setminus \partial S_d$. Along the boundary of the moduli space, we want the family to extend along the lines indicated in Definition \ref{th:smoothing}(ii). Of course, on a boundary stratum of codimension $k$, one has $k$ components that are being glued together, and the definition should be adapted accordingly. The outcome is a parametrized moduli space, which consist of points of $S_d \setminus \partial S_d$ together with a a solution of \eqref{eq:floer-2} on the associated surface. For generic choices, these parametrized moduli spaces are regular. Moreover, they are oriented relative to the orientation spaces at limit points, meaning that a choice of isomorphism $\bZ_{x_k} \iso \bZ$, $k = 0,\dots,d$, determines an orientation of the parametrized moduli spaces. A signed count of points in the zero-dimensional moduli spaces, with auxiliary signs as in \eqref{eq:weird-signs}, yields operations $\mu^d_{\scrC}$ for $d \geq 2$, which one combines with the Floer differential $\mu^1_{\scrC} = -d_{\scrC}$ to form the ($\bZ/2$-graded) quantum $A_\infty$-structure. 

One  can adapt the arguments from Section \ref{subsec:interval} to show that the quantum $A_\infty$-structure is, in a suitable homotopical sense, independent of all choices, including the Floer differential. Suppose that we have $(H,J)$ and $(\tilde{H},\tilde{J})$, leading to chain complexes $(\scrC,d_{\scrC})$ and $(\tilde{\scrC},d_{\tilde{\scrC}})$. For each of the two, we make the choices of inhomogeneous terms required to build the $A_\infty$-structures, denoted by $\mu_{\scrC}$ and $\mu_{\tilde\scrC}$. To relate them, we start by picking a third version of the chain complex, denoted by $(\check\scrC,d_{\check\scrC})$, based on some $(\check{H},\check{J})$. Next, we introduce maps 
\begin{equation} \label{eq:make-a-bimodule}
\phi^{p,1,q}_{\check{\scrC}}: \scrC^{\otimes p} \otimes \check{\scrC} \otimes \tilde{\scrC}^{\otimes q} \longrightarrow \check{\scrC}[1-p-q],
\end{equation}
with $\phi^{0,1,0}_{\check{\scrC}} = -d_{\check{\scrC}}$, which turn $\check{\scrC}[1]$ into an $A_\infty$-bimodule, with $\mu_{\scrC}$ acting on the left and $\mu_{\tilde\scrC}$ on the right. This is analogous to \eqref{eq:fg-operations}, except that the conditions on $\phi^{p,1,0}$ and $\phi^{0,1,q}$ which we imposed there are no longer satisfied. The geometric construction of \eqref{eq:make-a-bimodule} involves another family of inhomogeneous terms over $S_d \setminus \partial S_d$, $d = p+1+q$. Those terms are modelled on $H$ near $(z_1,\dots,z_p)$, on $\tilde{H}$ near $(z_{p+2},\dots,z_d)$, and on $\check{H}$ near the remaining points $(z_0,z_{p+1})$. Similarly, our surfaces carry varying families of complex structures. The behaviour under degeneration to $\partial S_d$ follows the pattern from Figure \ref{fig:marked-tree}, with some components of the limit carrying the inhomogeneous terms that define the two $A_\infty$-ring structures, and others, the $A_\infty$-bimodule structure; we have represented this in a more geometric way in Figure \ref{fig:colored-bubble}.
\begin{figure}
\begin{centering}
\includegraphics{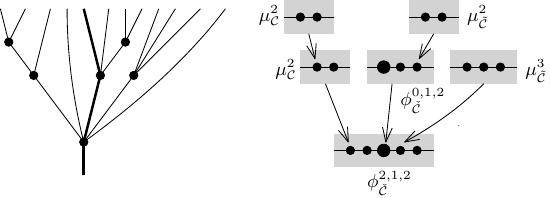}
\caption{\label{fig:colored-bubble}The $A_\infty$-bimodule operations \eqref{eq:make-a-bimodule}. We show the behaviour of the inhomogeneous terms on a sample (codimension $5$) boundary stratum, for $(p,q) = (4,7)$.}
\end{centering}
\end{figure}%

At this point, we add continuation maps to the mix. These arise from the configuration $(z_0 = \infty,\, z_1 = 0)$, meaning the surface $C = \bC^*$. In our application, the behaviour at $z_0$ is always given by $(\check{H},\check{J})$, and that at $z_1$ by either $(H,J)$ or $(\tilde{H},\tilde{J})$. The outcome are two chain maps
\begin{equation} \label{eq:two-continuation}
\xymatrix{
\scrC \ar[rr]^-{\psi^{1,0}_{\check\scrC}} && \check{\scrC} &&
\ar[ll]_-{\psi^{0,1}_{\check\scrC}} \tilde{\scrC}.
}
\end{equation}
Our sign conventions are nonstandard: on cohomology, $\psi^{0,1}$ induces the canonical isomorphism between Floer cohomology groups, whereas $\psi^{1,0}$ has the opposite sign. Extending \eqref{eq:two-continuation}, we want to build operations 
\begin{equation} \label{eq:new-psi}
\psi_{\check\scrC}^{p,q}: \scrC^{\otimes p} \otimes \tilde{\scrC}^{\otimes q} \longrightarrow \check{\scrC}[1-p-q], \;\; d = p+q > 0,
\end{equation}
which unlike their counterparts in \eqref{eq:fg-operations-2} are defined even if $p$ or $q$ are zero, and (that being taken into account) satisfy the same kind of relation \eqref{eq:g-relation-algebra}. Geometrically, the parameter space underlying \eqref{eq:new-psi} is no longer $[0,1] \times S_d$ as in Section \ref{subsec:interval}, but instead $S_{d+1}$, where we think of having inserted an additional point $s_\dag$, with $s_p < s_\dag < s_{p+1}$, into \eqref{eq:configuration-space-1}. The orientation is that associated to ordered configurations $(s_1,\dots,s_p,s_\dag,\dots,s_{p+q})$ multiplied by $(-1)^p$. When forming the associated surface \eqref{eq:surface}, we do not equip it with a puncture corresponding to $s_\dag$: the position of that point just serves as an additional modular variable. The inhomogeneous terms and almost complex structures are determined by $(H,J)$ near $z_1,\dots,z_p$; by $(\tilde{H},\tilde{J})$ near $z_{p+1},\dots,z_{p+q}$; and by $\check{H}$ near $\tilde{z}_0$. In the limit as we approach a point of $\partial S_{d+1}$, the screen containing $s_\dag$ corresponds to a component surface which carries data underlying \eqref{eq:new-psi}, while the other components have data underlying the $A_\infty$-ring structures or the $A_\infty$-bimodule structure.

\begin{example}
Consider the cases where $p+q = 2$. The algebraic relations are
\begin{equation}
\left.
\begin{aligned}
\psi^{1,0}_{\check\scrC}(\mu^2_{\scrC}(c_1,c_2))  - \phi^{1,1,0}_{\check\scrC}(c_1; \psi^{1,0}_{\check\scrC}(c_2)) 
\\ 
-\phi^{0,1,1}_{\check\scrC}(\psi^{1,0}_{\check\scrC}(c_1);\tilde{c}_2) - \phi^{1,1,0}_{\check\scrC}(c_1; \psi^{0,1}_{\check\scrC}(\tilde{c}_2)) 
\\  
\psi^{0,1}_{\check\scrC}(\mu^2_{\tilde\scrC}(\tilde{c}_1,\tilde{c}_2)) - \phi^{0,1,1}_{\check\scrC}(\psi^{0,1}_{\check\scrC}(\tilde{c}_1);\tilde{c}_2) 
\end{aligned}
\right\} =
(\text{\it terms involving differentials}).
\end{equation}
On the cohomology level, consequence is that if the classes $[c_k]$ and $[\tilde{c}_k]$ correspond to each other under canonical isomorphisms, meaning that $[\psi^{1,0}_{\check\scrC}(c_k)] = -[\psi^{0,1}_{\check\scrC}(\tilde{c}_k)]$, then their products inherit the same property:
\begin{equation}
\begin{aligned}
&
[\psi_{\check\scrC}^{1,0}(\mu^2_{\scrC}(c_1,c_2))] = [\phi_{\check\scrC}^{1,1,0}(c_1;\psi_{\check\scrC}^{1,0}(c_2))] = -[\phi_{\check\scrC}^{1,1,0}(c_1;\psi_{\check\scrC}^{0,1}(\tilde{c}_2))]
\\ & \qquad = [\phi_{\check\scrC}^{0,1,1}(\psi_{\check\scrC}^{1,0}(c_1);\tilde{c}_2)] =
-[\phi_{\check\scrC}^{0,1,1}(\psi_{\check\scrC}^{0,1}(\tilde{c}_1);\tilde{c}_2)] =
-[\psi_{\check\scrC}^{0,1}(\mu^2_{\tilde\scrC}(\tilde{c}_1,\tilde{c}_2))].
\end{aligned}
\end{equation}
\end{example}

\begin{example} \label{th:new-psi-12}
For $(p,q) = (1,2)$, the algebraic relation is
\begin{equation}
\begin{aligned}
& -\phi_{\check{\scrC}}^{0,1,2}(\psi^{1,0}_{\check\scrC}(c_1);\tilde{c}_2,\tilde{c}_3)
-\phi_{\check{\scrC}}^{0,1,1}(\psi^{1,1}_{\check\scrC}(c_1;\tilde{c}_2);\tilde{c}_3)
-\phi_{\check{\scrC}}^{1,1,1}(c_1;\psi^{0,1}_{\check\scrC}(\tilde{c}_2);\tilde{c}_3) \\ &
-\phi_{\check{\scrC}}^{1,1,0}(c_1;\psi^{0,2}_{\check\scrC}(\tilde{c}_2,\tilde{c}_3))
+ (-1)^{\|c_1\|} \psi_{\check\scrC}^{1,1}(c_1;\mu^2_{\tilde{\scrC}}(\tilde{c}_2,\tilde{c}_3))
= (\text{\it terms involving differentials}).
\end{aligned}
\end{equation}
Figure \ref{fig:psi-geometric} shows the relevant degenerations, corresponding to the boundary faces of $S_4$.
\end{example}
\begin{figure}
\begin{centering}
\includegraphics{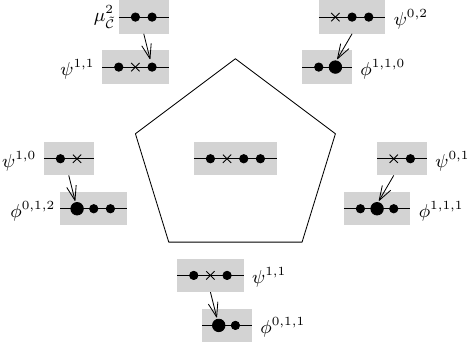}
\caption{\label{fig:psi-geometric}The construction of the maps \eqref{eq:new-psi}, in the case from Example \ref{th:new-psi-12}. The cross marks the additional point (not a puncture of the associated surface).}
\end{centering}
\end{figure}%

Using \eqref{eq:make-a-bimodule} and \eqref{eq:new-psi}, we define an $A_\infty$-structure on 
\begin{equation} \label{eq:new-h}
\scrH = (\scrC \otimes \bZ u) \oplus (\tilde{\scrC} \otimes \bZ\tilde{u}) \oplus (\check{\scrC} \otimes \bZ v),
\end{equation}
where the symbols $u$, $\tilde{u}$ and $v$ have degrees as in \eqref{eq:nc-interval}. The definition is a modified version of \eqref{eq:tensor-structure}. The differential is
\begin{equation}
\begin{aligned}
& \mu^1_{\scrH}(c \otimes u) = \mu^1_{\scrC}(c) \otimes u + (-1)^{\|c\|} \psi_{\check\scrC}^{1,0}(c) \otimes v, \\
& \mu^1_{\scrH}(\tilde{c} \otimes \tilde{u}) = \mu^1_{\tilde\scrC}(\tilde{c}) \otimes \tilde{u} +
(-1)^{\|\tilde{c}\|} \psi_{\check\scrC}^{0,1}(\tilde{c}) \otimes v, \\
& \mu^1_{\scrH}(\check{c} \otimes v) = \phi_{\check\scrC}^{0,1,0}(\check{c}) \otimes v.
\end{aligned}
\end{equation}
and we similarly change the higher $A_\infty$-operations in the case when the input consists of only terms from either $\scrC$ or $\tilde{\scrC}$:
\begin{equation}
\begin{aligned}
& \mu^d_{\scrH}(c_1 \otimes u, \dots, c_d \otimes u) = \mu^d_{\scrC}(c_1,\dots,c_d) \otimes u + (-1)^{\maltese_d} \psi^{d,0}_{\check\scrC}(c_1,\dots,c_d) \otimes v, \\
& \mu^d_{\scrH}(\tilde{c}_1 \otimes \tilde{u}, \dots, \tilde{c}_d \otimes \tilde{u}) = \mu^d_{\tilde\scrC}(\tilde{c}_1,\dots,\tilde{c}_d) \otimes \tilde{u} + (-1)^{\maltese_d} \psi^{0,d}_{\check\scrC}(\tilde{c}_1,\dots,\tilde{c}_d) \otimes v.
\end{aligned}
\end{equation}
As before, the projection maps from \eqref{eq:new-h} to $\scrC$ or $\tilde{\scrC}$ are compatible with $A_\infty$-ring structures, and are chain homotopy equivalences. This implies the desired well-definedness statement for the $A_\infty$-structure, as in \eqref{eq:a-tilde-a}.

\subsection{An alternative strategy for proving independence\label{subsec:alt}}
The approach to well-definedness of the quantum $A_\infty$-structure adopted above involves additional families of Riemann surfaces, leading to the larger $A_\infty$-ring $\scrH$ which serves as an intermediate object. Alternatively, as we will now explain, one can enlarge the target symplectic manifold.

Let's start by looking at a toy model, namely the symplectic manifold $S^2$.
\begin{equation} \label{eq:s2-terms}
\parbox{38em}{
Choose $(H_{S^2},J_{S^2})$ as in Section \ref{subsec:floer-setup}, satisfying the following additional technical condition. At a local maximum or minimum of $H_{S^2}$, the Hessian is $J_{S^2}$-invariant. This means that there are local $J_{S^2}$-holomorphic coordinates centered at that point, in which $H_{S^2}(y) = (\text{\it constant}) \pm |y|^2 + O(|y|^3)$. When choosing inhomogeneous terms, we also require that they should be zero at the local maxima and minima of $H_{S^2}$.
}
\end{equation}
As a consequence of the condition on inhomogeneous terms, the constant map at a local minimum or maximum will be a solution of \eqref{eq:floer-2}. One can use a counting-of-zeros argument for solutions of linear Cauchy-Riemann type operators on line bundles to show that the constant maps at minima have injective linearizations, and hence (since they have index $0$) are regular. A similar argument, applied to \eqref{eq:floer-2} itself rather than its linearization, shows the following:
 
\begin{lemma} \label{th:u-degree}
(i) Let $p \in S^2$ be a local maximum. For any solution of \eqref{eq:floer-2} on $S^2$, which is not constant equal to $p$, the homology class $[u] \in H_2(S^2) = \bZ$ satisfies
\begin{equation}
[u] \geq \# \{1 \leq i \leq d\;:\; x_i = p \} + \# \{ u^{-1}(p) \}.
\end{equation}

(ii) Let $p \in S^2$ be a local minimum. For any solution of \eqref{eq:floer-2} on $S^2$, with $x_0 = p$, and which is not constant equal to $p$, we have
\begin{equation}
[u] \geq 1 + \# \{u^{-1}(p) \}.
\end{equation}
\end{lemma}

Take the graded abelian group obtained by using only critical points of $H_{S^2}$ which have index $\leq 1$ as generators. We want to equip this with a version of the quantum $A_\infty$-structure, which uses inhomogeneous terms as in \eqref{eq:s2-terms}, but only considers solutions with degree $[u] = 0$. The argument showing that this works consists of three steps. First, Lemma \ref{th:u-degree}(i) implies that for all solutions, one has $[u] \geq 0$. It follows that if we consider a sequence of maps of degree zero which converges to a limit with several pieces, then each piece must again have degree zero. Suppose that our original sequence consisted of maps whose limits are critical points of index $\leq 1$. Our second point is that then, no critical point of index $2$ can appear in the limit, since it would cause one of the pieces to have positive degree, again by Lemma \ref{th:u-degree}(i). Thirdly, transversality of moduli spaces is unproblematic except possibly for the constant solutions at local minima; but we already know that such solutions are regular (in the ordinary sense of considering a fixed $C$, and therefore in the parametrized sense as well). We want to point out two properties of this $A_\infty$-structure: the maps involved stay away from the local maxima, because of Lemma \ref{th:u-degree}(i); and if $u$ is a map that contributes to it, and whose limit $x_0$ is a local minimum, the map must actually be constant, by Lemma \ref{th:u-degree}(ii) and the degree requirement.

Take a monotone symplectic manifold $X$. For each minimum or maximum $p$ of $H_{S^2}$, we choose a Morse function and almost complex structure on $X$, written as $(H_{X,p}, J_{X,p})$. On the product $X \times S^2$, we then proceed as follows.
\begin{equation}
\parbox{38em}{
Take a Hamiltonian $H_{X \times S^2}$ and almost complex structure $J_{X \times S^2}$, satisfying our usual conditions, and with the following additional properties. In a local $J_{S^2}$-holomorphic coordinate on $S^2$ around a local minimum or maximum $p$, we have $H_{X \times S^2} = H_{S^2} + H_{X,p} + O(|y|^3)$, and similarly $J_{X \times S^2} = J_{S^2} \times J_{X,p} + O(|y|^2)$. When we choose inhomogeneous terms, they should have the property that, when restricted to $X \times \{p\}$, they take values in vector fields tangent to that submanifold.
}
\end{equation}
As a consequence of this, we can have solutions of the associated equation \eqref{eq:floer-2} which are contained in $X \times \{p\}$. If $p$ is a local minimum, then any such solution is regular in $X \times S^2$ iff it is regular inside $X \times \{p\}$. The counterpart of Lemma \ref{th:u-degree} for $X \times S^2$, proved by projecting to $S^2$ and arguing as before, is:

\begin{lemma} \label{th:u-degree-2}
(i) Let $p \in S^2$ be a local maximum. For any solution of \eqref{eq:floer-2} on $X \times S^2$, which is not contained in $X \times \{p\}$, the homology class $[u] \in H_2(X \times S^2)$ satisfies
\begin{equation}
[X] \cdot [u] \geq \# \{1 \leq i \leq d\;:\; x_i \in X \times \{p\} \} + \# \{ u^{-1}(X \times \{p\}) \}.
\end{equation}

(ii) Let $p \in S^2$ be a local minimum. For any solution of \eqref{eq:floer-2} on $X \times S^2$, with $x_0 \in X \times \{p\}$, and which is not contained in $X \times \{p\}$, we have
\begin{equation}
[X] \cdot [u] \geq 1 + \# \{u^{-1}(X \times \{p\}) \}.
\end{equation}
\end{lemma}

For each local minimum $p$, we make choices of inhomogeneous terms which, building on the previously chosen $(H_{X,p}, J_{X,p})$, yield a quantum $A_\infty$-structure $\scrC_p$. On $X \times S^2$, we then make corresponding choices, which restrict to the previous ones on $X \times \{p\}$ for each local minimum $p$. When building the corresponding version of the $A_\infty$-structure on $X \times S^2$, denoted by $\scrK$, we use only those critical points of $H_{X \times S^2}$ which do not lie on $X \times \{p\}$ for a local maximum $p$; and only maps $u$ with $[X] \cdot [u] = 0$. This works for exactly the same reasons as in the previously considered toy model case. Moreover, the following two properties hold: those maps that contribute avoid the subsets $X \times \{p\}$, where $p$ is a local maximum; and projection to the subgroup generated by critical points in $X \times \{p\}$, where $p$ is a local minimum, is a map
\begin{equation} \label{eq:p-projection}
\scrK \longrightarrow \scrC_p
\end{equation}
compatible with the $A_\infty$-structure. At this point, we specialize to functions $H_{S^2}$ that have exactly one local maximum, but possibly several local minima. By looking at the Morse theory of $H_{X \times S^2}$, one sees that the projections \eqref{eq:p-projection} are chain homotopy equivalence. By looking at those maps for two local minima, one relates the $A_\infty$-structures $\scrC_p$ for different $p$.

\subsection{The formal group structure}
Take the parameter spaces \eqref{eq:mww-spaces}. We think of the interior of this space as parametrizing a family of punctured planes, which degenerate along the boundary. This is essentially constructed as in \eqref{eq:mww-into-fm}, but with two differences. First of all, we do include the spaces $\mathit{MWW}_{0,\dots,1,\dots,0}$, to which we associate a once-punctured plane (a cylinder) with an inhomogeneous term, which is that defining the Floer differential. Since we are not dividing by translation, the only isolated point in the associated moduli space is a stationary solution at a critical point of the Hamiltonian, and that is the geometric origin of \eqref{eq:010}. Hence, to each ``screen'' in the limit corresponds a surface (unlike our original construction \eqref{eq:mww-into-fm}, where some of the screens were collapsed). The second difference is that we need everything to depend smoothly on parameters (the original construction was purely topological, hence allowed us to get away with continuity). More precisely, near the codimension $1$ boundary points of $\mathit{MWW}_{d_1,\dots,d_r}$, we really need a situation as in Definition \ref{th:smoothing}(ii); but along the codimension $>1$ points, all we need is the situation from Definition \ref{th:smoothing}(i), since those points only appear in compactness arguments. In any case, given the structure of $\mathit{MWW}_{d_1,\dots,d_r}$ as a smooth manifold with generalized corners, it is unproblematic to define the required notion of smoothness, and to construct families of inhomogeneous terms satisfying it (by induction on dimension). The outcome are operations as in \eqref{eq:beta-operations}, with the difference that \eqref{eq:010} is now a geometric statement rather than a separately imposed condition. One can therefore define \eqref{eq:multiproduct} for the quantum $A_\infty$-structure, and Lemmas \ref{th:multiproduct}--\ref{th:inverse-4} carry over immediately. What's important for applications is that we can, if desired, choose the inhomogenous terms to be compatible with forgetting any color that has no marked points belonging to it; and therefore, to make our operations satisfy \eqref{eq:cancel-0}.

Well-defineness of \eqref{eq:multiproduct} can be proved by a combination of the approaches from Sections \ref{subsec:hideous} and \ref{subsec:quantum-ainfty}. Namely, suppose first that we have $A_\infty$-rings $\scrC_0,\dots,\scrC_r$, each defined by a separate choice of function and other data. One can generalize \eqref{eq:beta-operations} to obtain a map
\begin{equation} \label{eq:beta-again}
\mathit{MC}(\scrC_1;N) \times \cdots \times \mathit{MC}(\scrC_r;N) \longrightarrow \mathit{MC}(\scrC_0;N).
\end{equation}
Now, we want to change the $A_\infty$-structure on $\scrC_0$ and one of the $\scrC_{k+1}$. The new versions are related to the old ones by larger $A_\infty$-rings $\scrH_0$ and $\scrH_{k+1}$, constructed as in the uniqueness argument from Section \ref{subsec:quantum-ainfty}. The main tool in analyzing this change is an analogue of the middle $\rightarrow$ in \eqref{eq:interpolate}, which is a map
\begin{equation}
\mathit{MC}(\scrC_1;N) \times \cdots \times \mathit{MC}(\scrH_{k+1};N)
\times \cdots \times \mathit{MC}(\scrC_r;N) \longrightarrow \mathit{MC}(\scrH_0;N).
\end{equation}
The definition of this involves two kinds of parameter spaces. The first ones are again the $\mathit{MWW}_{d_1,\dots,d_r}$, but where we single out one of the $d_{k+1}$ points of color $k+1$, as in the definition of \eqref{eq:make-a-bimodule}, for special treatement when constructing the inhomogeneous terms and almost complex structures. The second class of parameter spaces are $\mathit{MWW}_{d_1,\dots,d_{k+1}+1,\dots,d_r}$, which have an additional point of color $k+1$ (more precisely, there is one such space for every possible position of the additional point with respect to the other $d_k$). That point will not correspond to a puncture of the resulting Riemann surface; we just use its position as a modular variable, following the idea from \eqref{eq:new-psi}. Of course, the additional marked point can in principle split off by itself into a mid-scale screen; when constructing the Riemann surface, that screen will not correspond to a component. We omit the details entirely. 

\begin{example} \label{th:comp-cont}
The simplest example of a parameter space of the second kind is $\mathit{MWW}_{1+1}$, where the additional marked point could be placed either on the left or right. The context in this case is that we have two versions of \eqref{eq:new-h}, namely $\scrH_k = \scrC_k u \oplus \tilde{\scrC}_k \tilde{u} \oplus \check{\scrC}_k v$ for $k = 0,1$. As part of their $A_\infty$-structure, we have continuation maps $\scrC_k \rightarrow \check{\scrC}_k$. On the other hand, as part of the $r = 1$ case of \eqref{eq:beta-again}, we have constructed continuation maps $\scrC_1 \rightarrow \scrC_0$, $\tilde{\scrC}_1 \rightarrow \tilde{\scrC}_0$, $\check{\scrC}_1 \rightarrow \check{\scrC}_0$. The two versions of our moduli space then yield chain homotopies between compositions of those continuation maps, drawn as dashed arrows here:
\begin{equation}
\xymatrix{
\ar@{-->}[dr]
\scrC_1 \ar[d] \ar[r] & \scrC_0 \ar[d] \\
\check{\scrC}_1 \ar[r] & \check{\scrC}_0 
}
\qquad \qquad
\xymatrix{
\ar@{-->}[dr]
\tilde\scrC_1 \ar[d] \ar[r] & \tilde\scrC_0 \ar[d] \\
\check{\scrC}_1 \ar[r] & \check{\scrC}_0 
}
\qquad
\end{equation}
The geometry behind the construction on the left is shown schematically in Figure \ref{fig:compose-continuation}.
\end{example}
\begin{figure}
\begin{centering}
\includegraphics{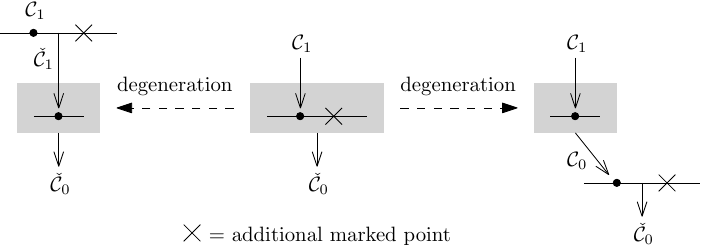}
\caption{\label{fig:compose-continuation}The construction from Example \ref{th:comp-cont}.}
\end{centering}
\end{figure}

\begin{remark}
Alternatively, one could prove the well-definedness of the maps $\Pi^r_{\scrC}$ using the approach from Section \ref{subsec:alt}, which means defining a corresponding structure using $X \times S^2$, which comes with quasi-isomorphic projections to different copies of $X$.
\end{remark}

The spaces $\mathit{SS}$ from \eqref{eq:strip-shrinking}, while more complicated, show the same geometric behaviour as $\mathit{MWW}$. Hence, the same kind of argument allows the proof of Proposition \ref{th:semi-associativity} to carry over, which completes our discussion of Proposition \ref{th:main}. There is a minor point which may be worth mentioning: in Section \ref{subsec:strip-shrinking}, we added two marked points in the definition of the map \eqref{eq:double-stabilization}, whose purpose was to break the symmetries of Fulton-MacPherson space. In a pseudo-holomorphic curve context, we treat the extra points as in the well-definedness arguments above, meaning that while their position gives additional modular variables on which the inhomogeneous term depends.

\subsection{Commutativity}
Adapting the arguments from Section \ref{subsec:commutativity}, we will now prove Proposition \ref{th:commutative}. This is the first time that one of the features of our Floer-theoretic setup, namely the time-independence of the Hamiltonians and almost complex structures, and the resulting $S^1$-symmetry of \eqref{eq:floer}, will be used in a substantial manner.

Throughout the following discussion, it is assumed that choices of inhomogeneous terms have been done so as to satisfy \eqref{eq:cancel-0}. Let's start with the moduli space underlying $\beta_{\scrC}^{1,1}$. It involves a family of surfaces depending on one parameter, which we denote by $C_s = \bC \setminus \{z_1(s),z_2(s)\}$, $s \in \bR$. One can assume that this family is symmetric outside a compact parameter range, in the sense that for some $S > 0$,
\begin{equation}
(z_1(-s),z_2(-s)) = (z_2(s),z_1(s)) \;\; \text{if $|s| \geq S$,}
\end{equation}
and that the inhomogeneous terms are chosen compatibly with this symmetry. As a consequence, there is partial cancellation between the two moduli spaces that enter into \eqref{eq:beta-difference}, with the parts having $|s| \geq S$ contributing only cancelling pairs of points. One can therefore say that \eqref{eq:beta-difference} is computed by a single parametrized moduli space, whose compact parameter space is a circle, obtained by gluing together the endpoints of two intervals $[-S,S]$. If we parametrize this circle by $r \in \bR/2\pi\bZ$ compatibly with its orientation, then that family of surfaces can be deformed to the simple form
\begin{equation} \label{eq:independent}
(z_1(r), z_2(r)) = (\exp(r\sqrt{-1}), -\exp(r\sqrt{-1})).
\end{equation}
Up to rotation, this is independent of $r$, and (it is here that we use time-independence) one can choose an inhomogeneous term to be compatible with that; in which case, the moduli space cannot have any isolated points, hence contributes zero. The deformation which ends up with \eqref{eq:independent}, which can be thought as a family of surfaces parametrized by a compact two-dimensional disc, therefore gives rise to a nullhomotopy \eqref{eq:first-kappa}. As in our previous discussion of  \eqref{eq:gamma12-21}, this implies commutativity of the formal group structure mod $N^3$, which is the first part of Proposition \ref{th:commutative}.
\begin{figure}
\begin{centering}
\includegraphics{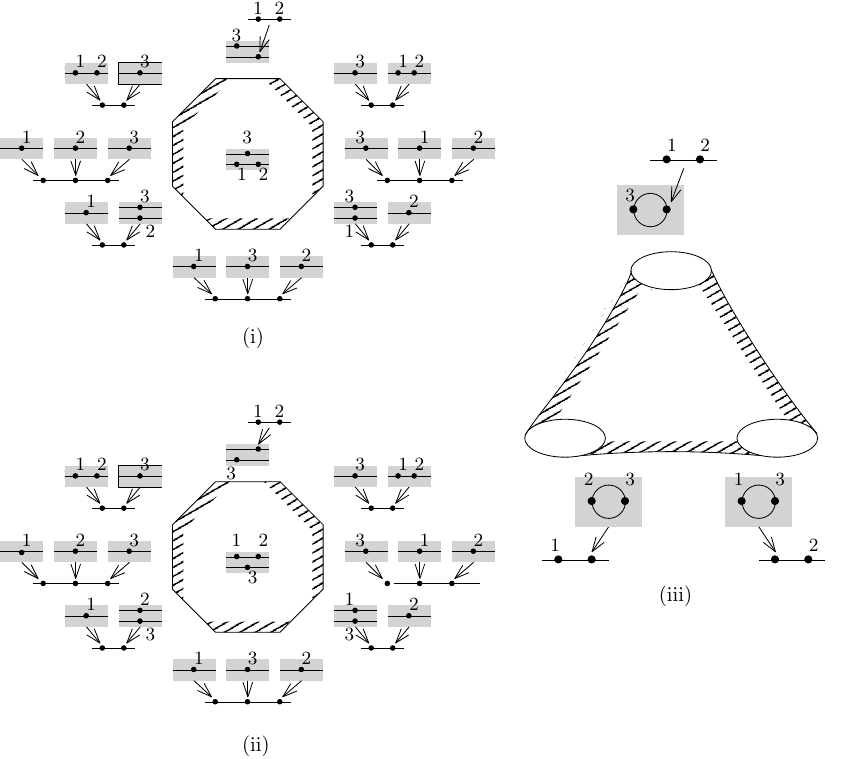}
\caption{\label{fig:kindergarten}The geometry underlying \eqref{eq:beta-difference-2}: (i) $\beta_{\scrC}^{2,1}(c_1,c_2;c_3)$, and (ii) $\beta_{\scrC}^{1,2}(c_3;c_1;c_2)$. Since part of the structure agrees, we can remove the hatched regions and join the rest together, with the outcome shown in (iii) after removing the ``trivial screens'' that have only one marked point. The pairs of points drawn as lying on a circle rotate around each other once in dependence on the parameter.s}
\end{centering}
\end{figure}%

Next, let's look at part of the formula \eqref{eq:k21},
\begin{equation} \label{eq:beta-difference-2}
\beta_{\scrC}^{2,1}(c_1,c_2;c_3) - (-1)^{\|c_3\|(\|c_1\|+\|c_2\|)} \beta_{\scrC}^{1,2}(c_3;c_1,c_2). 
\end{equation}
The underlying moduli spaces are two copies of the octagon from Figure \ref{fig:8-gon}. Five of the boundary sides of those octagons match up in pairs which carry the same inhomogeneous term. We may assume that this extends to a neighbourhood of those sides. As far as counting points in zero-dimensional moduli spaces is concerned, we can then cut out suitably matching neighbourhoods and glue the rest together. The outcome of this process, shown in Figure \ref{fig:kindergarten}, is that our expression can be computed by a single moduli space parametrized by a compact pair-of-pants surface. Moreover, along each boundary circle, we find that one of the components is a copy of the family of surfaces underlying \eqref{eq:beta-difference}, and which can be therefore filled in with a family parametrized by a disc. As a consequence, we find that the operation $K^{2,1}_{\scrC}$ defined in \eqref{eq:k21} is given by a family of surfaces parametrized by $S^2$. One can further deform that family so that degenerations happen only along three points (instead of the previous three discs) in the parameter space.

The outcome is that we have a family of four-punctured spheres, parametrized by $S^2$. Inspection of Figure \ref{fig:kindergarten} shows that this family has degree $1$ in $H_2(\mathit{DM}_3) \iso \bZ$. This is a Floer-theoretic implementation of the four-pointed Gromov-Witten invariant, which we can relate to the standard version by a gluing argument as in \cite{piunikhin-salamon-schwarz94}. As a consequence, identifying $H^*(\scrC) = H^*(X;\bZ)$, we have that on the cohomology level,
\begin{equation} \label{eq:gw4}
\int_X x_0\, K^{2,1}_{\scrC}(x_1,x_2;x_3) = \langle x_0,x_1,x_2,x_3 \rangle_4, \quad
x_k \in H^*(X;\bZ)
\end{equation}
(see \eqref{eq:gw} for our notational conventions in Gromov-Witten theory). For the particular case of $K^{2,1}_{\scrC}(x_1,x_1;x_2)$, where $x_1$ has odd degree, the graded symmetry of Gromov-Witten invariants means that $\langle x_0,x_1,x_1,x_2 \rangle_4 = 0$. Assuming additionally that $H^*(X;\bZ)$ is torsion-free, it follows that $K^{2,1}_{\scrC}(x_1,x_1;x_2)$ itself is zero. As in Proposition \ref{th:commutative-2}, this and the corresponding argument for $K^{1,2}$ imply the desired commutativity statement modulo $N^4$.

\subsection{The $p$-th power map}
We now carry over the required arguments from Sections \ref{sec:operations} and \ref{sec:power}, leading to the proof of Theorem \ref{th:p-power} as the Floer-theoretic analogue of Theorem \ref{th:abstract-counterpart}.

We can bring Floer-theoretic constructions closer to the abstract operadic framework, by making a generic choice of inhomogeneous terms which are parametrized by $\mathit{FM}_d$. In the interior, this means that for every complex configuration $(z_1,\dots,z_d)$ we choose an inhomogeneous term $\nu_C$ on the resulting surface \eqref{eq:surface}, in a way which is compatible with the action of the automorphisms which appear in \eqref{eq:configuration-space-2}. We then ask that this should extend to the ``screens'' associated to points in $\partial \mathit{FM}_d$, in a way which enables compactness arguments for boundary strata of any dimension, see Definition \ref{th:smoothing}(i), and gluing for codimension $1$ boundary strata, see Definition \ref{th:smoothing}(ii). We also ask that our choices should be $\mathit{Sym}_d$-equivariant (recall that the symmetric group acts freely on $\mathit{FM}_d$).

Suppose that we have maps \eqref{eq:mww-into-fm} based on smooth functions \eqref{eq:tau-functions}. By pullback, our previous choice induces a family of inhomogeneous terms parametrized by $\mathit{MWW}_{d_1,\dots,d_r}$. For a point in $\partial \mathit{MWW}_{d_1,\dots,d_r}$, any vertices that are collapsed under \eqref{eq:mww-into-fm} correspond to cylindrical components $C = \bC \setminus \{z_1\}$, which we equip with the standard inhomogeneous term $\nu_C = (Z_H \otimes \mathrm{re} (d\log(z-z_1)/2\pi \sqrt{-1}) \big)^{0,1}$. Moreover, a generic choice of \eqref{eq:mww-into-fm} ensures transversality for the parameterized moduli spaces associated to all $\mathit{MWW}_{d_1,\dots,d_r}$, and we may then use that choice to build the operations $\beta_{\scrC}^{d_1,\dots,d_r}$. Additionally, we may assume that the maps \eqref{eq:mww-into-fm} are chosen so that \eqref{eq:drop-color-2} holds, which means that the resulting operations satisfy \eqref{eq:cancel-0}. 

We will be specifically interested in $\mathit{MWW}_{\howmany{p}}$, in the notation from Section \ref{subsec:unordered}, and the associated operation $\beta_{\scrC}^{\underline{p}} = \beta_{\scrC}^{1,\dots,1}$, with $p$ prime. At this point, we fix an odd degree cocycle
\begin{equation}
c \in \scrC \otimes \bF_p, 
\end{equation}
which will remain the same throughout the subsequent discussion. Applying $\beta_{\scrC}^{\underline{p}}$ to $p$ copies of $c$ yields another such cocycle, hence a cohomology class
\begin{equation} \label{eq:beta-p-fold}
[\beta_{\scrC}^{\underline{p}}(c;\dots;c)] \in H^{\mathit{odd}}(\scrC;\bF_p).
\end{equation}
The underlying geometric phenomenon was explained in Section \ref{subsec:unordered}: the codimension $1$ boundary faces of $\mathit{MWW}_{\howmany{p}}$ correspond to nontrivial decompositions of $\{1,\dots,p\}$ into nonempty subsets $(I_1,\dots,I_r)$, for any $r \geq 2$. If we act by an element of $\bZ/p$ on such a decomposition, we get a new decomposition $(\tilde{I}_1,\dots,\tilde{I}_r)$, and the corresponding boundary faces, when mapped to $\mathit{FM}_p$, are related by the action of a suitable element of $\mathit{Sym}_p$, see \eqref{eq:two-decompositions}. The cohomology class in \eqref{eq:beta-p-fold} is independent of our choice of inhomogeneous terms.
 
Our first point is that we can realize \eqref{eq:beta-p-fold} using a family of surfaces without degenerations. To do that, let's choose a $\mathit{Sym}_p$-equivariant isotopy that pushes Fulton-MacPherson space into its interior,
\begin{equation}
\begin{aligned}
& \phi_r: \mathit{FM}_p \longrightarrow \mathit{FM}_p, \; r \in [0,\epsilon], \\
& \phi_0 = \mathit{id}, \quad
\phi_r(\mathit{FM}_p) \subset \mathit{FM}_p^\circ = \mathit{FM}_p \setminus \partial \mathit{FM}_p \; \text{for $r>0$.}
\end{aligned}
\end{equation}
Write $\iota_{\howmany{p}}$ for the original map $\mathit{MWW}_{\howmany{p}} \rightarrow \mathit{FM}_p$. The perturbed version
\begin{equation} \label{eq:push-1}
\tilde\iota_{\howmany{p}} = \phi_r \circ \iota_{\howmany{p}}: \mathit{MWW}_{\howmany{p}} \longrightarrow \mathit{FM}_p \setminus \partial \mathit{FM}_p,
\;\;\text{for some $r>0$,}
\end{equation}
will retain the same $\bZ/p$-action on codimension one boundary faces as $\iota_{\howmany{p}}$. 
Going back to the choice of inhomogeneous terms over $\mathit{FM}_p$, we want to also assume that the pullback of that family by \eqref{eq:push-1} should lead to a regular parametrized moduli space. Given that, from \eqref{eq:push-1} for some $r>0$ we get a new operation $\tilde{\beta}_{\scrC}^{\howmany{p}}$, which again yields a cohomology class
\begin{equation} \label{eq:beta-tilde-class}
[\tilde{\beta}_{\scrC}^{\howmany{p}}(c;\dots;c)] \in H^{\mathit{odd}}(\scrC;\bF_p).
\end{equation}
A similar construction, where one interpolates between $\iota_{\howmany{p}}$ and $\tilde{\iota}_{\howmany{p}}$, shows that this cohomology class agrees with \eqref{eq:beta-p-fold}. At this point, we no longer need to compactify configuration space: to define \eqref{eq:beta-tilde-class}, one can use families of perturbation data which are only defined on $\mathit{FM}_p^\circ$ (and still $\mathit{Sym}_p$-equivariant).
%

In the same vein as in \eqref{eq:freeing-up}, take
\begin{equation} \label{eq:freeing-up-2}
\EuScript{DM}^\circ_p = \mathit{DM}^\circ_p \times E^\circ_p,
\end{equation}
where $E^\circ_p$ is the interior of Fulton-MacPherson space for $\bR^\infty$, meaning point configurations up to translation and rescaling. More precisely, we think of this as the direct limit of the corresponding spaces in each finite-dimensional Euclidean space. There is an embedding
\begin{equation} \label{eq:extended-embedding}
\mathit{FM}^\circ_p \longrightarrow \EuScript{DM}^\circ_p,
\end{equation}
which takes each point configuration to the pair formed by its quotient in Deligne-Mumford space and its image in $E_p^\circ$. In this context, classes in $H_*^{\mathit{Sym}_p}(\mathit{DM}^\circ_p; \bF_p(1))$ are realized by smooth simplicial chains in \eqref{eq:freeing-up-2}, having $\bF_p$-coefficients, and quotiented out by the relation that acting on a chain by some $\sigma \in \mathit{Sym}_p$ is the same as multiplying the chain with $(-1)^{\mathrm{sign}(\sigma)}$. Let's write $C_*^{\mathit{Sym}_p}(\mathit{DM}^\circ_p; \bF_p(1))$ for this chain complex.

Choose a family of inhomogeneous terms on the family of surfaces pulled back by the projection $\EuScript{DM}^\circ_p \rightarrow \mathit{DM}^\circ_p$, and which is $\mathit{Sym}_p$-equivariant. For every smooth map from a simplex to $\EuScript{DM}^\circ_p$, that family of inhomogeneous terms gives rise to a parametrized moduli space. If that space is regular, we get an operation $(\scrC \otimes \bF_p)^{\otimes p} \rightarrow \scrC \otimes \bF_p$. By adding up those operations with coefficients, we extend the construction to chains. Let's specialize to using $p$ copies of our cocycle $c$ as input. Morally, this can be thought of as giving rise to a $\bZ/2$-graded chain map
\begin{equation} \label{eq:all-chains}
C_*^{\mathit{Sym}_p}(\mathit{DM}^\circ_p; \bF_p(1)) \longrightarrow \scrC^{p|c|+*} \otimes \bF_p.
\end{equation}
The cautionary ``morally'' figures here because of the regularity condition for moduli spaces, which makes it impossible to define such a map on the entire chain complex. However, any argument involving a relation between specific chains, such as the one we are about to give, only involves finitely many terms, and one can assume that the chains involved are embedded into the infinite-dimensional space $\EuScript{DM}^\circ_p$. One can a posteriori make a choice of perturbation terms over $\EuScript{DM}^\circ_p$ which makes the finitely many spaces involved regular. Hence, for all practical purposes, the consequence is the same as if we had a map \eqref{eq:all-chains}. In particular, we do get a map
\begin{equation} \label{eq:c-c-map}
H_*^{\mathit{Sym}_p}(\mathit{DM}^\circ_p;\bF_p(1)) \longrightarrow H^{p+*}(\scrC;\bF_p).
\end{equation}
One can think of \eqref{eq:beta-tilde-class} as an instance of this general construction, by smoothly triangulating the spaces $\mathit{MWW}_{\underline{p}}$, in a way which is compatible with the $\bZ/p$-action on codimension $1$ boundary strata, and then using the embedding \eqref{eq:extended-embedding}. Using Lemma \ref{th:open-moduli-space}, one identifies the relevant homology class with that underlying the $t^{(p-1)/2}$ coefficient of the quantum Steenrod operation, up to a coefficient which is spelled out in Lemma \ref{th:cycle-found}. This equality, applied to \eqref{eq:c-c-map}, implies Theorem \ref{th:p-power}.

\section{An alternative approach\label{sec:fukaya}}

The approach outlined in this section was pointed out to the author by Fukaya. It is an application of the results from \cite{fukaya17} (taking the Lagrangian correspondence to be the diagonal, but with a general bounding cochain, which is our Maurer-Cartan element). The basic building blocks are parameter spaces from \cite{mau10b}, which are close cousins of Stasheff associahedra (and in particular, are manifolds with corners in the classical sense). One can use them to define the composition law on Maurer-Cartan elements, a little indirectly, following \cite[Theorem 1.7]{fukaya17}; and to prove its associativity, following \cite[Theorem 1.8]{fukaya17}. On the other hand, it's not clear that there is a easier route from there to Theorem \ref{th:p-power}, which is one reason why we have not given first billing to this approach. Because of its complementary nature, our discussion will be quite sparse: not only are proofs omitted, we won't even make the distinction between the implementation of these arguments in an abstract operadic context (as in Section \ref{sec:group}, assuming homological unitality) or a concrete Floer-theoretic one (as in Section \ref{sec:floer}).

\subsection{The moduli spaces}
We start with basically the same configuration space as in \eqref{eq:configuration-space-4}, except that the ordering of points in the last color is reversed:
\begin{equation} \label{eq:configuration-space-x}
\frac{
\{ (s_{1,1},\dots,s_{1,d_1}; \dots ;s_{r,1},\dots,s_{r,d_r}), \;\; s_{k,1} <  \cdots < s_{k,d_k} \text{ for $k<r$, and }
s_{r,1} > \cdots > s_{r,d_r}
\} 
}{ 
\{ s_{k,i} \sim s_{k,i} + \mu \; \text{for $\mu \in \bR$}\} 
}.
\end{equation}
More importantly, we now consider a compactification of \eqref{eq:configuration-space-x} which is smaller than its counterpart from Section \ref{subsec:mww}. This compactification will be denoted by 
\begin{equation} \label{eq:f-space}
Q_{d_1,\dots,d_r}, \;\; \text{where } r \geq 2, \; d_1,\dots,d_r \geq 0, \; d = d_1 + \cdots + d_r > 0,
\end{equation}
partly following the ``quilted strips'' terminology from \cite{mau10b}. The recursive structure of boundary strata is expressed by maps
\begin{equation} \label{eq:glue-tree-x}
\prod_{j=1}^m Q_{\|v_{\circ,j}\|_1,\dots,\|v_{\circ,j}\|_r} \times \prod_{
\substack{v \text{ in } T_j \\ v \neq v_{\circ,j}}} S_{\|v\|} \xrightarrow{(T_1,\dots,T_m)} Q_{d_1,\dots,d_r}.
\end{equation}
Here, $T_1,\dots,T_m$ (for any $m \geq 1$) are trees of the following kind. In each $T_j$, denote by $v_{\circ,j}$ the vertex closest to the root. Then, the incoming edges at that vertex should carry one of $r$ colors, and are ordered within their color. The parts of the tree lying above $v_{\circ,j}$ have planar embeddings, and inherit a single color. The whole thing is arranged, of course, so that the total number of leaves of each respective color add up to $(d_1,\dots,d_r)$. Geometrically, what happens is that as groups of points move to $\pm \infty$, we split them up into separate screens, which correspond to the $Q$ factors in \eqref{eq:glue-tree-x} (in the terminology of Section \ref{subsec:mww}, these would be called mid-scale, since there is no rescaling involved, just translation); but we do not keep track of the relative speeds at which this divergence happens (no large-scale screens). The remaining factors in \eqref{eq:glue-tree-x} are small-scale screens, which describe the limit of points converging towards each other. The image of \eqref{eq:glue-tree-x} has codimension equal to the overall number of factors (vertices) minus one. This is related to the fact that $Q_{d_1,\dots,d_r}$ is a smooth manifold with corners.

Let's map our points to radial half-lines in the punctured plane,
\begin{equation} \label{eq:punctured-point}
z_{k,i} = \exp(-s_{k,i} - \textstyle\frac{2\pi k}{r} \sqrt{-1}) \in \bC^*,
\end{equation}
and add a marked point at $0$ (the $z_{k,i}$ are ordered lexicographically, and the extra point is inserted between the last two colors). This extends to a continuous map
\begin{equation} \label{eq:f-to-fm}
Q_{d_1,\dots,d_r} \longrightarrow \mathit{FM}_{d+1}.
\end{equation}
In terms more familiar from pseudo-holomorphic curve theory, one can think of the configurations \eqref{eq:punctured-point} as lying on parallel lines on a cylinder. In the limit, this breaks up into several cylinders, plus spheres (copies of $\bar{\bC}$ with a marked point at infinity, and other marked points lying on the real line) attached to them; see Figures \ref{fig:cylinder}(i) and \ref{fig:jumpsuit}(i). On the combinatorial level, the map \eqref{eq:f-to-fm} works as follows: starting with trees as in \eqref{eq:glue-tree-x}, one adds an incoming edge to each vertex $v_{\circ,j}$ except the last one, and then identifies those edges with the root edges of $T_{j+1}$, thereby combining all our trees into a single $T$, which is what appears in \eqref{eq:glue-tree-2}; see Figure \ref{fig:cylinder}(ii).
\begin{figure}
\begin{centering}
\includegraphics{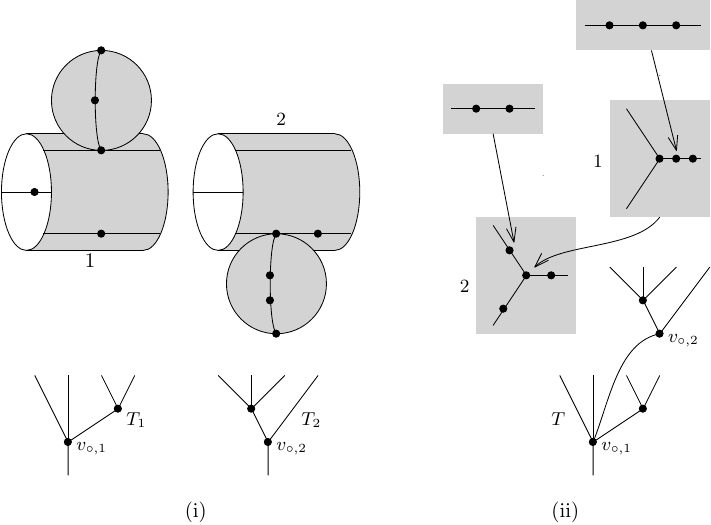}
\caption{\label{fig:cylinder}(i) A boundary point in \eqref{eq:f-space}, drawn in the way familiar from pseudo-holomorphic curve theory; and (ii) the corresponding picture in Fulton-MacPherson space.}
\end{centering}
\end{figure}

\subsection{The operations}
Algebraically, the outcome of using \eqref{eq:f-space} and \eqref{eq:f-to-fm} are operations
\begin{equation} \label{eq:multiop}
\chi_{\scrC}^{d_1,\dots,d_{r-1},1,d_r}:  \scrC^{d_1 + \cdots + d_r+1} \longrightarrow \scrC[1-d_1-\cdots-d_r].
\end{equation}
Additionally, we set
\begin{equation} \label{eq:10000}
\chi_{\scrC}^{0,\dots,1,0} = \mu^1_{\scrC}.
\end{equation}
The property of these operations, derived as usual from the structure of codimension $1$ boundary strata, and including \eqref{eq:10000}, is that 
\begin{equation}
\begin{aligned}
& \sum_{\substack{k<r \\ ij}} \pm \chi_{\scrC}^{d_1,\dots,d_k-j+1,\dots,1,d_r}\big(c_{1,1},\dots,c_{1,d_1};  \dots; 
c_{k,1},\dots, \mu_{\scrC}^j(c_{k,i+1},\dots,c_{k,i+j}), 
\\[-1.5em] & \qquad \qquad \qquad \qquad \dots,c_{k,d_k} ; \dots; c; c_{r,1},\dots,c_{r,d_r}\big) \\[.5em]
& + \sum_{p_1,\dots,p_r} \pm \chi_{\scrC}^{d_1-p_1,\dots,1,d_r-p_r}\big(c_{1,1},\dots,c_{1,p_1};\dots;
c_{r-1,1},\dots,c_{r,p_{r-1}}; 
\\[-1em] & \qquad \qquad \qquad 
\chi_{\scrC}^{p_1,\dots,1,p_r}(c_{1,p_1+1},\dots,c_{1,d_1};\dots;c;c_{r,p_1},\dots,c_{r,p_r});
c_{r,p_r+1},\dots,c_{r,d_r} \big) 
\\[.5em] & + 
\sum_{ij} \pm \chi_{\scrC}^{d_1,\dots,1,d_r-j+1}\big(c_{1,1},\dots;c;c_{r,1},\dots,c_{r,i},\mu_{\scrC}^j(c_{r,i+1},\dots,c_{r,i+j}),\dots,c_{r,d_r})
= 0.
\end{aligned}
\end{equation}
In terminology similar to \cite{mau10b}, these define the structure of an $A_\infty$-$(r-1,1)$-module, with the first $(r-1)$ factors acting on the left, and the last one on the right.

\begin{example} \label{th:up-down-0}
Suppose that $d_1 = \cdots = d_{r-1} = 0$. Then, the points \eqref{eq:punctured-point}, with the origin added as usual, lie on a half-line in $\bC$. One can use that to identify $Q_{0,\dots,0,d} \iso S_{d+1}$. The auxiliary data involved in defining \eqref{eq:multiop} can be chosen to be compatible with that, in which case one gets $\chi_{\scrC}^{0,\dots,1,d} = \mu_{\scrC}^{d+1}$.
\end{example}

\begin{example} \label{th:up-down}
For $r = 2$, the points \eqref{eq:punctured-point} still lie on $\bR \subset \bC$, hence $Q_{d_1,d_2} \iso S_{d_1+d_2+1}$ and, for suitable choices, $\chi_{\scrC}^{d_1,1,d_2} = \mu_{\scrC}^{d_1+d_2+1}$.
\end{example}

\begin{example} \label{th:its-the-product}
The operations with two inputs, $\chi_{\scrC}^{0,\dots,1,\dots,1,0}$ and $\chi_{\scrC}^{0,\dots,1,1}$, are all chain homotopic to the multiplication $\mu^2_{\scrC}$, simply because they come from a single two-point configuration in the plane.
\end{example}

Our purpose in defining these operations is the following:

\begin{definition} \label{th:trivial-tuple}
Let $\gamma_1,\dots,\gamma_r \in \scrC^1 \hat\otimes N$ be Maurer-Cartan elements. We say that $\gamma_r$ is the product of $(\gamma_1,\dots,\gamma_{r-1})$ if there is a $k \in \scrC^0 \hat\otimes (\bZ 1 \oplus N)$ which modulo $N$ reduces to a cocycle representing the unit $[e_{\scrC}]$, and such that 
\begin{equation} \label{eq:proof-of-triviality}
\sum_{d_1,\dots,d_r} \chi_{\scrC}^{d_1,\dots,d_{r-1},1,d_r}(\overbrace{\gamma_1,\dots,\gamma_1}^{d_1};\dots;
\overbrace{\gamma_{r-1},\dots,\gamma_{r-1}}^{d_{r-1}}; k;
\overbrace{\gamma_r,\dots,\gamma_r}^{d_r}) = 0.
\end{equation}
\end{definition}

The expression \eqref{eq:proof-of-triviality} includes a term $\mu^1_{\scrC}(k)$, corresponding to $(d_1,\dots,d_r) = (0,\dots,0)$. If write $k = e_{\scrC} + (\text{\it coboundary}) + h$ with $h \in \scrC^0 \hat\otimes N$, then (keeping Example \ref{th:its-the-product} in mind) the next order term in the equation says that
\begin{equation}
\begin{aligned}
& [\chi_{\scrC}^{1,0,\dots,1,0}(\gamma_1,k)] + \cdots + [\chi_{\scrC}^{0,\dots,1,1,0}(\gamma_{r-1},k)] + [\chi_{\scrC}^{0,\dots,1,1}(k,\gamma_r)] \\
& \qquad \qquad = [\mu^2_{\scrC}(\gamma_1,e_{\scrC})] + \cdots + 
[\mu^2_{\scrC}(\gamma_{r-1},e_{\scrC})] + [\mu^2_{\scrC}(e_{\scrC},\gamma_r)] \\ & \qquad \qquad
= [-\gamma_1 - \cdots - \gamma_{r-1} + \gamma_r] = 0 \;\; \text{ in } H^1(\scrC \otimes N/N^2).
\end{aligned}
\end{equation}
For $r = 2$, and assuming the choices have been made as in Example \ref{th:up-down}, the condition in \eqref{eq:proof-of-triviality} reduces to the criterion for equivalence of $\gamma_1$ and $\gamma_2$ given in Lemma \ref{th:alternative-mc}.

\begin{lemma}
The notion of product from Definition \ref{th:trivial-tuple} only depends on the equivalence class of the Maurer-Cartan elements involved.
\end{lemma}

This is the analogue of Lemma \ref{th:multiproduct}, and is proved in a similar way. Given $(\gamma_1,\dots,\gamma_r)$ and $k$ as in \eqref{eq:proof-of-triviality}, and an element $h \in \scrC^0 \hat\otimes N$ which provides an equivalence between $\gamma_j$ and $\tilde{\gamma}_j$, we can construct an explicit $\tilde{k}$ which shows that $(\gamma_1,\dots,\tilde{\gamma}_j,\dots,\gamma_r)$ satisfy the same condition:
\begin{equation}
\tilde{k} = k + \sum_{d_1,\dots,d_r,i} \chi_{\scrC}^{d_1,\dots,1,d_r}(\dots;\overbrace{\gamma_j,\dots,\gamma_j}^i,h,\overbrace{\tilde{\gamma}_j,\dots,\tilde{\gamma}_j}^{d_j-i-1};\dots;k;\dots)
\end{equation}
(the formula as written is for $j<r$, but the $j = r$ case is parallel).

\begin{lemma}
Given $(\gamma_1,\dots,\gamma_{r-1})$, there is a unique equivalence class $\gamma_r$ which satisfies Definition \ref{th:trivial-tuple}.
\end{lemma}

This is roughly analogous to Lemma \ref{th:inverse-3}. It is maybe helpful to reformulate the issue as follows. We have a right $A_\infty$-module structure, defined by 
\begin{equation}
(c;c_1,\dots,c_d) \longmapsto \sum_{d_1,\dots,d_{r-1}} \chi_{\scrC}^{d_1,\dots,d_{r-1},1,d}(\overbrace{\gamma_1,\dots,\gamma_1}^{d_1};\dots;
\overbrace{\gamma_{r-1},\dots,\gamma_{r-1}}^{d_{r-1}}; c; c_1,\dots,c_d),
\end{equation}
which (thanks to Example \ref{th:up-down-0}) is a deformation of the free module $\scrC$. One then wants to modify that module structure through $\gamma$ insertions, so as to ``undo'' the deformation, rendering it trivial. This is a purely algebraic question, which can be reduced to the strictly unital situation if desired (using Lemma \ref{th:unit-unit}).


\begin{proposition} \label{th:jumpsuit}
In the sense of Definition \ref{th:trivial-tuple}, if $\gamma$ is the product of $(\gamma_i,\gamma_{i+1})$ for some $i<r-1$, and $\gamma_r$ is the product of $(\gamma_1,\dots,\gamma_{i-1},\gamma,\gamma_{i+2},\dots,\gamma_{r-1})$, then $\gamma_r$ is also the product of $(\gamma_1,\dots,\gamma_{r-1})$.
\end{proposition}

This is the associativity statement for our notion of product.
The proof uses a moduli space of points lying on certain lines in the punctured plane. It is convenient to draw that plane as a pair-of-pants, see Figure \ref{fig:jumpsuit}(ii), which is half of the ``double pants diagram'' in \cite[Section 11.2]{fukaya17}. The two known statements about products come with their respective elements $k$ as in \eqref{eq:proof-of-triviality}. One inserts those elements at the two bottom ends, and the Maurer-Cartan elements at points on the respective lines (the arrows denote the ordering of the points), the outcome being another element $k$ which establishes the desired statement.
\begin{figure}
\begin{centering}
\includegraphics{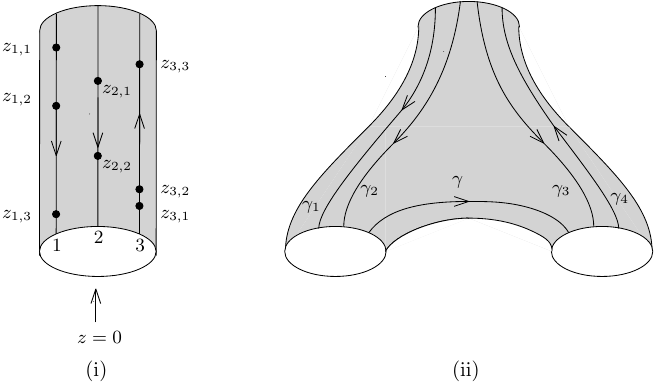}
\caption{\label{fig:jumpsuit}(i) Another picture of a point configuration \eqref{eq:punctured-point}, with correct ordering of the lines and points. (ii) An analogous picture of the moduli spaces that enter into the proof of Proposition \ref{th:jumpsuit} ($r = 4, \, i = 1$).}
\end{centering}
\end{figure}

\begin{remark} \label{th:tuple-algebra}
Let's briefly discuss the counterpart of Definition \ref{th:trivial-tuple} in homological algebra, in the spirit of Section \ref{sec:mc}. Take a homologically unital $A_\infty$-ring $\scrA$, and its Hochschild complex $\scrC$, see \eqref{eq:hochschild-complex}. Given a Maurer-Cartan element $\gamma \in \scrC^1 \hat\otimes N$, one can define two $A_\infty$-bimodules $\scrL_\gamma$ and $\scrR_\gamma$, whose underlying space is $\scrA \hat\otimes (\bZ 1 \oplus N)$, with bimodule structure:
\begin{align} 
\label{eq:insert-left}
&
\begin{aligned}
& \mu_{\scrL_\gamma}^{p;1;q}(a_1,\dots,a_p;a_{p+1};a_{p+2},\dots,a_{p+q+1}) = \pm \mu_{\scrA}^{p+q+1}(a_1,\dots,a_{p+q+1})
\\ & \qquad \qquad
+ \sum_{ij} \pm \mu_{\scrA}^{p+q+2-j}(a_1,\dots,\dots,a_i,\gamma^j(a_{i+1},\dots,a_{i+j}),\dots,a_{p+1}, \dots, a_{p+q+1}), 
\\ & \qquad \qquad
+ \sum_{i_1 j_1 i_2 j_2} \pm \mu_{\scrA}^{p+q+3-j_1-j_2}(a_1,\dots,a_{i_1},\gamma^{j_1}(a_{i_1+1},\dots,a_{i_1+j_1}), \dots
\\[-1em]
& \qquad \qquad \qquad \qquad \qquad a_{i_2},\gamma^{j_2}(a_{i_2+1},\dots,a_{i_2+j_2}),\dots,a_{p+1},\dots,a_{p+q+1}) 
\\ & \qquad \qquad + \cdots
\end{aligned}
\\
& \label{eq:insert-right}
\begin{aligned}
& \mu_{\scrR_\gamma}^{p;1;q}(a_1,\dots,a_p;a_{p+1};a_{p+2},\dots,a_{p+q+1}) = \pm \mu_{\scrA}^{p+q+1}(a_1,\dots,a_{p+q+1})
\\ & \qquad \qquad
+ \sum_{ij} \pm \mu_{\scrA}^{p+q+2-j}(a_1,\dots,a_{p+1},\dots,a_i,\gamma^j(a_{i+1},\dots,a_{i+j}),\dots,a_{p+q+1})
\\ & \qquad \qquad
+ \sum_{i_1 j_1 i_2 j_2} \pm \mu_{\scrA}^{p+q+3-j_1-j_2}(a_1,\dots,a_{p+1},\dots,a_{i_1},\gamma^{j_1}(a_{i_1+1},\dots,a_{i_1+j_1}), \dots
\\[-1em]
& \qquad \qquad \qquad \qquad \qquad a_{i_2},\gamma^{j_2}(a_{i_2+1},\dots,a_{i_2+j_2}),\dots,a_{p+q+1}) 
\\ & \qquad \qquad + \cdots
\end{aligned}
\end{align}
Here, the rule is that an arbitrary number of $\gamma$ terms are inserted, but always to the left \eqref{eq:insert-left}, or right \eqref{eq:insert-right}, of $a_{p+1}$. For $\gamma = 0$, this reduces to $\scrA$ with the diagonal bimodule structure extended to $\scrA \hat\otimes (\bZ 1 \oplus N)$, which will denote by $\scrL_0 = \scrR_0 = \scrD$. More generally, $\scrL_\gamma$ and $\scrR_\gamma$ can be viewed as pullbacks of $\scrD$ by the formal automorphism \eqref{eq:gamma-to-phi} acting on one of the two sides. Using that, one sees easily that the bimodules are inverses: there are bimodule homotopy equivalences
\begin{equation} \label{eq:inverse-bimodule}
\scrR_{\gamma} \otimes_\scrA \scrL_{\gamma} \htp \scrL_{\gamma} \otimes_{\scrA} \scrR_{\gamma} \htp \scrD. 
\end{equation}
Here, the tensor product notation is shorthand: we are really taking the tensor product of $A_\infty$-bimodules relative to $\scrA \hat\otimes (\bZ 1 \oplus N)$, and making sure that completion with respect to the filtration of $N$ is taken into account. Modulo $N$, all our bimodules reduce to the diagonal bimodule. Consider the Hochschild complex of $\scrA$ with coefficients in a bimodule $\scrB$, denoted here by $\mathit{CC}^*(\scrB)$, see e.g.\ \cite[Section 2.9]{ganatra13}. We say that $\gamma_r$ is the product of $(\gamma_1,\dots,\gamma_{r-1})$ if there is a cocycle
\begin{equation} \label{eq:bimodule-tensor}
k \in \mathit{CC}^0(\scrR_{\gamma_1} \otimes_{\scrA} \cdots \otimes_{\scrA} \scrR_{\gamma_{r-1}}
\otimes_{\scrA} \scrL_{\gamma_r})
\end{equation}
which, after reduction modulo $N$, represents the identity in $\mathit{HH}^0(\scrA)$. For $r = 1$, one can use \eqref{eq:inverse-bimodule} to show that this is the case if and only if $\gamma_1,\gamma_2$ are equivalent.
On the other hand, for $\bullet$ defined as in \eqref{eq:gamma-composition}, there are homotopy equivalences
\begin{equation}
\begin{aligned}
& \scrL_{\gamma_1} \otimes_{\scrA} \cdots \otimes_{\scrA} \scrL_{\gamma_{r-1}} \htp \scrL_{\gamma_{r-1} \bullet \cdots \bullet \gamma_1}, \\
& \scrR_{\gamma_1} \otimes_{\scrA} \cdots \otimes_{\scrA} \scrR_{\gamma_{r-1}} \htp \scrR_{\gamma_1 \bullet \cdots \bullet \gamma_{r-1}}.
\end{aligned}
\end{equation}
Combining that with the previous observation shows that our definition of product is the same as saying that $\gamma_r$ is equivalent to $\gamma_1 \bullet \cdots \bullet \gamma_{r-1}$.
\end{remark}

\subsection{Relating the two approaches}
To conclude our discussion, we'll mention a possible way to connect the construction in this section to the rest of the paper, or more precisely: to prove that, if $\gamma_r$ is the product of $(\gamma_1,\dots,\gamma_{r-1})$ in the sense of Definition \ref{th:trivial-tuple}, then $\gamma_r$ is also equivalent to $\Pi^{r-1}_{\scrC}(\gamma_1,\dots,\gamma_{r-1})$, which is the product from Definition \ref{th:def-multiproduct}. The reader is exhorted to treat this as what it is, a suggestion: it relies on new moduli spaces whose structure has not been fully developed. 

We begin by introducing an additional parameter $t \in [0,1)$, and changing \eqref{eq:punctured-point} by letting the first $(r-1)$ radial lines collide in the limit $t \rightarrow 1$:
\begin{equation} \label{eq:punctured-point-2}
z_{k,i} = \begin{cases} 
\exp(-s_{k,i} - (1-t) \textstyle\frac{2\pi k}{r} \sqrt{-1} - t \pi \sqrt{-1}) \in \bC^* & k = 1,\dots,r-1, \\
\exp(-s_{r,i}) & k = r.
\end{cases}
\end{equation}
This leads to a compactification of $[0,1)$ times \eqref{eq:configuration-space-x}, which we write as
\begin{equation} \label{eq:pq-space}
t: \mathit{PQ}_{d_1,\dots,d_r} \longrightarrow [0,1], \;\; r \geq 3.
\end{equation}
Over each $t \in [0,1)$, the fibre is a copy of $Q_{d_1,\dots,d_r}$. In the limit $t \rightarrow 1$, points of the first $(r-1)$ colors bubble off into screens which have the structure of $\mathit{MWW}$ spaces (Figure \ref{fig:moveable-strips}). This means that we have maps
\begin{equation} \label{eq:pq-glue}
Q_{j,r_d} \times \prod_{i=1}^j \mathit{MWW}_{d_{1,i},\dots,d_{r-1,i}} \longrightarrow \mathit{PQ}_{d_1,\dots,d_r}.
\end{equation}
for each partition $d_1 = d_{1,1} + \cdots + d_{j,1}, \dots, d_{r-1} = d_{r-1,1} + \cdots + d_{r-1,j}$, whose images are the top-dimensional parts of the fibre of \eqref{eq:pq-space} over $t = 1$. The space \eqref{eq:pq-space} comes with a map to $\mathit{FM}_{d+1}$, which over the fibre $t = 0$ reduces to \eqref{eq:f-to-fm}.
\begin{figure}
\begin{centering}
\includegraphics{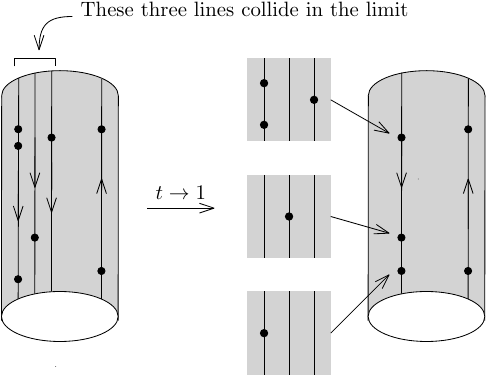}
\caption{\label{fig:moveable-strips}A limit in the space $\mathit{PQ}_{3,1,1,2}$ which lies in the image of \eqref{eq:pq-glue}. For compatibility with Figure \ref{fig:jumpsuit}, we have drawn the $\mathit{MWW}$ components rotated by $90$ degrees.}
\end{centering}
\end{figure}

The definition \eqref{eq:punctured-point-2} suffers from the usual disadvantage of parametrized spaces, meaning that its compactification contains strata that are fibre products over $[0,1]$. To bypass that difficulty, one can try to use those spaces in a ``time-ordered'' form (the same strategy as in \cite[Section 10e]{seidel04}), which means that we consider $k$-tuples of points in $\mathit{PQ}_{d_{1,1},\dots,d_{1,r}} \times \cdots \times \mathit{PQ}_{d_{k,1},\dots,d_{k,r}}$, where $(d_{i,k})$ is again a partition of $(d_k)$, such that the associated parameters satisfy $t_1 \leq \cdots \leq t_k$. Pairs of points with equal parameters now occur in the boundary of two different such spaces (when $t_i = t_{i+1}$ for some $i$; and as a boundary face of one of the $PQ$ factors involved).  If we insert Maurer-Cartan elements $(\gamma_1,\dots,\gamma_r)$ at the marked points in $\bC^*$, and add a trivial term which is the identity map, the outcome should be a map
\begin{equation} \label{eq:relating-map}
\Phi_{\scrC}: \scrC \longrightarrow \scrC
\end{equation}
which, setting $\gamma = \Pi^{r-1}_{\scrC}(\gamma_1,\dots,\gamma_{r-1})$, satisfies
\begin{equation}
\begin{aligned}
& \Phi_{\scrC} \big( \sum_{d_1,\dots,d_r} \chi_{\scrC}^{d_1,\dots,d_{r-1},1,d_r}(\overbrace{\gamma_1,\dots,\gamma_1}^{d_1};\dots;
\overbrace{\gamma_{r-1},\dots,\gamma_{r-1}}^{d_{r-1}}; c; \overbrace{\gamma_r,\dots,\gamma_r}^{d_r}) \big)
\\[-1em] & \qquad \qquad \qquad = \sum_{d_1,d_2} \chi_{\scrC}^{d_1,1,d_2}(\underbrace{\gamma,
\dots, \gamma}_{d_1}; \Phi_{\scrC}(c); \underbrace{\gamma_r,\dots,\gamma_r}_{d_2}).
\end{aligned}
\end{equation}
The left hand side of this equation represents what happens for $t_1 = 0$ (the $t_1 = 0$ component then gives $\chi_{\scrC}$, and the other components give $\Phi_{\scrC}$), while the right hand side represents what happens for $t_k = 1$ (with the $\gamma$ factors coming from the collision of the first $r-1$ lines). Suppose that $\gamma_r$ is the product of $(\gamma_1,\dots,\gamma_{r-1})$, see Definition \ref{th:trivial-tuple}. Then, inserting the associated element $k$ into \eqref{eq:relating-map} produces another element, which shows that ``$\gamma_r$ is the product of $(\gamma)$'' in the same sense. As pointed out before, in that special case, the definition just amounts to saying that $\gamma_r$ and $\gamma$ are equivalent.

\section{Computing quantum Steenrod operations\label{sec:quantum}}

By definition, quantum Steenrod operations belong to genus zero enumerative geometry. Generally speaking, it's an open question what their role is within that theory. However, for low degree contributions one can give a satisfactory answer, in terms of the usual Gromov-Witten invariants. After explaining this, we will turn to specific example computations.

\subsection{Gromov-Witten theory background\label{subsec:gw}}
Let's start in a context which is a little different than the rest of the paper. Take $X$ to be a closed symplectic $2n$-manifold, with the only restriction (for notational simplicity, since we only want to use power series in the Novikov variable $q$) that the symplectic form must lie in an integral cohomology class, denoted here by $\Omega_X \in \mathit{Im}(H^2(X;\bZ) \rightarrow H^2(X;\bQ))$. Genus zero Gromov-Witten invariants for $m$-pointed curves, and their generalizations that include gravitational descendants, will be written as
\begin{equation} \label{eq:gw}
\langle \psi^{r_1} x_1, \dots, \psi^{r_m} x_m \rangle_m = \sum_A q^{\Omega_X \cdot A} \langle \psi^{r_1} x_1,\dots, \psi^{r_m} x_m \rangle_{m,A} \in \bQ[[q]], \quad x_i \in H^ *(X;\bQ),
\end{equation}
where the sum is over $A \in H_2(X;\bZ)$. For the contribution of $A$ to be potentially nonzero, one should either consider classes with positive symplectic area $\Omega_X \cdot A = \int_A \omega_X > 0$, or take $A = 0$ (the case of constant curves) and $m \geq 3$. For expositions of Gromov-Witten that include the properties used in this paper, see e.g.\ \cite[Section 1]{pandharipande} or \cite[Chapter 26]{clay}.

We introduce another formal variable $t$, so that the coefficient ring for our algebraic considerations will be $\bQ[t^{\pm 1}][[q]]$. The small quantum product, and the small quantum connection, on the $\bZ/2$-graded space $H^*(X;\bQ)[t^{\pm 1}][[q]]$ are defined by 
\begin{equation}
\int_X (y_1 \ast y_2)\, y_3 = \langle y_1,y_2, y_3 \rangle_3
\end{equation}
and
\begin{equation} \label{eq:quantum-connection}
\nabla y = q\partial_q y + t^{-1} \Omega_X \ast y.
\end{equation}
We will consider endomorphisms $\Phi$ of $H^*(X;\bQ)[t^{\pm 1}][[q]]$ which are (linear over the coefficient ring and) covariantly constant with respect to $\nabla$. Concretely, this means that
\begin{equation} \label{eq:covariant-constancy}
(q \partial_q \Phi)(y) + t^{-1} \Omega_X \ast \Phi(y) - t^{-1} \Phi(\Omega_X \ast y) = 0.
\end{equation}
If we expand $\Phi = \Phi^{(0)} + q \Phi^{(1)} + q^2 \Phi^{(2)} + \cdots$, \eqref{eq:covariant-constancy} becomes
\begin{align} \label{eq:phi-0}
& \Phi^{(0)}(\Omega_X y) = \Omega_X \Phi^{(0)}(y), \\
&\Phi^{(k)}(y) = t^{-1} k^{-1} (\Phi^{(k)}(\Omega_X y) - \Omega_X \Phi^{(k)}(y) ) + \text{\it (recursive terms)}, \quad k>0, \label{eq:expand-phi-2}
\end{align}
where ``recursive terms'' is a generic name for expressions involving only $\Phi^{(0)},\dots,\Phi^{(k-1)}$. By repeatedly inserting \eqref{eq:expand-phi-2} into itself, we get
\begin{equation} \label{eq:expand-phi-3}
\Phi^{(k)}(y) = t^{-m} k^{-m} \sum_i (-1)^i \begin{pmatrix} m \\ i \end{pmatrix} \Omega_X^i\Phi^{(k)}(\Omega^{m-i}_X y) + \text{\it (recursive terms)}.
\end{equation}
Setting $m>2n$ means that in the sum we have $i>n$ or $m-i>n$, so all those terms vanish. One therefore gets explicit recursive formulae, which show that the constant term $\Phi^{(0)}$, subject to \eqref{eq:phi-0}, determines all of $\Phi$. The case we are interested in is where $\Phi^{(0)}(y) = xy$ is the cup product with a given class $x \in H^*(X;\bQ)$. There is a formula for the resulting $\Phi = \Phi_x$ in terms of gravitational descendants, closely related to the standard formula for solutions of the quantum differential equation:
\begin{equation} \label{eq:phi-map}
\begin{aligned}
& 
\int_X y_0\,\Phi_x(y_1) = \int_X y_0 x y_1 \\ & \qquad \qquad
- t^{-1} 
\langle y_0, (1+t^{-1}\psi)^{-1} x y_1 \rangle_2
+ t^{-1}
\langle (1 - t^{-1}\psi)^{-1} y_0 x, y_1 \rangle_2
\\ & \qquad \qquad
- t^{-2} 
\sum_k 
\langle y_0, (1+t^{-1}\psi)^{-1} x e_k \rangle_2
\langle (1-t^{-1}\psi)^{-1} e_k^\vee, y_1 \rangle_2.
\end{aligned}
\end{equation}
In principle, the terms $(1 + \cdots)^{-1}$ are supposed to be expanded into geometric series; but for degree reasons, only one term in this series is nonzero for each class $A$ that contributes to the expressions in \eqref{eq:phi-map}. The $(e_k)$, $(e_k^\vee)$ are Poincar{\'e} dual bases in $H^*(X;\bQ)$, meaning that in the Kunneth decomposition,
\begin{equation}
\sum_k e_k \otimes e_k^\vee = [\text{\it diagonal}] \in H^{2n}(X \times X;\bQ).
\end{equation}
Checking that $\Phi_x$ satisfies \eqref{eq:covariant-constancy} is an exercise using basic properties (divisor equation and TRR) of Gromov-Witten invariants. Using the string equation, one can write the special case $y_0 = y$, $y_1 = 1$ as
\begin{equation} \label{eq:string}
\begin{aligned}
& 
\int_X y\,\Phi_x(1) = \int_X yx 
- t^{-1} 
\langle y, (1+t^{-1}\psi)^{-1} x \rangle_2 
+ t^{-2} 
 \langle (1 - t^{-1}\psi)^{-1} y x \rangle_1
\\ & 
\qquad \qquad
- t^{-3} \sum_k \langle y, (1+t^{-1}\psi)^{-1} x e_k \rangle_2
\langle (1-t^{-1}\psi)^{-1} e_k^\vee \rangle_1.
\end{aligned}
\end{equation}

Let's modify the context slightly, and assume that $X$ is weakly monotone. Moreover, choose an integer lift of the symplectic cohomology class, again denoted by $\Omega_X$. Then, one can define mod $p$ versions of Gromov-Witten invariants counting curves in $A \in H_2(X;\bZ)$, for which we use the same notation:
\begin{equation} \label{eq:gw-mod-p}
\langle x_1,\dots,x_m \rangle_{m,A} \in \bF_p, \quad x_i \in H^ *(X;\bF_p), \quad
\text{provided that }\left\{\begin{aligned} &
m \geq 3, \text{ or: }\\
&
\text{any $m$ and } 0 <\Omega_X \cdot A<p.
\end{aligned}
\right.
\end{equation}
For $m \geq 3$, this is the classical definition in terms of an inhomogenous $\bar\partial$-equation (Gromov's trick). The definition in the second case can be reduced to the first case by taking the divisor equation as an axiom, where the class inserted is always (the mod $p$ reduction of) $\Omega_X$. Alternatively, one could argue more geometrically: if $\Omega_X \cdot A < p$, then no stable map in class $A$ can have an automorphism group whose order is a multiple of $p$. This should allow one to define virtual fundamental classes in homology with $\bF_p$-coefficients (we say ``should'' since this has not, to our knowledge, been carried out in the literature). The discussion of the second case also applies to gravitational descendants, with the same assumption $0 <\Omega_X \cdot A < p$. Geometrically, this uses the fact that orbifold line bundles whose isotropy groups have orders coprime to $p$ have Chern classes in mod $p$ cohomology; algebraically, one can use the formula (involving the divisor relation and TRR) that reduces invariants involving gravitational descendants to ordinary Gromov-Witten invariants.

The quantum product and connection can be considered as acting on $H^*(X;\bF_p)[t^{\pm 1}][[q]]$, where one now thinks of $t$ as in \eqref{eq:t-theta}. Formal linear differential equations in characteristic $p$ have a much larger space of solutions than their characteristic $0$ counterparts, simply because $\frac{d}{dx} x^p = 0$. As an instance of that, the uniqueness statement derived from \eqref{eq:expand-phi-3} now holds only up to order $q^{p-1}$, because of the division by $k^m$. If one truncates the formula \eqref{eq:phi-map} modulo $q^p$, then all terms appearing in it are defined with $\bF_p$-coefficients; and it yields the unique solution modulo $q^p$ of \eqref{eq:covariant-constancy}, whose $q^0$ term equals the cup product with $x$.

\subsection{Application to quantum Steenrod operations}
We adapt our previous definition of quantum Steenrod operations to the weakly monotone context, by adding the variable $q$. This means that, with $(t,\theta)$ as in \eqref{eq:t-theta} (and omitting the manifold $X$ for the sake of brevity),
\begin{equation} \label{eq:qst-q}
Q\mathit{St}_{p} = \sum_A q^{\Omega_X \cdot A} \, Q\mathit{St}_{p,A}: 
H^*(X;\bF_p) \longrightarrow H^*(X;\bF_p)[t,\theta][[q]].
\end{equation}
We find it convenient to introduce a minor generalization, which is a bilinear map on cohomology. More precisely, for each $x \in H^*(X;\bF_p)$ one gets an endomorphism of $H^*(X;\bF_p)[t,\theta][[q]]$, denoted by
\begin{equation} \label{eq:q-sigma}
Q\Sigma_{p,x} = \sum_A q^{\Omega_X \cdot A} Q\Sigma_{p,x,A}: H^*(X;\bF_p)[t,\theta][[q]] \longrightarrow H^*(X;\bF_p)[t,\theta][[q]].
\end{equation}
Geometrically, while quantum Steenrod operations are obtained from holomorphic maps which have \eqref{eq:fixed-point} as a domain, we use the remaining $\bZ/p$-fixed point ($z = 0$) on that curve as an additional input point to define \eqref{eq:q-sigma}. In other words, one can view it as an equivariant version of the ``quantum cap product'', obtained from the $\bZ/p$-equivariant curve in Figure \ref{fig:equivariant-cylinder}. On a technical level, the definition is entirely parallel to that of quantum Steenrod operations, by looking at moduli spaces parametrized by cycles in the classifying space $B\bZ/p$ \cite{seidel-wilkins21}. The $q^0$ term of \eqref{eq:q-sigma} is the cup product with the classical Steenrod operation,
\begin{equation} \label{eq:st-and-sigma}
Q\Sigma_{p,x,0}(y) = \mathit{St}_p(x)y.
\end{equation}
The relation between \eqref{eq:qst-q} and \eqref{eq:q-sigma} is that
\begin{equation} \label{eq:quantum-sigma-constant}
Q\mathit{St}_p(x) = Q\Sigma_{p,x}(1).
\end{equation}
\begin{figure}
\begin{centering}
\includegraphics{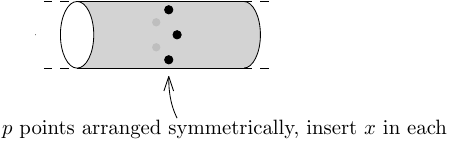}
\caption{\label{fig:equivariant-cylinder}The Riemann surface underlying the definition of $Q\Sigma_{p,x}(\cdot)$, see \eqref{eq:q-sigma}.}
\end{centering}
\end{figure}

\begin{remark}
It is natural to extend the definition of \eqref{eq:q-sigma} to $x \in H^*(X;\bF_p)[[q]]$ in a Frobenius-twisted way, meaning that $\Sigma_{p,qx} = q^p\Sigma_{p,x}$. Then, 
\begin{equation} \label{eq:compose-sigma}
Q\Sigma_{p,x_1} \circ Q\Sigma_{p,x_2} = (-1)^{\frac{p(p-1)}{2} |x_1|\,|x_2|} Q\Sigma_{p,x_1 \ast x_2}.
\end{equation}
Note that as a consequence of \eqref{eq:st-and-sigma} and \eqref{eq:compose-sigma},
\begin{equation} \label{eq:double-sigma}
\begin{aligned} &
Q\Sigma_{p,x}(Q\mathit{St}_p(y)) = Q\Sigma_{p,x} \circ Q\Sigma_{p,y}(1) = 
(-1)^{\frac{p(p-1)}{2} |x|\,|y|} Q\Sigma_{p,x \ast y}(1) \\
& \quad = (-1)^{\frac{p(p-1)}{2} |x|\,|y|} Q\mathit{St}_p(x \ast y).
\end{aligned}
\end{equation}
Every class in $H^*(X;\bF_p)[t,t^{-1},\theta]$ can be written as $\mathit{St}_p(y)$ for some $y$. An analogous statement holds for quantum Steenrod operations, and by combining that with \eqref{eq:double-sigma}, one sees that $Q\mathit{St}_p$ actually determines $Q\Sigma_p$.
\end{remark}

For our purpose, the key point is the following result:

\begin{theorem} \cite[Theorem 1.4]{seidel-wilkins21} \label{th:flat}
For any $x$, the endomorphism $Q\Sigma_{p,x}$ is covariantly constant for the small quantum connection \eqref{eq:quantum-connection}, meaning that it satisfies \eqref{eq:covariant-constancy} (we have tacitly extended the coefficient ring to include $\theta$).
\end{theorem}

As a consequence of that and the discussion at the end of Section \ref{subsec:gw}, $Q\Sigma_{p,x}$ is determined modulo $q^p$ by the classical term \eqref{eq:quantum-sigma-constant}. More explicitly, comparison with \eqref{eq:phi-map} shows that
\begin{equation}
Q\Sigma_{p,x} = \Phi_{\mathit{St}_p(x)} \; \text{modulo } q^p.
\end{equation}
Specializing to \eqref{eq:st-and-sigma} and using \eqref{eq:string} leads to:

\begin{corollary} \label{th:low-degree}
The low degree contributions to the quantum Steenrod operation are:
\begin{equation} \label{eq:low-degree}
\begin{aligned}
& 
\sum_{\!\!\Omega_X \cdot A < p\!\!} q^{\Omega_X \cdot A} \int_X y\,\mathit{QSt}_{p,A}(x) = \int_X y \,\mathit{St}_p(x)
\\ & 
- t^{-1} \sum_{\!\!0< \Omega_X \cdot A < p\!\!} q^{\Omega_X \cdot A}
\langle y, (1+t^{-1}\psi)^{-1} \mathit{St}_p(x) \rangle_{2,A} \\ &
+ t^{-2} \sum_{\!\!0 < \Omega_X \cdot A < p\!\!} q^{\Omega_X \cdot A}
 \langle (1 - t^{-1}\psi)^{-1} y\, \mathit{St}_p(x) \rangle_{1,A}
\\ & 
- t^{-3} \sum_{\substack{\Omega_X \cdot A_0 > 0 \\ \Omega_X \cdot A_1 > 0 \\
\Omega_X \cdot (A_0 + A_1) < p}} \sum_k q^{\Omega_X \cdot (A_0+A_1)}
\langle y, (1+t^{-1}\psi)^{-1} \mathit{St}_p(x)\, e_k \rangle_{2,A_0}
\langle (1-t^{-1}\psi)^{-1} e_k^\vee \rangle_{1,A_1}.
\end{aligned}
\end{equation}
\end{corollary}

Note that even though there are negative powers of $t$ in the formula, we know a priori that none of them can appear in $\mathit{QSt}_p$, so all terms involving them must cancel. 

\subsection{A localization argument\label{subsec:direct-localization}}
The approach to quantum Steenrod operations via Theorem \ref{th:flat} is formally slick, but maybe somewhat indirect; we will therefore suggest a possible alternative. For simplicity, we will work out only the most elementary case. Namely, let's assume that our symplectic manifold $X$ is an algebraic variety, and that we use the given complex structure. We fix some $A \in H_2(X;\bZ)$ which is holomorphically indecomposable: this means that one can't find nonzero classes $A_1,\dots,A_r$, $r \geq 2$, each of them represented by a holomorphic map $\bC P^1 \rightarrow X$, such that $A_1 + \cdots + A_r = A$. 
This implies that the space of unparametrized rational curves in class $A$ is compact, and contains no multiple covers. We further assume that this space is regular. Let's recall some notation,
\begin{equation}
\xymatrix{
& \!\!\!\!\!\!\!L_{p+1,0},\dots,L_{p+1,p}\!\!\!\!\!\!\! \ar[d] \\
\bar\scrM_{p+1}
& \ar[l]
\bar\scrM_{p+1}(X;A)
\ar[rrrr]^-{\mathit{ev}_{p+1} = (\mathit{ev}_{p+1,0},\dots,\mathit{ev}_{p+1,p})}
&&&& X^{p+1}.
}
\end{equation}
Here, $\bar\scrM_{p+1}$ is genus zero Deligne-Mumford space (we prefer to use the conventional algebro-geometric notation rather than that in the rest of the paper); $\bar\scrM_{p+1}(X;A)$ is the space of stable maps; and the $L_{p+1,k}$ are the tautological line bundles (cotangent bundles of the curve) at the marked points. Our assumption was that $\bar\scrM_0(X;A)$ is regular, hence smooth of complex dimension $n+c_1(A)-3$. As a consequence, all $\bar\scrM_{p+1}(X;A)$ are smooth of dimension $n+c_1(A)+(p-2)$, and actually fibre bundles over $\bar\scrM_0(X;A)$ with fibre $\bar\scrM_{p+1}(\bC P^1;1)$. The $\mathit{Sym}_p$-action on Deligne-Mumford space has a canonical lift to $\bar\scrM_{p+1}(X;A)$. 

At this point, we (re)impose the assumption that $p$ is prime. Let $\bar\scrM_{p+1}^\diamond(X;A) \subset \bar\scrM_{p+1}(X;A)$ be the subset of stable maps which, under the forgetful map to Deligne-Mumford space, are mapped to the point from \eqref{eq:fixed-point}. Clearly, that subset is invariant under $\bZ/p \subset \mathit{Sym}_p$. As before, one can describe its geometry explicitly: $\bar\scrM_{p+1}^\diamond(X;A)$ is  a fibre bundle over $\bar\scrM_0(X;A)$ with three-dimensional fibre $\bar\scrM_{p+1}^\diamond(\bC P^1;1)$, and $\bZ/p$ acts in a fibre-preserving way.

\begin{lemma} \label{th:fixed-locus}
The fixed point set $F  \subset \bar\scrM_{p+1}^\diamond(X;A)$ of the $\bZ/p$-action is the disjoint union of:
\begin{itemize} \itemsep.5em
\item[(i)] A copy of $\bar\scrM_2(X;A)$. The restriction of $\mathit{ev}_{p+1}$ to that component can be identified with $(\mathit{ev}_{2,0},\mathit{ev}_{2,1},\dots,\mathit{ev}_{2,1})$. Moreover, the normal bundle $N$ of this component is the dual of the tautological bundle $L_{2,1} \rightarrow \bar\scrM_2(X;A)$, and the $(\bZ/p)$-action on it has weight $-1$.

\item[(ii)] A copy of $\bar\scrM_1(X;A)$. The restriction of $\mathit{ev}_{p+1}$ to that component can be identified with $(\mathit{ev}_{1,0},\dots,\mathit{ev}_{1,0})$. Topologically, the normal bundle $N$ of this component is a direct sum of a trivial line bundle and the dual of $L_{1,0}$, with the $(\bZ/p)$-action having weight $1$ on each component. 
\end{itemize}
\end{lemma}
\begin{figure}
\begin{centering}
\includegraphics{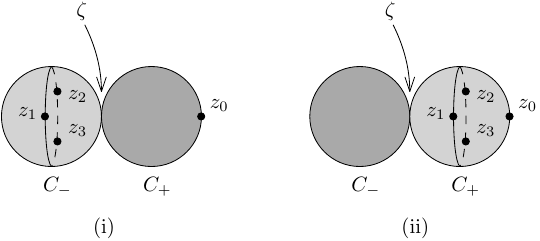}
\caption{\label{fig:fixed-loci}The fixed loci from Lemma \ref{th:fixed-locus} (with $p = 3$). The lighter shaded components are those where the map is constant.}
\end{centering}
\end{figure}%

\begin{proof}
(i) The relevant rational curves have two components $C = C_- \cup_{\zeta} C_+$ (see Figure \ref{fig:fixed-loci}). $C_-$ carries marked points $z_1,\dots,z_p$ and also the node $\zeta$, and can be identified with \eqref{eq:fixed-point}, in such a way that $\zeta$ corresponds to $\infty$; the stable map is constant on that component. The other component $C_+$ carries the node $\zeta$ and the marked point $z_0$; the stable map on that component represents $A$. The fibre of the normal bundle to the fixed locus at such a point can be canonically identified with $T_\zeta C_- \otimes T_\zeta C_+$ (see e.g.\ \cite[Proposition 3.31]{harris-morrison}). The identification of $C_-$ with \eqref{eq:fixed-point} mentioned above provides a distinguished isomorphism $T_\zeta C_- \iso \bC$ over the entire stratum, and also shows that the action of $\bZ/p$ on $T_\zeta C_-$ has weight $-1$. By definition, $T_\zeta C_+$ is the dual of the cotangent line which is the fibre of $L_{2,1}$, and carries the trivial $\bZ/p$-action.

(ii) Here, the curves also have two components $C = C_- \cup_{\zeta} C_+$, with details as follows (see again Figure \ref{fig:fixed-loci}). There are no marked points on $C_-$, and $u|C_-$ represents $A$. The other component is isomorphic to \eqref{eq:fixed-point}, compatibly with all marked points and so that the node $\zeta$ corresponds to $0 \in \bar{\bC}$; and $u|C_+$ is constant. The fibre of the normal bundle to the fixed locus at such a point can be written as an extension
\begin{equation} \label{eq:normal-extension}
0 \rightarrow T_{\zeta}(C_+) \longrightarrow N \longrightarrow T_{\zeta}(C_-) \otimes T_{\zeta}(C_+) \rightarrow 0.
\end{equation}
The tensor product in the right term again expresses gluing together the two components. This time, because $\zeta$ is identified with the point $0$ in \eqref{eq:fixed-point}, the $\bZ/p$-action has weight $1$ on $T_{\zeta}(C_+) \iso \bC$. The subspace on the left in \eqref{eq:normal-extension} corresponds to staying inside the stratum of nodal curves, but moving the position of the node on $C_+$. Topologically \eqref{eq:normal-extension} splits, even compatibly with the $(\bZ/p)$-action, leading to the desired statement.
\end{proof}

Reformulating the definition of quantum Steenrod operations, one can say that the pairing $(y,x) \mapsto \int_X y Q\mathit{St}_{p,A}(x)$ is obtained as follows:
\begin{equation} \label{eq:dual-quantum-steenrod}
\xymatrix{
\ar[dd]
H^j(X;\bF_p) \otimes H^l(X;\bF_p) \ar[rr] && H^{pj}_{\bZ/p}(X^p;\bF_p) \otimes H^l(X;\bF_p) \ar@{=}[d]
\\
&& H^{pj+l}_{\bZ/p}(X^{p+1};\bF_p) \ar[d]_-{\mathit{ev}_{p+1}^*} \\
H^{pj+l-2n-2c_1(A)+6}_{\bZ/p}(\mathit{point};\bF_p)
&& H^{pj+l}_{\bZ/p}(\bar{\scrM}_{p+1}^\diamond(X;A)) \ar[ll]^-{\displaystyle\int_{\bar\scrM_{p+1}^\diamond(X;A)}^{\bZ/p}}
}
\end{equation}
The $\rightarrow$ is (the topological version of) the equivariant diagonal map \eqref{eq:equivariant-diagonal}, and the integration map $\leftarrow$ is the pairing with the equivariant fundamental class of the moduli space. The localization theorem for $\bZ/p$-actions \cite[Proposition 5.3.18]{allday-puppe} shows that this integral can be computed in terms of the fixed locus:
\begin{equation} \label{eq:localization-theorem}
\int_{\bar{\scrM}_{p+1}^\diamond(X;A)}^{\bZ/p} w = \int_F (w|F) e_{\bZ/p}(N)^{-1}
\quad \text{for $w \in H^*_{\bZ/p}(\bar{\scrM}_{p+1}^{\diamond}(X;A))$},
\end{equation}
where $e_{\bZ/p}(N)$ is the equivariant Euler class of the normal bundle. Applying this to the description from Lemma \ref{th:fixed-locus}, with $w$ being a class pulled back by evaluation, we get:

\begin{corollary} \label{th:elementary-localise}
The contribution of a holomorphically indecomposable class $A$, which has regular moduli spaces, to the quantum Steenrod operation is
\begin{equation} \label{eq:indecomposable-descendants}
\int_X y\, \mathit{QSt}_{p,A}(x)
 = -t^{-1}\langle y, (1 + t^{-1}\psi)^{-1} \mathit{St}_p(x) \rangle_{2,A}  + t^{-2}\langle (1 - t^{-1}\psi)^{-1} y\, \mathit{St}_p(x) \rangle_{1,A}.
\end{equation}
\end{corollary}

The use of an integrable complex structure is not really necessary. It was convenient for expository purposes, because it makes the moduli spaces into differentiable manifolds, so that we can talk about the normal bundle to the fixed locus. However, \eqref{eq:localization-theorem} also applies to actions on topological manifolds, provided that they are linear in local topological charts near the fixed locus (thereby defining a notion of normal bundle). One can prove that property for regular moduli spaces of pseudo-holomorphic curves by standard gluing methods. 

Clearly, \eqref{eq:indecomposable-descendants} is compatible with the formula \eqref{eq:low-degree} which we obtained by other means (conversely, one could specialize our earlier argument to only use the class $A$, and thereby recover \eqref{eq:indecomposable-descendants} from it). In principle, this localisation method should apply more generally to classes $A$ that are $p$-indecomposable: by that, we mean that one can't write $A = pA_1 + A_2 + \cdots + A_r$, for $r \geq 1$ and nonzero classes $A_1,\dots,A_r$ which are represented by holomorphic curves (this is always satisfied if $\Omega_X \cdot A < p$). In that situation, there is another component to the fixed locus, which corresponds to the final sum in \eqref{eq:low-degree}. However, note that there are significant technical issues: one can no longer assume that the moduli spaces under consideration are smooth, hence presumably needs to apply a virtual analogue of localisation, analogous to \cite{graber-pandharipande99}. Maybe the most salient argument in favor of the direct approach is that it shows how, when going beyond the situations we have considered so far, the existence of $p$-fold covered curves complicate the situation: such curves yield yet more components of the fixed locus of the $\bZ/p$-action, whose contributions would need to be studied separately.

\subsection{Basic examples}
From this point onwards, we return to the monotone context from the main part of the paper (actually, our first two examples are only spherically monotone, meaning that the symplectic class and first Chern class are positive proportional on $\pi_2(X)$, but that's just as good for our purpose). In terms of formulae such as \eqref{eq:low-degree}, this simply means that all expressions are polynomials in $q$ for degree reason, and hence it is permitted to remove the formal variable by setting $q = 1$.

\begin{example} \label{th:cp2}
Let $X = T^2 \times \bC P^2$. We consider $p = 2$, take any $x \in H^1(T^2;\bF_2)$, and let $l \in H^2(\bC P^2;\bF_2)$ be the generator. The classical contribution is
\begin{equation} \label{eq:p2-classical}
\left\{
\begin{aligned}
& \Xi_2(x \otimes 1) = x \otimes 1, \\
& \Xi_2(x \otimes l) = x \otimes l^2, \\
& \Xi_2(x \otimes l^2) = 0.
\end{aligned}
\right.
\end{equation} 
For degree reasons, the only other contribution comes from the class $A$ of a line in $\bC P^2$, and in view of the symplectic automorphism invariant, that contribution must be that $Q\Xi_{2,A}(x \otimes l^2)$ is some multiple of $x \otimes l$. Using the notation $\theta = t^{1/2}$, we have $\mathit{QSt}(x \otimes l^2) = t^{5/2}  (x \otimes l^2)$. Take some $y \in H^1(T^2;\bF_2)$. Then \eqref{eq:indecomposable-descendants} says that
\begin{equation}
\int_X (y \otimes l)\, \mathit{QSt}_{2,A}(t^{5/2} (x \otimes l^2)) = t^{1/2}\langle y \otimes l, \psi\, (x \otimes l^2) \rangle_{2,A} = t^{1/2} \Big(\int_{T^2} yx\Big) \langle l, \psi l^2 \rangle_{\bC P^2,2,A}.
\end{equation} 
In the rightmost term, the Gromov-Witten invariant is taken in $\bC P^2$, and we have marked that notationally (to get to that expression, we have used the fact that all our curves are constant in $T^2$-direction). Geometrically, what we are considering in that term is the space of all lines in $\bC P^2$ going through a specific line $Z$ (representing $l$) and a point $q \notin Z$ (representing $l^2$). That moduli space can be identified with $Z$, and the normal bundle to $Z$ is the dual to the line bundle which gives rise to the gravitational descendant in the formula. Hence, $\langle l, \psi l^2 \rangle_{\bC P^2,2,A} = 1$.

Let's pass to the algebraic closure $\bar{\bF}_2$, with a nontrivial third root of unity $\zeta \in \bar{\bF}_2$. This yields a splitting of $1 \in \mathit{QH}^*(\bC P^2;\bar{\bF}_2)$ into idempotents $u_j = 1 + \zeta^j l + \zeta^{2j} l^2$. The natural extension of $Q\Xi_2$ to $H^{\mathrm{odd}}(X;\bar{\bF}_2)$ is linear in a Frobenius-twisted sense, meaning that $Q\Xi_2(\zeta \cdot) = \zeta^2 Q\Xi_2(\cdot)$. From our previous computations, it then follows that
\begin{equation} \label{eq:idem}
Q\Xi_2(x \otimes u_j) = x \otimes u_j.
\end{equation}
\end{example}

Let's see how this might look from a categorical viewpoint, along the lines of Remark \ref{th:new-remark}. The Fukaya category of $\bC P^2$ over $\bar{\bF}_2$ is more properly described as a collection of three categories, each of which is semisimple, and which correspond to the idempotent summands $\bar{\bF}_2 u_j$ in quantum cohomology. Unfortunately, we cannot introduce a meaningful version of the Fukaya category of $T^2$ without Novikov parameters, because there are no nontrivial monotone Lagrangian submanifolds. However, the diagonal in the Fukaya category of $X \times X$ makes sense in a monotone context, so we can frame the discussion in that language. Algebraically, the diagonal splits as the direct sum of three objects, again corresponding to the $u_j$. We can take one of those summands and equip it with a flat line bundle, or formal family of flat line bundles, in the $T^2$-direction. The convolution of such Lagrangian correspondences gives us $H^1(T^2;\bG_m) = \bG_m^2$, respectively its formal completion. Since we can do that independently for all three summands, we see the formal completion of $\bG_m^6$ appearing, which is consistent with \eqref{eq:idem}. If one wanted to work over $\bF_2$ itself, only the idempotents $u_0$ and $u_1+u_2$ would be defined, giving rise to a more complicated picture of the Fukaya category.

\begin{example}
Along the same lines, take $X = T^2 \times \bC P^1 \times \bC P^1$, but now with $p = 3$. Take $x,y \in H^1(T^2;\bF_3)$ and $k,l \in H^2(\bC P^1 \times \bC P^1;\bF_3)$. The only classes $A$ which contribute to $Q\Xi_3(x \otimes k)$ are those which yield the two rulings of $X$, and hence satisfy $\langle [\mathit{point}] \rangle_{1,A} = 1$. For each such class, \eqref{eq:indecomposable-descendants} yields
\begin{equation}
\int_X (y \otimes l) Q\Xi_{3,A}(x \otimes k) = 
 - \langle (y \otimes l)(x \otimes k)\rangle_{1,A} = -\int_X (y \otimes l)(x \otimes k).
\end{equation}
Adding up their contributions yields
\begin{equation} \label{eq:xi-3-product}
Q\Xi_3 = -2 \, \mathit{id} = \mathit{id} \quad \text{on $H^3(X;\bF_3)$.}
\end{equation}
The same holds on $H^1(X;\bF_3)$, see \eqref{eq:deg1}. The corresponding question for $H^5(X;\bF_3)$ is just outside the reach of our methods, because there is a potential contribution from classes that are $3$ times that of a ruling.
%
%
%
%
\end{example}

This time, the Fukaya category of $\bC P^1 \times \bC P^1$ splits into four semisimple pieces over $\bF_3$, so one expects to see the product of four copies of the formal group associated to $T^2$, meaning a total of $\hat\bG_m^8$, which is compatible with our (partial) computation.
%

\subsection{Fano threefolds\label{subsec:3d}}
The remaining examples will be monotone symplectic six-manifolds, which have $H_1(X;\bZ) = 0$ and $H_*(X;\bZ)$ torsion-free (in fact, they will be algebraic, meaning Fano threefolds). We will assume that there is some $\lambda \in \bZ$ such that 
\begin{equation} \label{eq:lambda-x}
c_1(X) \ast x = \lambda x \text{ for $x \in H^3(X)$, or equivalently }
 \langle y, x \rangle_2 = \lambda \int_X yx \text{ for $x,y \in H^3(X)$.}
\end{equation}

From now on, we fix a prime $p$, and our notation will be that $x,y \in H^3(X;\bF_p)$. The classical Steenrod operations applied to $x$ have potentially nontrivial components in degrees $3$, $4$ and $6$. The degree $4$ component is the Bockstein $\beta$, which is zero because all our classes come from $H^3(X;\bZ)$. The degree $6$ component is the $t^{(3p-7)/2}\theta$-part of $\mathit{Sq}(x)$. For $p = 2$, this is just the cup square, which again is zero by lifting to $H^3(X;\bZ)$; and for $p>2$, it involves the Bockstein, see \eqref{eq:classical-steenrod}, hence is again zero. The outcome is that only the degree $3$ component survives. Taking the constants in \eqref{eq:classical-steenrod} into account (and omitting $X$ from the notation), this says that
\begin{equation} \label{eq:classical-steenrod-3}
\mathit{St}_p(x) = \begin{cases} x t^{3/2} & p = 2, \\
- \big( {\textstyle\frac{p-1}{2}!} \big) x t^{\frac{(3p-3)}{2}} & p>2.
\end{cases}
\end{equation}
As a consequence of \eqref{eq:lambda-x} (and the Divisor and TRR relations in Gromov-Witten theory), we have 
for $0 < d \leq p-2$,
\begin{equation} \label{eq:odd-psi}
\begin{aligned}
\langle y, \psi^d x \rangle_2 = 
\frac{\lambda}{d+1} \langle y, \psi^{d-1} x \rangle_2 = \cdots = \frac{\lambda^{d+1}}{(d+1)!} \int_X yx.
\end{aligned}
\end{equation}

Let's look at $\int_X y\, Q\Xi_p(x)$. For degree reasons, the curves that contribute to this lie in classes $A$ with $c_1(A) = p-1$, hence (setting $\Omega_X = c_1(X)$ in view of monotonicity), Corollary \ref{th:low-degree} applies. At first sight, the outcome reads as follows:
\begin{equation}
\begin{aligned}
& \int_X y Q\Xi_p(x) = -\sum_{\Omega \cdot A = p-1} \langle y, \psi^{p-2} x \rangle_{2,A}
- \int_X yx \,\sum_{\Omega \cdot A = p-1} \langle \psi^{p-3} [\mathit{point}] \rangle_{1,A} 
\\ & \qquad \qquad
+ \sum_{\substack{d_0 = \Omega \cdot A_0 > 0 \\ d_1 = \Omega \cdot A_1 > 0\\ d_0+d_1 = p-1}} \langle y,  (-1)^{d_0-1} \psi^{d_0-1} x \rangle_{2,A_0} 
\langle \psi^{d_1-2} [\mathit{point}] \rangle_{1,A_1}.
\end{aligned}
\end{equation}
Using \eqref{eq:odd-psi}, one simplifies this to
\begin{equation} \label{eq:threefold}
\begin{aligned}
\int_X y\, Q\Xi_p(x) = \big(\int_X xy \big) \big(-{\textstyle\frac{\lambda^{p-1}}{(p-1)!}} + \sum_{2 \leq d \leq p-1} (-1)^{d-1} {\textstyle\frac{\lambda^{p-1-d}}{(p-1-d)!}} \langle \psi^{d-2}[\mathit{point}] \rangle_1 \big).
\end{aligned}
\end{equation}
Besides $\lambda$, the enumerative ingredient that enters is the quantum period (see \cite{coates-corti-galkin-golyshev-kasprzyk14}, where the notation is $G_X$)
\begin{equation} 
\Pi = 1 + \sum_{d \geq 2} q^d \langle \psi^{d-2} [\mathit{point}] \rangle_{1,A},
\end{equation}
or more precisely, what's obtained from it by truncating mod $q^p$ and then considering the coefficients as lying in $\bF_p$. In that notation, one can also write \eqref{eq:threefold} as
\begin{equation} \label{eq:threefold-2}
Q\Xi_p = -\big(\text{\it $q^{p-1}$-coefficient of }e^{-\lambda q}\Pi \big)
\mathit{id}.
\end{equation}
All the examples that we will consider are instances of \cite[Theorem 4.7]{coates-corti-galkin-golyshev-kasprzyk14}, itself based on Givental's work.

\begin{example}
Let $X$ be the intersection of two quadrics in $\bC P^5$, which is also a moduli space of stable bundles (with rank two and fixed odd degree determinant) on a genus two curve. This has
\begin{equation}
H_l(X;\bZ) = \begin{cases} \bZ^4 & l = 3, \\ \bZ & l = 0,2,4,6, \\ 0 & \text{otherwise.} \end{cases}
\end{equation}
The first Chern class is twice a generator of $H^2(X;\bZ)$. For degree reasons, this implies that $Q\Xi_2 = 0$. This is not necessarily indicative of the general picture, since we already know that the prime $p = 2$ is exceptional \cite[p.\ 137]{donaldson93}: the small quantum cohomology ring has $\mathit{QH}^{\mathit{even}}(X) \iso \bZ[h]/h^2(h^2-16)$, hence does not split into summands if one reduces coefficients to $\bF_2$.

Let's look at odd primes. We have $\lambda = 0$ since there are no classes $A$ with $\Omega_X \cdot A = 1$. The quantum period is \cite[p.\ 135]{coates-corti-galkin-golyshev-kasprzyk16}
\begin{equation} 
\Pi = \sum_{d \geq 0}\frac{(2d)!^2}{(d!)^6} q^{2d} = 
1 + 2^2 q^2 + 3^2 q^4 + \textstyle\big(\frac{10}{3}\big)^2 q^6 + \big(\frac{35}{12}\big)^2 q^8 + \cdots
\end{equation}
Applying \eqref{eq:threefold} yields
\begin{equation} \label{eq:strange}
Q\Xi_p = - \frac{ (p-1)!^2}{(\frac{p-1}{2}!)^6} \mathit{id} = (-1)^{\frac{p-1}{2}} \mathit{id}
\quad \text{for odd $p$.}
\end{equation}
We should point out that the first nontrivial case $p = 3$, where the enumerative geometry is that of lines on $X$, is amenable to the more direct method of Section \ref{subsec:direct-localization}. The space of lines is regular \cite[Theorem 2.6]{reid72} (it is isomorphic to the Jacobian of the genus two curve associated to $X$ \cite[Theorem 4.7]{reid72}), and there are $4$ lines passing through a generic point \cite[p.\ 135]{donaldson92}. That information enables one to apply \eqref{eq:indecomposable-descendants} and obtain $Q\Xi_3 = -4\mathit{id} = -\mathit{id}$.
\end{example}

When thinking about the outcome of this computation, it may be useful to know that there is an algebraic group whose formal completion shows the same behaviour, namely
\begin{equation} \label{eq:rotation-group}
G = \big\{ x+iy \;:\; x^2 + y^2 = 1\big\},
\end{equation}
where $i$ is an abstract symbol such that $i^2 = -1$ (some readers may feel more comfortable writing this as a group of 2x2 rotation matrices). It is a non-split torus, which becomes isomorphic to $\bG_m$ over any coefficient ring that contains an actual root of $-1$. To write down the group law for the completion $\hat{G}$, one can use the rational parametrization $z = \frac{y}{1+x}$, in which it is given by
\begin{equation} \label{eq:t-parametrization}
z_1 \bullet z_2 = \frac{z_1+z_2}{1-z_1z_2}.
\end{equation}
The $p$-th power map, for primes $p>2$, is $(x+iy)^p \equiv x^p + (-1)^{\frac{p-1}{2}} iy^p\; \mathrm{mod}\; p$, or for \eqref{eq:t-parametrization},
\begin{equation}
\overbrace{z \bullet \cdots \bullet z}^p \equiv \frac{(-1)^{\frac{p-1}{2}} y^p}{1+x^p} = (-1)^{\frac{p-1}{2}} z^p = (-1)^{\frac{p-1}{2}} z \quad \text{for $z \in \bF_p$.}
\end{equation}
A natural conjecture would be that the formal group associated to $X$ is $\hat{G}^4$.
Note that in \cite{smith10}, a direct summand of the Fukaya category of $X$ was shown to be equivalent to the Fukaya category of the genus two curve. This seems to suggest a role for $\hat{\bG}_m^4$ rather than $\hat{G}^4$ (compare Remark \ref{th:new-remark}). However, \cite{smith10} works with complex number coefficients. To the author's best knowledge, we do not have a version of that argument that would work over $\bZ$ or $\bF_p$, and hence, it remains an open question to interpret the computation above in terms of mirror symmetry.


\begin{example}
Let $X \subset \bC P^4$ be a cubic threefold. This situation is parallel to the previous example, except that $H^3(X)$ is ten-dimensional. The quantum period is \cite[p.\ 134]{coates-corti-galkin-golyshev-kasprzyk16}
\begin{equation}
\Pi = \sum_d \frac{(3d)!}{(d!)^5} q^{2d} = \textstyle 1 + 6 q^2 + \frac{45}{2} q^4 + \frac{140}{3} q^6 + \frac{1925}{32} q^8 + \cdots
\end{equation}
One again has $Q\Xi_2 = 0$ and $\lambda = 0$ for degree reasons, but this time 
\begin{equation}
Q\Xi_p = - \frac{ (\frac{3p-3}{2})!}{(\frac{p-1}{2})!^5} \mathit{id} = 0 \quad \text{for $p>2$.}
\end{equation}
\end{example}

\begin{example}
The quartic threefold has \cite[p.\ 136]{coates-corti-galkin-golyshev-kasprzyk16}
\begin{equation}
\Pi = e^{-24 q} \sum_d \frac{(4d)!}{(d!)^5} q^d.
\end{equation}
We have $\lambda = -24$, the ``big eigenvalue'' in the terminology of \cite[Corollary 1.14]{sheridan16}. The $e^{-24 q}$ cancels out the corresponding term in \eqref{eq:threefold-2}, and as a result we again have $Q\Xi_p = 0$.
\end{example}

In both of the previous examples, homological mirror symmetry is known to hold \cite{sheridan16}, but again, arithmetic aspects are not addressed in that paper. Even assuming that the answer given there is true arithmetically, it takes the form of an orbifold LG (Landau-Ginzburg) model with a highly degenerate singular point. For our purpose, one would need to understand the formal completion near the identiy of the derived automorphism group of such an LG model, which is a purely algebro-geometric problem, but one whose answer is not known to this author. This means that we do not have a mirror symmetry interpretation for the vanishing of $Q\Xi_p$.

\begin{example} \label{th:blowup-example-2}
(Previously mentioned in Example \ref{th:blowup-example})
Let $X$ be a hypersurface of bidegree $(1,2)$ in $\bC P^1 \times \bC P^3$; equivalently, this is obtained by blowing up the intersection of two quadrics (which is an elliptic curve) in $\bC P^3$. It is a Fano threefold satisfying
\begin{equation}
H_l(X;\bZ) = \begin{cases} \bZ^2 & l = 2,3,4, \\ \bZ & l = 0,6, \\ 0 & \text{otherwise.} \end{cases}
\end{equation}
More precisely, we have $H_2(X;\bZ) \iso H_2(\bC^1 \times \bC^3;\bZ)$ by inclusion, and for the exceptional divisor $T^2 \times \bC P^1 \subset X$, we similarly have $H_3(T^2 \times \bC P^1) \iso H_3(X;\bZ)$. The classes potentially represented by holomorphic curves are
\begin{equation} \label{eq:d-classes}
A = (d_1,d_2), \quad \text{with $d_1,d_2 \geq 0$; and}\quad \Omega_X \cdot A = d_1 + 2d_2.
\end{equation}
Curves in the unique class $A = (1,0)$ with $\Omega_X \cdot A = (1,0)$ form the ruling of the exceptional divisor. From that, one easily sees that $\lambda = -1$. We have (\cite[p. 183]{coates-corti-galkin-golyshev-kasprzyk16})
\begin{equation}
e^{-\lambda q} \Pi = e^q \Pi = \sum_{d_1,d_2} \frac{(d_1+2d_2)!}{(d_1!)^2 (d_2!)^4} q^{d_1+2d_2},
\end{equation}
and hence
\begin{equation} \label{eq:combinatorics}
Q\Xi_p = \mathit{id}
\sum_{d_1+2d_2 = p-1} \frac{1}{(d_1!)^2 (d_2!)^4},
\end{equation}
see \eqref{eq:numbers} for the first few terms. As before, the lowest degree case $p = 2$ is amenable to more direct methods, and was in fact determined in \cite{wilkins18}.
\end{example}

One can write
\begin{equation} \label{eq:two-quadrics}
\begin{aligned}
& \sum_{d_1+2d_2 = m} \frac{m!^2}{(d_1!)^2 (d_2!)^4} = 
\sum_{d_1+2d_2 = m} \binom{m}{d_1\;d_2\; d_2}^2 =
\text{\it constant coefficient of } \tilde{W}^m, \\
& \qquad \qquad \text{where }
\tilde{W}(x_0,x_1,x_2,x_3) = \frac{(x_0^2+x_1^2+x_2x_3)(x_0x_1 + x_2^2 + x_3^2)}{x_0x_1x_2x_3}.
\end{aligned}
\end{equation}
This is an elementary combinatorial argument: when we expand $(x_0^2+x_1^2+x_2x_3)^m(x_0x_1+x_2^2+x_3^2)^m$, the monomial $(x_0x_1x_2x_3)^m$ arises by picking each $x_k^2$ term an equal number ($d_2$ in our formula) of times, which leads to the multinomial coefficients. Setting $m=p-1$ allows us to apply that to \eqref{eq:combinatorics}.
%
Next, consider the intersection of the two quadrics that appear in \eqref{eq:two-quadrics},
\begin{equation} \label{eq:two-quadrics-2}
C = \{ x_0^2 + x_1^2 + x_2x_3 = 0, \; x_0x_1 + x_2^2 + x_3^2 = 0\} \subset {\mathbb P}^3.
\end{equation}
There is an elementary number theory argument which allows one to count the number of points of $C(\bF_p)$ modulo $p$; it goes as follows. By little Fermat,
\begin{equation} \label{eq:0123}
\sum_{x_0,\dots,x_3} (1-x_0^2-x_1^2-x_2x_3)^{p-1}(1-x_0x_1-x_2^2-x_3^2)^{p-1} \in \bF_p
\end{equation}
counts the number of points in $\bF_p^4$ lying on the intersection of our quadrics. On the other hand, 
\begin{equation} \label{eq:0123-2}
\sum_{x_0,\dots,x_3} x_0^{i_0}\cdots x_3^{i_3} = (\sum_{x_0} x_0^{i_0})(\sum_{x_1} x_1^{i_1})(\sum_{x_2} x_2^{i_2})(\sum_{x_3} x_3^{i_3}) = 
\begin{cases} 
1 & (i_0,\dots,i_3) = (p-1,\dots,p-1), \\
0 & \text{all other } 0 \leq i_0,\dots,i_3 \leq p-1.
\end{cases}
\end{equation}
If we expand \eqref{eq:0123} and apply \eqref{eq:0123-2} to the resulting terms, the outcome is that \eqref{eq:0123} is the $x_0^{p-1}x_1^{p-1}x_2^{p-1}x_3^{p-1}$-coefficient of $(x_0^2-x_1^2-x_2x_3)^{p-1}(x_0x_1-x_2^2-x_3^2)^{p-1}$, which is what appears in \eqref{eq:two-quadrics}. Adjusting that to the point-count in projective space, we get
\begin{equation}
1-\# C(\bF_p) \equiv \text{\it constant coefficient of } \tilde{W}^m \text{ mod $p$}.
\end{equation}

The isogeny class of the elliptic curve $\mathit{Jac}(C)$ is listed as \cite[Isogeny class 15.a]{modular-forms-database}, and its associated modular form is \eqref{eq:modular-form}. Point-counting becomes relevant for us through a theorem of Honda \cite{honda68, olson76, hazewinkel86}, which says that $1 - \# C(\bF_p)$ mod $p$ can be identified with the $p$-th coefficient of the $p$-th power map for the formal group which is the completion of $\mathit{Jac}(C)$ (this coefficient is sometimes called the Hasse invariant of the mod $p$ reduction of $C$; maybe more precisely, it is a special case of the Hasse-Witt matrix of an algebraic curve). The natural interpretation of this in terms of mirror symmetry is the following:

\begin{conjecture} \label{th:isogeny}
Take $X$ as in Examples \ref{th:blowup-example}, \ref{th:blowup-example-2}. The Fukaya category of $X$ contains a direct summand equivalent to the derived category of sheaves on a genus one curve, whose Jacobian is isogenous to that of $C$. (The remaining summands are expected to be semi-simple: they do not contribute to $H^{\mathit{odd}}(X)$ or to our formal group.)
\end{conjecture}

It turns out that this is compatible with predictions coming from ``classical'' enumerative mirror symmetry. There (see e.g.\ \cite[Definition 4.9]{coates-corti-galkin-golyshev-kasprzyk14}), a mirror superpotential $W \in \bZ[y_1^{\pm 1},y_2^{\pm 1},y_3^{\pm 1}]$ for $X$ needs to have the property that
\begin{equation} \label{eq:oscillating}
\Pi = \int_{|y_1| = |y_2| = |y_3| = 1} e^{qW} \frac{dy_1 \wedge dy_2 \wedge dy_3}{y_1y_2y_3} = \sum_{d=0}^{\infty}
\frac{\text{\it constant coefficient of } W^d}{d!} q^d.
\end{equation}
There can be infinitely many different superpotentials for the same $X$, related by certain birational changes of variables. Assuming that the anticanonical linear system for $X$ contains a smooth divisor, then the actual mirror of $X$, formed relative to that divisor, should come with a proper (fibres are compact) function that specializes to those superpotentials in different Zariski charts. 

Getting back to our example: the function $\tilde{W}(1,x_1,x_2,x_3)-1$ satisfies the property \eqref{eq:oscillating}, as a consequence of \eqref{eq:two-quadrics}, but fails another requirement for mirror superpotentials, that of having a reflexive Newton polyhedron. Instead, the precise relation is as follows. One of the superpotentials for our specific $X$, given in \cite[Polytope 198]{fanosearch}, is
\begin{equation}
W = y_1 + y_2 + y_3 + y_1^{-1} y_2 y_3 + y_1 y_3^{-1} + y_2^{-1} + y_1^{-1} + y_2^{-1}y_3^{-1}.
\end{equation}
One can then write
\begin{equation} \label{eq:v-w}
\tilde{W}(1,x_1,x_2,x_3) - 1 = W(x_1^{-1}x_2x_3^{-1}, x_1^{-1}x_2^{-1}x_3, x_1^{-1}x_2^2).
\end{equation}
The monomial coordinate change in \eqref{eq:v-w} is a $\bZ/4$-cover of the $(y_1,y_2,y_3)$-torus by the $(x_1,x_2,x_3)$-torus; such coordinate changes do not affect oscillating integrals as in \eqref{eq:oscillating}. It is clear from the definition \eqref{eq:two-quadrics} that the critical locus of $\tilde{W}(1,x_1,x_2,x_3)$ contains an affine part of $C$, lying in the fibre $\tilde{W}^{-1}(0)$. Hence, the critical locus of $W$ contains an affine part of a $\bZ/4$-quotient of $C$, lying in the fibre $W^{-1}(-1)$. Moreover, the Hessian in transverse direction to those critical loci is nondegenerate over $\bQ$ (or over $\bF_p$, provided that $p$ is large). In view of the expected correspondence between the Fukaya category of $X$ and the category $D^b_{\mathit{sing}}$ associated to a compactification of $W$, this provides strong support for Conjecture \ref{th:isogeny}, and also gives a specific candidate genus one curve (within the specified isogeny class).

\section{Sign conventions\label{sec:signs}}
Signs are important for some of our example computations. This section clarifies the conventions used for $\bZ/p$-equivariant (and therefore $\mathit{Sym}_p$-equivariant) cohomology, and for the Steenrod operations.

\subsection{Equivariant cohomology\label{subsec:sign-of-t}}
Take the standard classifying space $BS^1 = S^\infty/S^1 = \bC P^\infty$. Let $t \in H^2_{S^1}(\mathit{point}) = H^2(\bC P^\infty)$ be the Chern class of ${\mathcal O}(-1)$. Given a representation $V$ of $S^1$, form the associated vector bundle
\begin{equation} \label{eq:associated}
(V \times S^\infty)/S^1 \longrightarrow \bC P^\infty, \quad \text{where }  g \cdot (v,z) = (gv,g^{-1}z).
\end{equation}
In this way, the representation $V_k$ of weight $k$ corresponds to the line bundle ${\mathcal O}(-k)$. We use the same convention as in \eqref{eq:associated} when forming the Borel construction (the equivariant cohomology of a space), and similarly for equivariant Euler classes of vector bundles.

Given a representation $V_{k_1} \oplus \cdots \oplus V_{k_d}$, its equivariant Euler class, defined as the Euler class of the associated bundle \eqref{eq:associated}, is therefore
\begin{equation} \label{eq:equivariant-euler-class}
e_{S^1}(V) = k_1 \cdots k_d \, t^d.
\end{equation}
We embed $\bZ/p \subset S^1$ in the obvious way, and take the mod $p$ reduction of $t$ to be the generator of $H^2_{\bZ/p}(\mathit{point};\bF_p)$, leading to a corresponding version of \eqref{eq:equivariant-euler-class}. We take $\theta \in H^1_{\bZ/p}(\mathit{point};\bF_p)$ to be the tautological generator, meaning the one associated to the identity map $\bZ/p = \pi_1(B(\bZ/p)) \rightarrow \bF_p$. Then, the Bockstein satisfies
\begin{equation} \label{eq:theta-bockstein}
\beta(\theta) = -t.
\end{equation}

\begin{example} \label{th:cyclic-representation}
Fix some odd $p$. Take the fundamental representation of $\bZ/p$ on $\bR^p$, by cyclic permutations, and let $V$ be its quotient by the trivial subspace $\bR (1,\dots, 1)$. Our orientation convention is that taking first $(1,\dots,1)$, and then after that lifts of an oriented basis of $V$, yields an oriented basis of $\bR^p$. If we temporarily ignore orientations, then clearly
\begin{equation} \label{eq:discrete-fourier}
V \iso V_1 \oplus V_2 \oplus \cdots V_{\frac{p-1}{2}}.
\end{equation}
This decomposition can be made explicit in terms of a discrete Fourier basis. A computation of the determinant of that basis (compare e.g. \cite{mcclellan-parks72}) shows that \eqref{eq:discrete-fourier} is in fact orientation-preserving. As a consequence,
\begin{equation} \label{eq:cyclic-euler}
e_{\bZ/p}(V) =  \textstyle (\frac{p-1}{2}!)\, t^{\frac{p-1}{2}}.
\end{equation}
\end{example}

\begin{example}
To check the sign in \eqref{eq:theta-bockstein}, let's replace the infinite-dimensional space $K(\bZ/p,1)$ by the lens space $L(p,1) = S^3/(\bZ/p)$, with $\bZ/p$ acting diagonally on $S^3 = \{|z_1|^2 + |z_2|^2 = 1\} \subset \bC^2$. The relation between the homological Bockstein $b$ and its cohomological counterpart $\beta$ is that
\begin{equation} \label{eq:two-bocksteins}
\langle \beta(y), x \rangle + (-1)^{|y|} \langle y, b(x) \rangle = 0.
\end{equation}
Consider $\{|z_1| \leq 1, \; z_2 = \sqrt{1-|z_1|^2}\} \subset L(p,1)$, with the complex orientation from $z_1$. This is a $\bZ/p$-cycle, whose homology class we write as $x$. Applying the homological Bockstein yields a $1/p$ fraction of the boundary, which is exactly the circle $\{z_2 = 0\} \subset L(p,1)$, with its orientation given by going around $z_1$ anticlockwise from $1$ to $e^{2\pi i/p}$. This means that by definition of $\theta$,
\begin{equation} \label{eq:homology-bockstein}
\langle \theta, b(x) \rangle = 1.
\end{equation}
The class $-t$ is Poincar\'e dual to the zero-locus of a section of the pullback of ${\mathcal O}(1)$, hence represented by the cycle $\{z_1 = 0\}$, with the usual orientation of the $z_2$ circle. The intersection number of that and the mod $p$ cycle defined above is 
\begin{equation} \label{eq:minus-t}
\langle -t, x \rangle = 1.
\end{equation}
From \eqref{eq:two-bocksteins}, for $y = \theta$, and \eqref{eq:homology-bockstein} we get $\langle \beta(\theta),x \rangle = 1$, which together with \eqref{eq:minus-t} yields the desired \eqref{eq:theta-bockstein}.
\end{example}

\subsection{Steenrod operations}
The appearance of combinatorial constants similar to those in \eqref{eq:classical-steenrod} goes back to the classical literature (see e.g.\  \cite[p.~107 and p.~112]{steenrod-epstein}). The point of introducing those is to make sure the operations satisfy the Steenrod axioms. Since a comparison between different definitions is made more complicated by sign conventions for equivariant cohomology, we want to explain one way of checking the choices made here.

Fix an odd prime $p$. Consider the Steenrod axiom which says that $P^0(x) = x$ for $x \in H^*(X;\bF_p)$. With our convention \eqref{eq:classical-steenrod}, this is equivalent to
\begin{equation} \label{eq:st-p-is-the-identity}
\mathit{St}_p(x) = 
(-1)^\ast {\textstyle \big( \frac{p-1}{2}! \big)}^{|x|}
 t^{\frac{p-1}{2} |x|} x + 
\text{\it terms of higher degree in $H^*(X;\bF_p)$,}
\end{equation}
with $\ast$ as in \eqref{eq:the-sign}. Suppose that $X$ is an oriented closed manifold, and that we apply this to $x = [\mathit{point}] \in H^{\mathrm{dim}(X)}(X;\bF_p)$. By definition, $\mathit{St}_p(x)$ is obtained from
\begin{equation}
H^{\mathrm{dim}(X)}(X;\bF_p) \xrightarrow{\text{\it $p$-th power}} H^{p\,\mathrm{dim}(X)}_{\bZ/p}(X^p;\bF_p)
\xrightarrow{\text{ \it restriction to diagonal}} H^{p\,\mathrm{dim}(X)}_{\bZ/p}(X;\bF_p).
\end{equation}
Hence, it maps $x$ to itself times the equivariant Euler class of the normal bundle to the diagonal $X \subset X^p$, restricted to a point. If $X$ is one-dimensional, that normal bundle is given by the representation $V$ from Example \ref{th:cyclic-representation}. In general, it can be identified with $\mathrm{dim}(X)$ copies of $V$, up to a Koszul reordering sign $(-1)^\dag$,
\begin{equation}
\dag = \textstyle \frac{|x|(|x|-1)}{2} \frac{p(p-1)}{2} \equiv \frac{|x|(|x|-1)}{2} \frac{p-1}{2} \;\; \text{mod } 2.
\end{equation}
Combining that with the $|x|$-th power of \eqref{eq:cyclic-euler} precisely yields the constant factor in \eqref{eq:st-p-is-the-identity}.
%

\end{document}